        \def\version{20 June, 2025}		                   %
\documentclass[a4paper,10pt,reqno]{amsart}

\usepackage{geometry}
\usepackage{etoolbox}
\makeatletter
\patchcmd{\section}{\scshape}{\bigskip\sc\LARGE}{}{}
\makeatother

\usepackage{amsfonts,amssymb,amsmath,amsthm,latexsym,epsfig,amscd,bbm,stmaryrd, mathrsfs}

\newtheorem{theorem}{Theorem}[section] 
\newtheorem{lemma}[theorem]{Lemma} 
\newtheorem{prop}[theorem] {Proposition} 
\newtheorem{proposition}[theorem] {Proposition}

\theoremstyle{definition}
\newtheorem{definition}[theorem] {Definition}

\newtheorem{remark}[theorem]  {Remark}

\newtheorem*{nobs}{\!\!}

\usepackage{enumitem}
\setenumerate{label={\normalfont(\alph*)},topsep=4pt,itemsep=0pt} 

\usepackage[numbers,round]{natbib}
 
\makeatletter\@addtoreset{equation}{section}\makeatother

\usepackage{float}

\usepackage{hyperref} 
\hypersetup{
    colorlinks=true,
    linkcolor=blue,
    filecolor=magenta,      
    urlcolor=cyan,
}

\usepackage{ulem} 

\renewcommand{\subsection}{\secdef \subsct\sbsect}
\newcommand{\subsct}[2][default]{\refstepcounter{subsection}\vspace{0.15cm}{\flushleft\bf \arabic{section}.\arabic{subsection}~\bf #1  }\nopagebreak\nopagebreak}

\newcommand{\sbsect}[1]{\vspace{0.1cm}\noindent{\bf #1}\vspace{0.1cm}}

\newcommand\ncite[2][]{%
\textnormal{\cite[#1]{#2}}
}

\newcommand{\e}{{\rm e}}
\newcommand{\eps}{\varepsilon}

\renewcommand{\phi}{\varphi}

\newcommand{\ssup}[1] {{\scriptscriptstyle{({#1}})}}

\newcommand{\Exp}{{\rm Exp}}
\newcommand{\Poi}{{\rm Poi}}

\newcommand{\Prob}{{\bf P}}
\newcommand{\Expec}{{\bf E}}
\newcommand{\Pol}{\P} 

\newcommand{\dd}{\,\mathrm{d}}
\newcommand{\ee}{\mathrm{e}}

\newcommand{\R}{\mathbb R}
\newcommand{\Z}{\mathbb Z}
\newcommand{\Q}{\mathbb Q}
\newcommand{\N}{\mathbb N}

\newcommand{\E}{\mathbb E}
\renewcommand{\P}{\mathbb P}

\newcommand{\bPi}{{\boldsymbol{\Pi}}}
\newcommand{\bH}{\boldsymbol{H}}
\newcommand{\bW}{\boldsymbol{W}}
\newcommand{\bZ}{\boldsymbol{Z}}
\newcommand{\br}{\boldsymbol{r}}

\newcommand{\bgamma}{\boldsymbol{\gamma}}
\newcommand{\bcL}{\boldsymbol{\cL}}
\newcommand{\bcA}{\boldsymbol{\cA}}

\newcommand{\btau}{\boldsymbol{\tau}}
\newcommand{\bsigma}{\boldsymbol{\sigma}}

\newcommand{\radon}[2]{\frac{\dd #1}{\dd #2}}

\newcommand{\slope}{s}
\newcommand{\height}{h}

\renewcommand{\epsilon}{\varepsilon}
\renewcommand{\P}{\mathbb{P}}
\newcommand{\1}{\mathbbm{1}}

\newcommand\forqq[1]{,\qquad #1}

\newcommand{\cA}{\mathcal{A}}
\newcommand{\cB}{\mathcal{B}}
\newcommand{\cC}{\mathcal{C}}
\newcommand{\cD}{\mathcal{D}}
\newcommand{\cE}{\mathcal{E}}
\newcommand{\cF}{\mathcal{F}}

\newcommand{\cH}{\mathcal{H}}

\newcommand{\cL}{\mathcal{L}}
\newcommand{\cM}{\mathcal{M}}
\newcommand{\cN}{\mathcal{N}}

\newcommand{\cP}{\mathcal{P}}
\newcommand{\cQ}{\mathcal{Q}}

\newcommand{\cW}{\mathcal{W}}

\newcommand{\fa}{\mathfrak{a}}
\newcommand{\fb}{\mathfrak{b}}  
\newcommand{\fc}{\mathfrak{c}}   
\newcommand{\fd}{\mathfrak{d}}

\newcommand{\fE}{\mathfrak{E}}
\newcommand{\fF}{\mathfrak{F}}

\newcommand{\fQ}{\mathfrak{Q}}

\newcommand{\fS}{\mathfrak{S}}

\newcommand{\sumn}{\sum_{n\in \N}}

\newcommand{\sumk}{\sum_{k\in \N}}

\newcommand{\limn}{\lim_{n\rightarrow \infty}}

\newcommand{\limN}{\lim_{N\rightarrow \infty}}
\newcommand{\limL}{\lim_{L\rightarrow \infty}}

\newcommand{\liminfn}{\liminf_{n\rightarrow \infty}}

\newcommand{\limsupk}{\limsup_{k\rightarrow \infty}}

\newcommand{\limsupn}{\limsup_{n\rightarrow \infty}}

\newcommand{\arrown}{\xrightarrow{n\rightarrow \infty}}

\newcommand{\arrowL}{\xrightarrow{L\rightarrow \infty}}
\newcommand{\arrowgamman}{\xrightarrow{\Gamma,n\rightarrow \infty}}

\newcommand{\bigcupn}{\bigcup_{n\in\N}}

\newcommand{\supk}{\sup_{k\in \N}}

\DeclareMathOperator*{\argmax}{arg\,max}
\DeclareMathOperator{\supp}{supp}

\newcommand{\rp}{\mathrm{p}}
\newcommand{\rb}{\mathrm{b}}

\newcommand{\bred}{\bgroup\color{red}}
\newcommand{\ered}{\egroup}

\usepackage{xcolor}

\newenvironment{colorwil}
    {\color[rgb]{0,0.5,0}    }
       
\newcommand{\bw}{\bgroup \color[rgb]{0,0.5,0} }
\newcommand{\ew}{\egroup}    

\newcommand{\bb}{\bgroup \color{blue} }
\newcommand{\eb}{\egroup}

\newcommand{\bwil}{\begin{colorwil}}
\newcommand{\ewil}{\end{colorwil}}

\usepackage[draft]{todonotes}

\definecolor{calccolor}{rgb}{0.6,0.3,0}

\newenvironment{calc}    {\color{calccolor}     }    {     }
\newcommand{\cnewline}{\\}
\newcommand{\cand}{&}

\newcommand{\bcalc}{\begin{calc}}
\newcommand{\ecalc}{\end{calc}}

\usepackage{environ}

\RenewEnviron{calc}{}  
\renewcommand{\cnewline}{}  
\renewcommand{\cand}{}  

\newcommand{\UU}{U}

\begin{document}
\title[Weakly self-avoiding walk in a Pareto-distributed random potential]
{\Large Weakly self-avoiding walk\\\medskip in a Pareto-distributed random potential}

\author[Wolfgang K\"onig, Nicolas P\'etr\'elis, Renato Soares dos Santos,  Willem van Zuijlen]{Wolfgang K\"onig \\ Nicolas P\'etr\'elis \\ Renato Soares dos Santos \\  Willem van Zuijlen}

\renewcommand{\thefootnote}{}
\footnote{\textit{AMS Subject Classification:} Primary 60H25, 60G70.
Secondary 82C44, 60F10, 60G55, 60G57.}
\footnote{\textit{Keywords: } random walk in random potential, random variational problem, parabolic Anderson model, path localisation, intermittent islands, weakly self-avoiding walk, Poisson point process convergence, spatial extreme-value analysis, Pareto potential distribution. }
\footnote{\textit{Acknowledgements:}
RSdS was partially supported by CNPq grants 313921/2020-2, 406001/2021-9 and FAPEMIG grant APQ-02288-21.}
\renewcommand{\thefootnote}{1}

\maketitle

\centerline{\small
\version}

\begin{abstract}
We investigate a model of continuous-time simple random walk paths in $\mathbb{Z}^d$ undergoing two competing interactions: an attractive one towards the large values of a random potential, and a self-repellent one in the spirit of the well-known weakly self-avoiding random walk. We take the potential to be i.i.d.~Pareto-distributed with parameter $\alpha>d$, and we tune the strength of the interactions in such a way that they both contribute on the same scale as $t\to\infty$.

Our main results are (1) the identification of the logarithmic asymptotics of the partition function of the model in terms of a random variational formula, and, (2) the identification of the path behaviour that gives the overwhelming contribution to the partition function for $\alpha>2d$: 
the random-walk path follows an optimal trajectory
that visits each of a finite number of random lattice sites for a positive random fraction of time. We prove a law of large numbers for this behaviour, i.e., that all other path behaviours give strictly less contribution to the partition function.
The joint distribution of the variational problem and of the optimal path can be expressed in terms of a limiting Poisson point process arising by a rescaling of the random potential. The latter convergence is in distribution
and is in the spirit of a standard extreme-value setting for a rescaling of an i.i.d. potential in large boxes, like in \cite{KLMS09}.
\end{abstract}

\setcounter{tocdepth}{1} 
\tableofcontents

\section{Introduction and main results}

\noindent In the last decades, there was a significant interest in the study of random motions that are attracted to certain regions defined by a surrounding  random medium. The most-studied type of models is called a {\em  random motion in a random potential}, which appears in the study of the parabolic Anderson model (PAM). The methods have been refined and extended  significantly in recent years, and a number of specific models  have been successfully treated in detail. The present paper makes a contribution to this line of research by studying a model that combines attraction with repulsion and shows, as a result, a much more pronounced behaviour.

We explain the model and the main purpose in Section~\ref{sec:model_purpose}. 
The crucial rescaling that we take explained in Section~\ref{sec:rescaling}, in which we also introduce our main objects of interest.  The key variational formula and the main results are presented in Section~\ref{varformainres}. In Section~\ref{sec:heuristics} we provide some heuristic explanations  for our results. The remainder of the paper is described in Section~\ref{sec:Orga}. Remarks on the literature are given in Section \ref{sec:Lit}.

\subsection{The model and main purpose}\label{sec:model_purpose}

\noindent Let $d\in \N$ and $\xi = (\xi(z))_{z\in\Z^d}$ be a random potential with distribution $\Prob$ that consists of i.i.d.\ random variables. 
Let $\P$ be the law of a continuous-time simple random walk $X=(X_s)_{s\geq 0}$ on the lattice $\Z^d$ with generator the discrete Laplacian $\Delta$ starting from the origin.  We take into account two types of microscopic interactions. The random walk interacts with the random field $\xi$ and 
undergoes a self-repulsion of strength $\beta$. 
This leads us to associate with every trajectory $X$ the Hamiltonian
\begin{equation} \label{Hamilton}
H_{t}^{\ssup{\xi,\beta}}(X)=\int_{0}^{t} \xi (X_s)\,\dd s -\beta \int_0^t \int_0^t \1_{\{X_s=X_u\}} \  \dd s\,  \dd u,
\end{equation} 
where $\beta\in[0,\infty)$ is the intensity of the self-repulsion. The first term is the interaction with the random potential $\xi$, the second is the {\em self-intersection local time (SILT)}, the amount of time pairs at which the random walk is at the same site. We introduce a polymer measure $\Pol_{t}^{\ssup{\xi,\beta}}$ that is absolutely continuous with respect to $\P$ (more precisely, to its restriction to paths on $[0,t]$) with Radon-Nikodym derivative  given by 
\begin{equation}\label{defpol}
\frac{\text{d} \Pol_{t}^{\ssup{\xi,\beta}}}{\text{d} \P} (X)= \frac{\ee^{H_{t}^{\ssup{\xi,\beta}}(X)}}{Z_{t}^{\ssup{\xi,\beta}}},
\end{equation}
where the normalizing constant $Z_{t}^{\ssup{\xi,\beta}}=\E[\ee^{H_{t}^{\ssup{\xi,\beta}}}]$ is the partition function of the model. We call this model the {\em weakly self-avoiding random walk in a random potential}. We want to study its large-$t$ behaviour. 

When $\beta =0$, 
the Feynman--Kac formula shows that $Z_{t}^{\ssup{\xi,0}}$ 
equals the total mass $U(t)=\sum_{x\in\Z^d}u(t,x)$  of the solution $u$ to the \underline{parabolic Anderson model} (PAM), the heat equation with random potential $\xi$:
$$
\partial_t u(t,x)=\Delta u(t,x)+\xi(x) u(t,x),\qquad x\in\Z^d,
$$
with localised initial condition $u(0,0)=1$ and $u(0,x)=0$ for $x\in \Z^d \setminus \{0\}$. 
On the other side, with $\xi=0$ and $\beta>0$,  $\Pol_t^{\ssup{0,\beta}}$ is the law of a weakly self-avoiding walk in continuous time. Since the SILT is not an additive functional, there is no obvious connection between this model and any partial differential equation.

It is clear that the Hamiltonian is a function of the walker's local times $\ell_t$, given by 
\begin{align}
\label{eqn:local_times}
\ell_t(z)= \ell_t^{\ssup{X}}(z) = \int_0^t \1 \{X_s = z\} \,\dd s \forqq{z\in \Z^d}. 
\end{align}
Indeed, 
\begin{equation}\label{Hamiltonlocaltimes}
H_{t}^{\ssup{\xi,\beta}}(X)=\sum_{z\in\Z^d}\xi(z)\ell_t(z)-\beta \sum_{z\in\Z^d}\ell_t(z)^2=\langle \xi,\ell_t\rangle-\beta \|\ell_t\|_2^2,
\end{equation}
where we wrote $\langle\cdot,\cdot\rangle$ for the standard inner product on $\Z^d$ and $\|\cdot\|_2$ for the standard $\ell^2$-norm on $\Z^d$. 

In earlier work on the PAM, it turned out that the model is the easier to analyse and the resulting picture is 
 more pronounced for heavy-tailed potentials.
Here we will assume that the potential variables~$\xi(z)$ at all sites $z\in\Z^d$ are independently \emph{Pareto-distributed}
with parameter $\alpha>d$, i.e., have the distribution function
\begin{equation}\label{defPareto}
F(r)=\Prob[\, \xi(z)\le r \,]=1-r^{-\alpha},\qquad r\ge 1.
\end{equation}
In particular, we have $\xi(z)\geq 1$ for all $z\in\Z^d$, almost surely. This is the most heavy-tailed distribution for which the PAM is well defined; indeed, \cite[Theorem~2.1]{GM90} says that $\alpha>d$ is necessary and sufficient for the partition function for $\beta=0$ to be finite. 
Hence,  by positivity of the self-interaction  our model is well-defined for any $\beta\geq 0$.

In \cite{KLMS09}, it turned out that the typical behaviour of the random walk
in the polymer measure for $\beta=0$ is to rush quickly to one of the peak points of the potential and to spend the remainder of the time in it,
and the highest peak sites form a rescaled Poisson point process in the spirit
of spatial extreme-value theory.
Now we add a self-repellent force and show that the picture is much more pronounced.
Indeed, the typical path in our polymer model visits not only one
of these peak sites, but several of them after each other, spending a
specified amount of time in them each.
We will describe this behaviour in terms of a random variational problem,
defined on a Poisson point process that we introduce below.

\subsection{Rescaling and point measures}\label{sec:rescaling}

\noindent It is the purpose of this paper to study the counterplay between the effects coming from the two terms in the Hamiltonian and the underlying probability distribution of the walk. 
To make sure that these three effects (i.e., attraction by the potential, self-repulsion and entropy -- 
see the heuristics in Section~\ref{sec:heuristics}) 
all run on the same scale, we take $\beta$ depending on $t$ as follows. 
Fix a parameter $\theta \in (0,\infty)$ and set 
\begin{equation}\label{betadef}
\beta_t:=\theta\,  \frac{t^{q-1}}{(\log t)^q}, \qquad\mbox{where}\quad q=\frac{d}{\alpha-d}.
\end{equation}
Note that $q$ and the large-$t$ behaviour of $\beta_t$ are decreasing in $\alpha$; for $\alpha > 2d$, we  have that $\beta_t\to 0$ as $t\to\infty$. 
To reduce the amount of parameters, in the following we write 
\begin{equation}\label{modeldefinition}
H_t^{\ssup{\xi}} : = H_t^{\ssup{\xi,\beta_t}}, 
\qquad 
Z_t^{\ssup{\xi}} : = Z_t^{\ssup{\xi,\beta_t}}, 
\qquad 
\Pol_t^{\ssup{\xi}} := \Pol_t^{\ssup{\xi,\beta_t}}. 
\end{equation}

We denote by  
\begin{equation}\label{defr}
r_t:=\Big(\frac{t}{\log t}\Big)^{1+q}
\end{equation} 
an important characteristic spatial length scale. More precisely, $r_t$ will turn out to be the typical distance of the relevant islands from the origin at which we  will find $(X_s)_{s\in[0,t]}$ with high probability under $\Pol_{t}^{\ssup{\xi}}$.
 Furthermore, it will turn out that the largest potential values in boxes of radius $\approx r_t$ are of order $r_t^{d/\alpha}$. It is convenient to express statements like these in terms of point processes, i.e., random variables with values in the set $\cM_\rp((0,\infty) \times \R^d)$ of Radon measures on $(0,\infty) \times \R^d$ with values in $\N_0 \cup \{\infty\}$, also called point measures since they are of the form $\sum_{n\in\N} \delta_{x_n}$ with $x_n \in (0,\infty) \times \R^d$. 
The crucial point is that the rescaled point process
\begin{equation}\label{Pitdef}
\Pi_t = \sum_{z\in \Z^d} \delta_{\big(\frac{\xi(z)}{r_t^{d/\alpha}}, \frac{z}{r_t} \big)}
\end{equation}
converges as $t\to\infty$ , weakly with respect to the vague topology in $\cM_\rp((0,\infty) \times \R^d)$, towards a Poisson point process (PPP) $\Pi$ with intensity measure 
$\alpha f^{-1-\alpha}  \dd f \otimes \dd y$:
\begin{equation}\label{PPPdef}
\Pi\sim {\rm PPP}\big((0,\infty)\times\R^d,\alpha f^{-1-\alpha}  \dd f \otimes \dd y\big).
\end{equation}
This is a basic result from spatial extreme-value analysis; see Lemma~\ref{l:P(t)toPi} for the precise statement. We will often write the points in $(0,\infty)\times\R^d$ as $(f,y)$. The process $\Pi$ may also be seen as a standard Poisson point process in $\R^d$ with Fr\'echet-distributed i.i.d.\ marks. We will write the probability with respect to $\Pi$ also by $\Prob$. 

In order to formulate the path behaviour of the walk in terms of the local times $\ell_t(x)=\ell^{\ssup{X}}_t(x)=\int_0^t \1 \{X(s) = x\} \,{\rm d }s$, 
we need to rescale them in time by $t$ and in space by $r_t$. 
Those rescaled local times are considered as a density with respect to $\Pi_t$ of 
the measure $W_t^{\ssup{\xi,X}}$ defined by
\begin{equation}\label{Wtdensity}
\frac{\dd W^{\ssup{\xi,X}}_t}{\dd \Pi_t}(f,y)=\frac{\ell_t(yr_t)}{t},\qquad (f,y)\in (0,\infty)\times \R^d,
\end{equation}
where we have extended the local times to a function $\ell_t\colon\R^d\to [0,t]$ satisfying $\ell_t=0$ on $\R^d\setminus \Z^d$. 
 Note that 
\begin{align*}
W^{\ssup{\xi,X}}_t = \sum_{z\in \Z^d} \frac{\ell_t(z)}{t} \delta_{\big(\frac{\xi(z)}{r_t^{d/\alpha}}, \frac{z}{r_t} \big)}
\end{align*}  
We will often omit the superscripts and write simply $W_t$ for $W_t^{\ssup{\xi,X}}$. 
It does not encode the number nor the order of the visits of the random walk to the sites. $W_t$ lies in the set $\cW$ of all measures $\mu$ on $(0,\infty)\times\R^d$ with total mass $\mu((0, \infty)\times\R^d)\leq 1$. 
Our main object of study will be $W_t$. Certainly the rescaled local times are an object of high interest themselves, but their behaviour may be deduced from that of $\Pi_t$ and $W_t$. 

By using the identities
\begin{equation}\label{r_tpropertis}
r_t^{d/\alpha}t= r_t \log t \quad (\tfrac{d}{\alpha} = 1- \tfrac{1}{1+q})
\qquad\mbox{and}\qquad 
\beta_t t^2=\theta r_t\log t,
\end{equation}
from \eqref{Hamiltonlocaltimes} it is easily seen that 
\begin{equation}\label{Hamil_Pi_t}
\begin{aligned}
 H_t^{\ssup{\xi}}(X)
&= \sum_{z\in\Z^d}\xi(z)\ell_t(z)-\beta_t \sum_{z\in\Z^d}\ell_t(z)^2
= \sum_{z\in \frac{\Z^d}{r_t} } t \xi(z r_t) \frac{\ell_t(z r_t)}{t}-\beta_t t^2 \sum_{z\in \frac{\Z^d}{r_t} } \Big(\frac{\ell_t(z r_t)}{t}\Big)^2\\
& = r_t \log t \int_{(0,\infty)\times\R^d}\big[f w(f,y)-\theta w(f,y)^2\big]\,\dd \Pi_t(f,y), \qquad\mbox{with } w = \frac{\dd W_t}{\dd \Pi_t},
\end{aligned}
\end{equation}
i.e., the Hamiltonian is an explicit functional of the rescaled local times and the point process $\Pi_t$.  
This is the starting point of our analysis.

\subsection{Main results: convergence towards a variational formula}\label{varformainres}

\noindent  In this section we formulate and comment on the main results of our paper. 
 In Theorem~\ref{thm:variational_formula1}~\ref{item:partition_convergence} we prove the asymptotic behaviour of the partition function $Z_t^{\ssup{\xi}}$ defined in \eqref{defpol} in terms of a limiting variational formula $\Xi$, which we define in Definition~\ref{def-functionals} in terms of an energy and entropy functional. 
In Theorem~\ref{thm:variational_formula1}~\ref{item:2d_part} we show that for $\alpha>2d$,
 the rescaled local times measure $W_t$ of a typical trajectory sampled from the mixture of $\Prob$ and $\Pol_t^{\ssup{\xi}}$ converges in distribution to the maximizer of the variational formula. 

Contrary to the parabolic Anderson model, which corresponds to $\theta=0$, this maximizer is -- with probability larger than $0$ -- not a Dirac measure. However, which may be quite unexpected, the support of this maximizer is still finite. More precisely, the number of points in the support is a random variable which attains any value of $\N$ with positive probability. 
 Interestingly, the behaviour of this variational formula changes for $\alpha \in (d,2d)$, we comment on this in Remark~\ref{rem:alpha_in_d_2d}.  

 We introduce a few definitions of functionals and notation as a preparation for the statements of our main result, Theorem~\ref{thm:variational_formula1}.  

\begin{definition}\label{def-functionals}
Let $\cP \in \cM_\rp((0,\infty) \times \R^d)$.  
We define 
\begin{enumerate}
\item the {\em energy functional} $\Phi_\cP \colon \cW \rightarrow [-\infty,\infty]$ by
\begin{align}
\label{eqn:def_Phi}
\Phi_\cP(\mu) 
:= \begin{cases}
\int_{(0,\infty) \times \R^d}\big[ f w(f,y)- \theta w(f,y)^2\big]\, \dd \cP (f,y) &\mbox{if } \mu \ll \cP, w  = \radon{\mu}{\cP}, \\
-\infty &\mbox{otherwise,} 
\end{cases}
\end{align}

\item 
\label{item:entropy_functional}
the {\em entropy functional} $\cD_\cP \colon \cW \rightarrow [0,\infty]$ by 
\begin{align}\label{defdist}
\cD_\cP(\mu) 
:= \begin{cases}
\sup_{Y \subset \supp_{\R^d} \mu , \# Y <\infty } D_0( Y) &\mbox{if } \mu \ll \cP, \\
\infty &\mbox{otherwise,} 
\end{cases}
\end{align}
where $D_0(\emptyset) = 0$ and $D_0(Y)$ for a finite nonempty set $Y\subset \R^d$ is the smallest possible $|\cdot|$-length of a path from the origin that reaches all points in $Y$; i.e., the minimum over $\sum_{i=1}^{N} |\sigma_i - \sigma_{i-1}|$ of bijections $\sigma: \{0,\dots, N\} \rightarrow Y\cup \{0\}$ with $\sigma_0 =0$, 
 where $N= \# (Y \setminus \{0\})$. 
We wrote 
\begin{align}
\label{eqn:support_in_R_d}
\supp_{\R^d} \mu= \{ y \in \R^d \colon \exists f>0, (f,y) \in \supp \mu \},
\end{align} 
for the support of the projection of $\mu$ on $\R^d$. 

\item the \emph{functional} $\Psi_\cP \colon \cW \rightarrow [-\infty,\infty)$ by 
\begin{align}\label{Psidef}
\Psi_\cP(\mu) 
:= 
\begin{cases}
\Phi_\cP(\mu) - q \cD_\cP(\mu) & 
\mbox{if }\Phi_\cP(\mu) <\infty,\\
 - \infty & 
  \mbox{otherwise.}
\end{cases}
\end{align}
\end{enumerate} 
For $\cP \in  \cM_\rp((0,\infty) \times \R^d) $ we define
\begin{equation}\label{Xidef}
\Xi(\cP) =\sup_{\mu \in \cW} \Psi_{\cP}(\mu). 
\end{equation}
Because $\Psi_\cP(0)=0$, \eqref{Xidef} defines a function $\Xi : \cM_\rp((0,\infty) \times \R^d) \rightarrow [0,\infty]$. 

\hfill$\Diamond$
\end{definition}

We write $\Prob \rtimes \P_t^{\ssup{\xi}}$ for the mixture of the laws of $\Prob$ and $\P_t^{\ssup{\xi}}$,  i.e., 
\begin{align}
\label{eqn:mixture}
\Prob \rtimes \P_t^{\ssup{\xi}} [ \, A\times B \,]= \Expec\big[\1_A(\xi) \, \P_t^{\ssup{\xi}} [\, X\in B\, ] \big]. 
\end{align}

We equip $\cW$, the space of all measures on $(0,\infty) \times \R^d$ with total mass $\le 1$, with the vague topology. By \cite[Corollary 13.31]{Kl08} and \cite[Theorem 15.7.7]{Ka83},  $\cW$ is a compact Polish space, which will be convenient for the formulation of our results and proofs.
Therefore, by \cite[Corollary 13.30]{Kl08} (corollary of Prohorov’s theorem) and \cite[Theorem 15.7.7]{Ka83}, 
the set of probability measures on $\cW$ forms a compact Polish space. 


\begin{theorem}
\label{thm:variational_formula1} 
Fix $\theta\in(0,\infty)$ and $\alpha\in(d,\infty)$. 
We have the following convergences in distribution:  
\begin{enumerate}
\item 
\label{item:partition_convergence}
{\em partition function:}
 \begin{equation}\label{limitfreeener}
 \frac{1}{r_t \log t} \log Z_t^{\ssup{\xi}} 
\ \overset{t\to\infty}{\Longrightarrow} \ \Xi(\Pi),
\end{equation}
 and, 
\begin{align*}
\Prob\big[\, \Xi(\Pi) \in [0,\infty) \,\big]=1. 
\end{align*}

\item
\label{item:2d_part}
 {\em law of the rescaled local times:} 
 Let $\alpha \in (2d,\infty)$. Then the following hold:  
\begin{enumerate}[label=\normalfont{(\roman*)}] 

\item 
\label{i:mu_star_optimizes_Psi}
 There exists 
a random variable $\mu^*$ with values in $\cW$ that maximizes $\Psi_\Pi$, in the sense that 
$\Prob$-almost surely,
\begin{equation}\label{eqn:mu_star_max_xi}
 \Psi_{\Pi}(\mu^*)=\Xi(\Pi).
\end{equation}

 \item 
 \label{i:existence_n_props_of_max}
 This maximizer $\mu^*$ is $\Prob$-almost surely unique in the sense that,
$\Prob$-almost surely, $\Psi_\Pi(\nu) < \Psi_\Pi(\mu^*)$ for any $\nu \in \cW\setminus \{ \mu^*\}$. 
 Furthermore, 
\begin{itemize}
\item $\Prob$-almost surely, $\mu^*$ is a probability measure with finite support and $\mu^*\ll \Pi$. In particular, $\Prob[\, \mu^* = 0 \,] = \Prob [\, \# \supp \mu^* =0 \,] = 0$. 
\item For all $k\in \N$,  $\Prob [\, \# \supp \mu^* =k \,] >0$. 
\end{itemize}
Consequently, 
\begin{align}
\label{eqn:Xi_Pi_positive_and_finite}
\Prob \big[\, \Xi(\Pi) \in (0,\infty) \, \big] = 1. 
\end{align}

\item 
\label{i:law_W_convergence}
 Under the mixed measure $\Prob \rtimes \P_t^{\ssup{\xi}}$ as in \eqref{eqn:mixture}, 
 $W_t = W_t^{\ssup{\xi,X}}$ converges in distribution to $\mu^*$; 
 \begin{align}
\label{eqn:W_t_convergence}
 W_t \ \overset{t\to\infty}{\Longrightarrow} \ \mu^* \ \mbox{ in } \cW, 
 \end{align}
 more precisely, for all $g\in C_\rb(\cW)$, 
 \begin{align}
 \label{eqn:weak_convergence_W_t_expressed}
 \Expec\big[ \E_t^{\ssup{\xi}} [ g(W_t^{\ssup{\xi,X}}) ]\big] \rightarrow \Expec[ g(\mu^*)]. 
 \end{align}
\end{enumerate}
\end{enumerate}
\end{theorem}

The proof of Theorem \ref{thm:variational_formula1} is given in Section~\ref{prooofthvarfor} conditionally on crucial assertions for the lower bound part of the convergence in \ref{item:partition_convergence} (which are proved in Section~\ref{sec:LowBound}) 
and crucial assertions about the upper bound in \ref{item:partition_convergence} and for \ref{i:law_W_convergence} (which are proved in Section \ref{sec:UppBound}).

\begin{remark}[The appearance of the energy and entropy functionals]
\label{rem-Explanation}
We will later (in Section~\ref{sec:heuristics} on the heuristics)
see that $q\, \cD_{\Pi}$ plays the role of the large-deviation rate functional for the rescaled local times on the scale $r_t\log t$; hence we called it an entropy functional. 
Observe that, see for example  \eqref{Hamil_Pi_t}, 
\begin{align}
\label{eqn:Hamiltonian_in_terms_of_Phi}
H_t^{\ssup{\xi}}(X) = (r_t \log t) \, \Phi_{\Pi_t}(W_t^{\ssup{\xi,X}}), 
\end{align}
and, therewith, 
\begin{equation}
\label{Zrepr}
Z_{t}^{\ssup{\xi}}=\E\big[\ee^{(r_t \log t)\, \Phi_{\Pi_t}(W_t)}\big].
\end{equation}
This says that the random walker gains on the exponential scale $r_t \log t$ the potential reward that is given for $w= \frac{\dd W_t}{\dd \Pi_t}$ by $\int_{(0,\infty) \times \R^d} f w(f,y) \, \dd \cP (f,y)$ and it pays the self-repellence price that is given by the expression $\int_{(0,\infty) \times \R^d} \theta w(f,y)^2 \, \dd \cP (f,y)$. 
See Section \ref{sec:heuristics} for a more precise heuristic explanation.
\hfill$\Diamond$
\end{remark}

\begin{remark}[Interpretation of Theorem~\ref{thm:variational_formula1}~\ref{item:2d_part}]\label{remmesmaxstructure}
Since $\mu^*$ has finite support and is absolutely continuous with respect to 
$\Pi$, there exist (random) $k^*\in \N$ and $(f_1^*,y_1^*), \dots, (f_{k^*}^*,y_{k^*}^*) \in \supp(\Pi)$ and
$w^*_1, \dots, w^*_{k^*} \in (0,1]$ satisfying $\sum_{i=1}^{k^*} w^*_i =  1$ such that 
\begin{align*}
\mu^*:=\sum_{i=1}^{k^*} w^{*}_i \delta_{(f_i^{*},y^{*}_i)}.
\end{align*}
Hence, if we interpret the convergence in \eqref{eqn:W_t_convergence} as almost sure convergence, the typical path under $\Pol_{t}^{\ssup{\xi}}$ spends $\sim w^*_i t$ time units  in the site $\sim \lfloor y^{*}_i r_t\rfloor$ with value $\xi(\lfloor y^{*}_i r_t\rfloor)\sim f_i^{*} r_t^{d/\alpha}$ for any $i\in\{1,\dots,k^*\}$
 and elsewhere only $o(t)$ time units 
 
 (or does not even reach them).

%
%

The above is an informal interpretation. More formally, we obtain the following consequence from Theorem~\ref{thm:variational_formula1}~\ref{item:2d_part}.
Given $h\in C_c((0,\infty) \times \R^d)$, observe that the function $g: \cW \rightarrow \R$ defined by $g(\mu) = \int_{(0,\infty)\times \R^d} h \dd \mu$ is continuous and bounded on $\cW$, i.e., $g\in C_\rb(\cW)$. 
Because 
\begin{align*}
g(W_t^{\ssup{\xi,X}}) 
& = \int_{(0,\infty)\times\R^d} h(f,y) \frac{\ell_t(r_t y)}{t} \dd \Pi_t(f,y)
 = \sum_{z\in \Z^d} h \Big( \frac{\xi(z)}{r_t^{d/\alpha}}, \frac{z}{r_t} \Big) \frac{\ell_t(z)}{t}, \\
 g(\mu^*) &= \int_{(0,\infty) \times \R^d} h \dd \mu^* =  \sum_{i=1}^{k^*} w_i^* h(f_i^*,y_i^*),
\end{align*}
and since $W_t \Longrightarrow \mu^*$ in $\cW$, see \eqref{eqn:weak_convergence_W_t_expressed}, we obtain 
\begin{align*}
\Expec \bigg[ 
\sum_{z\in \Z^d} \frac{\E_t^{\ssup{\xi}}[\ell_t(z)]}{t} h \Big( \frac{\xi(z)}{r_t^{d/\alpha}}, \frac{z}{r_t} \Big) 
\bigg] 
\ \longrightarrow \ \Expec\bigg[ \sum_{i=1}^{k^*} w_i^* h(f_i^*,y_i^*) \bigg]. 
\end{align*}
\hfill$\Diamond$ \end{remark}

\begin{remark}
\label{remark:more_general_statement}
[Generalization of Theorem~\ref{thm:variational_formula1}~\ref{item:2d_part}~\ref{i:law_W_convergence}]
Actually, we are able to prove a slightly more general but more abstract convergence than in Theorem~\ref{thm:variational_formula1}~\ref{item:2d_part}~\ref{i:law_W_convergence}. Indeed, let us write $\cL^{\ssup{\xi}}_t$ for the law of $W_t$ under $\P_t^{\ssup{\xi}}$, so that for a Borel set $A \subset \cW$ we have 
\begin{align}
\label{eqn:law_of_W_under_P_t_xi_beta_t}
\cL^{\ssup{\xi}}_t (A)
= \Pol_t^{\ssup{\xi}} (\{X \colon W_t^{ \ssup{\xi,X}}\in A\})
= \E_t^{\ssup{\xi}} [ \1\{W_t^{\ssup{\xi,X}} \in A\}]
\begin{calc}
= \int \1\{W_t^{\ssup{\xi,X}} \in A\} \dd \P_t^{\ssup{\xi}}(X)
\end{calc}. 
\end{align}
Then for $\mu^*$ as in Theorem~\ref{thm:variational_formula1}, we can show the following convergence in distribution  with respect to the weak topology  (see Remark~\ref{remark:proof_more_general_conv}), 
 \begin{align}
 \label{eqn:law_converges_to_delta_mu_star}
 \cL^{\ssup{\xi}}_t \ \overset{t\to\infty}{\Longrightarrow} \ \delta_{\mu^*} . 
 \end{align}
More precisely, for all continuous bounded functionals $f$ defined on the set of probability measures on $\cW$, 
\begin{align}
\label{eqn:weak_convergence_cL_expressed}
\Expec \big[ f(\cL^{\ssup{\xi}}_t) \big] 
\ \overset{t\to\infty}{\Longrightarrow} \
\Expec \big[ f(\delta_{\mu^*}) \big]. 
\end{align}
The convergence in Theorem~\ref{thm:variational_formula1}~\ref{item:2d_part}~\ref{i:law_W_convergence} follows from this by taking $f(\nu) = \int g \dd \nu$ in \eqref{eqn:weak_convergence_cL_expressed} for $g\in C_\rb(\cW)$ so that $f(\cL_t^{\ssup{\xi}}) = \E_t^{\ssup{\xi}}[g(W_t^{\ssup{\xi,X}})]$ and $f(\delta_{\mu^*})= g(\mu^*)$, implying \eqref{eqn:weak_convergence_W_t_expressed}. 
\hfill$\Diamond$ 
\end{remark}

\begin{remark}[Large-deviations explanation]\label{rem-LDPExplanation}
Standard ideas from the theory of large deviations applied to the formula in \eqref{Zrepr} already suggest that the statements in Theorems~\ref{thm:variational_formula1} are true. Indeed, if $(W_t)_{t\in(0,\infty)}$ would satisfy a large-deviations principle (LDP) on the scale $r_t\log t$ with rate function $q\mathcal D_{\Pi}$, and if the limit $\Pi_t\Longrightarrow\Pi$ could be combined with this LDP, and if the energy functional $\Phi_\Pi$ would have appropriate continuity and boundedness properties, then Varadhan's lemma would imply the validity of our main statement. Roughly, this is also our strategy for proving the theorems, but a lot of technicalities need to be overcome along the way.
\hfill$\Diamond$
\end{remark}

\begin{remark}[Traveling Saleman Distance]
Observe that $D_0(Y)$ in the definition of $\cD_\cP$ (see \eqref{defdist}) represents the Travelling Salesman Distance of a path connecting $0$ to all the points of $Y$ but without returning to $0$. 
\end{remark}

 \begin{remark}[A particular case: $d=1$]\label{cased=1}
In dimension $1$, both the expressions of the energy functional $\Phi_\Pi$ in \eqref{eqn:def_Phi} and of the entropy functional $\cD_\Pi$ in \eqref{defdist} turn out to be much easier.  To be more specific, we consider 
$x,z\in [0,\infty)^2$ and $\mu\in \mathcal{W}$ such that $ \mu \ll \Pi$ and $\min \supp_{\R} \mu=-x$ and $\max \supp_{\R} \mu=z$. Then,  
$\mathcal D_\Pi(\mu)=(x+z)- \min\{x,z\}$ since it is the shortest distance that one has to travel so that starting from the origin both sites $-x$ and $z$ are visited (here we can of course say that the path first visits $x$ and then $z$ if $x<z$ and the other way around if $x>z$). 
Moreover, by screening effect in dimension $1$, that is since every
site in $[-x,z]$ is visited by a trajectory that
reaches both $-x$ and $z$, any  $(f,y)$ belonging to $\supp \Pi$ with $y \in  [-x,z]$ may be in the support of such $\mu\in \cW$ without increasing the entropy. 
Note that, $\Prob$-a.s. $\Pi((0,\infty)\times \{0\})=0$ and therefore, it is sufficient to consider $(x,z)$ that are not simultaneously null. 
We introduce the order statistics $(f^{\ssup{i}}_{[-x,z]})_{i\in \N}$ of the field inside $[-x,z]$ such that for a sequence $(y^{\ssup{i}}_{[-x,z]})_{i\in\N}$ 
\begin{align*}
\supp \Pi \cap \big[(0,\infty) \times [-x,z] \big] = \{ (f^{\ssup{i}}_{[-x,z]},y^{\ssup{i}}_{[-x,z]}) : i\in\N\}. 
\end{align*}
%

Because every $\mu \in \cW$ with $ \supp \mu \subset  (0,\infty) \times [-x,z]$ and $\mu \ll \cP$ is of the form 
$\mu = \sum_{i=1}^\infty w_i \delta_{(f^{\ssup{i}}_{[-x,z]},y^{\ssup{i}}_{[-x,z]})}$ for a sequence $(w_i)_{i\in\N}$ in $[0,\infty]$ with $\sum_{i=1}^\infty w_i \le 1$, we have the following identity by definition of $\Phi_\cP$ \eqref{eqn:def_Phi}, 
\begin{align}
\label{eqn:sup_Phi_over_interval_1d}
\sup_{\mu \in \cW \colon \supp \mu \subset  (0,\infty) \times [-x,z]} \Phi_\Pi(\mu) 
= \sup_{ w_1,w_2,\dots \ge 0, \sum_{i=1}^\infty w_i \le 1}  \sum_{i=1}^\infty \Big( f^{\ssup{i}}_{[-x,z]} w_i - \theta w_i^2 \Big). 
\end{align}
As the $f^{\ssup{i}}_{[-x,z]}$ are chosen decreasingly in $i$, it turns out that for large $i$, it is not worth taking $w_i$ positive, or in other words, that there exists a $k$ for which one may restrict to those sequences $(w_i)_{i\in\N}$ with $w_{k+1}=w_{k+2}=\cdots =0$. 
This is made precise in Proposition~\ref{p:AnaPhik}: 
This finite $k$ is given in terms of a formula of the the order statistics $(f^{\ssup{i}}_{[-x,z]})_{i\in\N}$, let us call it $k_{\star}$ here (see \eqref{e:defkstar}). 
By writing $\varphi_{k_\star} \big(f^{(1)}_{[-x,z]},\dots,f^{(k_\star)}_{[-x,z]}\big)$ for the right-hand side of \eqref{eqn:sup_Phi_over_interval_1d} with the restriction $w_i =0$ for $i \ge k_\star+1$ (which agrees with the definition of $\varphi_k$ given in \eqref{eqn:varphi_k_def}), 
 we then have the following description of $\Xi(\Pi)$ in terms of  the leftmost and rightmost points in $\supp \mu$:
\begin{equation}\label{xirepdef}
\Xi(\Pi):= \sup_{x,y\in [0,\infty)^2\setminus\{(0,0)\}} \varphi_{k_\star}\big(f^{(1)}_{[-x,z]},\dots,f^{(k_\star)}_{[-x,z]}\big)
-q(x+z)- q\min\{x,z\}.
\end{equation}

%
%

\hfill$\Diamond$ 
\end{remark}

\begin{remark}[Suggested scenario for $\alpha\in(d,2d)$]
\label{rem:alpha_in_d_2d}
In the course of our proof for Theorem \ref{thm:variational_formula1}
we in particular show that the characteristic variational formula $\Xi(\Pi)$ is finite almost surely for $\alpha>d$ and positive for $\alpha>2d$.
The latter assertion seems crucial for the behaviour of the path in the random potential.
It is not easy to give a short argument for that; apparently the PPP possesses sufficiently many sufficiently high
potential values with not too large distances between them, 
such that trajectories exist for which it is worth paying the travels in order to profit from spending time in those large potentials. 

This is different for $\alpha\in(d,2d)$. Indeed, in a forthcoming paper we will show that both $\{\Xi(\Pi)=0\}$
and $\{\Xi(\Pi)>0\}$ have positive probability here. This can be roughly explained as follows: 
With positive probability the PPP, like for $\alpha >2d$, possesses sufficiently many high potentials with not too large distances. 
Also, the complement has positive probability, leading to no such preferable locations as it is not worth travelling that far to profit from the large potential. 
We conjecture that the intermediate order statistics need to be considered to reflect the true behaviour of the random path. 
A closer description of this scenario will be given in a future work.
\hfill$\Diamond$ 
\end{remark}

%
%
%

\subsection{Heuristic explanation}\label{sec:heuristics}

\noindent Let us give here an explanation of the main result,  Theorem~\ref{thm:variational_formula1}. In Section \ref{sec:LowBound} we will turn the following heuristics into a proof of the lower bound (however, the proof of the upper bound in Section~\ref{sec:UppBound} is very different).

We need to understand the large-$t$ behaviour of the partition function $Z_t^{\ssup \xi}$ defined in \eqref{modeldefinition}, i.e., the expectation of $\ee^{H_t^{\ssup{\xi}}}$ with $\beta_t$ defined in \eqref{betadef}  and the Hamiltonian as in \eqref{Hamil_Pi_t}. The first step is to understand the scales on which the probability from the simple random walk and the contribution from the potential $\xi$ run, where we first ignore the self-intersection term and concentrate on the potential-interaction term. Hence, this part of the heuristics is the same as for the behaviour of the PAM with Pareto-distributed potential in \cite{KLMS09}; let us give an overview now. Note that there is a competition for the random walk between a reward  (called \lq energy\rq) from staying much time in sites with extremely large values of the potential  and the probabilistic cost (called \lq entropy\rq)  to reach such preferable sites quickly: travelling far  and fast, the walker finds a larger potential value  where it can stay longer, but this is more costly. We need to find an optimal balance.

As in \cite{KLMS09}, we obtain a lower bound by inserting the indicator on the event  $\cA_t^{z,s}$ that the walker wanders on some fixed shortest path to a site $z$ during the time interval $[0,st)$ and stays at $z$ during $[st,t]$. 
The probability of this event is  (recall that 
the jump rate of our random walk is $2d$)  
\begin{align*}
\P(\cA_t^{z,s}) = \Poi_{2dst}(|z|)  (2d)^{-|z|} \ee^{-(1-s)2dt} \#\{\mbox{ shortest paths } 0 \longleftrightarrow	z \}. 
\end{align*}

Taking $|z|\gg t$, using Stirling's estimation for the term $|z|!$ that appears in the $\Poi$-term, we see that the dominating terms in the exponent are $ |z| \log |z| $ and $|z| \log (st)$, so that, dropping all lower-order terms,
$$
\P(\cA_t^{z,s}) \approx \exp \left\{ |z| \left[ \log \frac{t}{|z|} + \log s \right] \right\}.
$$
Now let us examine the contribution from the potential $\xi$. In order to obtain a preferably large lower bound, we pick $z$ as a maximizing point of the potential $\xi$ within a box of radius $r$. According to the Pareto-tails, we are able to pick $z$ such that $\xi(z)\approx r^{d/\alpha}$, and this site will be
 approximately of the order $r$, 
 i.e., $|z|\approx r$. Hence, from the stay at $z$ during $[st,t]$, the potential contributes $\approx \e^{t(1-s) r^{d/\alpha}}$. The potential values that the random walk experiences on the fast rush during $[0,st]$ are negligible. Hence, we have the lower bound
$$
Z_t^{\ssup \xi}\geq\exp \left\{ r \left[ \log \frac{t}{r} + \log s \right] \right\} \e^{t r^{d/\alpha}} \e^{-st r^{d/\alpha}},
$$
and we have still the freedom to optimize over small $s$ and large $r$.
The optimal choice of $s\in[0,1]$ for the second and the last term is $s\approx\frac 1 t r^{1-d/\alpha}$, which implies the lower bound
\begin{equation}\label{probesti}
Z_t^{\ssup \xi}\geq\exp\Big\{r \log \frac tr +t r^{d/\alpha}-\log \big(t r^{d/\alpha -1}\big)\Big\}=\exp\Big\{tr^{d/\alpha}-\frac d\alpha r\log r\Big\}.
\end{equation}
The maximal $r$ satisfies $t r^{d/\alpha -1}=1+\log r$, and this is asymptotically satisfied by $r=r_t=(t/\log t)^{1+q}$ as in \eqref{defr} with $q=\frac d{\alpha-d}$ as in \eqref{betadef}. Then both the energy term $t r^{d/\alpha}$ and the entropy term $-\frac d\alpha r\log r\approx -q r_t\log t$ are on the scale $r_t \log t$. Interestingly, the latter comes exclusively from the probability of the crucial event $\cA^{z,s}_t$, after optimizing on $s\approx 1/\log t$, whose choice depends on the potential value. This explains the appearance of the prefactor $q$ in \eqref{Psidef} and the notion of an \lq entropy functional\rq\ in Definition \ref{def-functionals}.

So far, this was the first part of the explanation, which applied also to \cite{KLMS09}, since we considered only the potential interaction. Now let us become specific to our model, where an  additional self-intersection term in the Hamiltonian appears and makes the path paying an extra energy price when staying a long time in a single site. If this time is of order $t$, then the price is of order $\beta_t t^2=\theta r_t\log t$ (see \eqref{r_tpropertis}) i.e., it is on the same scale. Therefore, the strategy has to be improved by not only visiting one site, but several after each other and staying in each of them some time $\asymp t$. 
Standard assertions from spatial extreme-value theory guarantee that there are not only one, but many sites with potential values $\asymp r_t^{d/\alpha}$,  and they are homogeneously distributed over a centered ball with radius $\approx r_t$, so there are many good candidates for sites to be visited. One needs to make a choice of the number of the visited sites and the order in which they are visited during the time interval $[0,t]$.  The travel between them costs an additional price of the same order as the first travel from the origin to one of them since the distances of all these travels are on the same scale. The functional $\Phi_{\mathcal P}( \mu)$ in \eqref{eqn:def_Phi} describes the energetic gain (staying $\approx w(f,y)t$ time units in a site $\approx y r$ with potential value $f r_t^{d/\alpha}$ for all the $(f,y)$ in $\mathcal P$ and paying $\theta w(f,y)^2$ for the self-intersections), and the functional $q \mathcal D_{\mathcal P}$ in \eqref{defdist} describes the exponential probabilistic cost payed by the simpe random walk. Hence, the rate functional $\Psi_{\mathcal P}=\Phi_{\mathcal P}-q \mathcal D_{\mathcal P}$ in \eqref{Psidef} gives the entire exponential cost of this path strategy on the scale $r_t$ for $\mathcal P=\Pi_t$, as we explained in Remark~\ref{rem-Explanation}. Then the exponential behaviour of the partition function $Z_t^{\ssup \xi}$ is given by the maximum of $\mu\mapsto \Psi_{\Pi}(\mu)$, as in Varadhan's lemma. An additional technical difficulty is the combination of the large-deviation arguments with the point process convergence $\Pi_t\to \Pi$; see Remark~\ref{rem-Explanation}.

\subsection{Organization of the paper}\label{sec:Orga}

\noindent The remainder of the paper is organized as follows. In Section~\ref{sec:VP}, we explain our strategy for proving Theorem~\ref{thm:variational_formula1} 
and formulate two types of intermediate results:
the first comprises a deterministic version of Theorem~\ref{thm:variational_formula1} for point measures that possess certain properties, while the second states that $\Pi_t$ and $\Pi$ possess these properties. 
This version of Theorem~\ref{thm:variational_formula1} is proved in Section~\ref{sec:AnaVP} along with fundamental compactness and continuity properties of the energy functional $\Phi_\cP$ and the entropy functional $\cD_\cP$, as well as the existence and main properties of the maximizer. 
The proof that $\Pi_t$ and $\Pi$ have the good properties is given in Section~\ref{subsec:optimizers_with_finite_support}. 
Hence, Sections \ref{sec:VP}--\ref{subsec:optimizers_with_finite_support} derive all the properties of the variational formula for $\Xi(\Pi)$ as formulated in Theorem~\ref{thm:variational_formula1}. 
In Section~\ref{prooofthvarfor} we start the proof of the large-$t$ analysis of the model, Theorem~\ref{thm:variational_formula1}, by formulating two main ingredients for the proof for the lower respectively upper bound (Propositions~\ref{lowerboundrand} resp.~\ref{p:uppbound}). The two propositions are proved in Sections~\ref{sec:LowBound} and \ref{sec:UppBound}, respectively.

\subsection{Literature remarks}\label{sec:Lit}

\noindent Let us give some survey on the literature on random motions in random potential and localisation properties. First examples appeared in work by Sznitman on Brownian motion among Poissonian obstacles in the early 1990s, see his monograph \cite{Sz98}. Among many other things, he proved almost-sure attraction to one island, but did not identify this island. An analogous localization result (i.e., for the solution of the PAM rather than for the random motion) in the space-discrete setting with an i.i.d.\ doubly-exponentially distributed potential was \cite{GKM07}. Around 2010, it turned out that the strongest attraction to the intermittent islands is present for potentials with heavy tails, since they have a particularly pronounced profile: indeed, the  islands are just singletons here. This has been observed for the first time for the most heavy-tailed potential distribution, the {\em Pareto distribution}, in \cite{KLMS09} and has been investigated in great detail in \cite{MOS11} and also for the exponential distribution in \cite{LM12}; see the survey \cite{M11}. For double-exponentially distributed potential, localization (and much more) was proved in \cite{BKS18}. Most of these localization results are formulated and proved for the solution of the PAM rather than for the random walk in the Feynmna--Kac formula. See \cite[Sect.\ 6]{Ast16} and \cite[Sect.\ 6.3]{K16} for two comprehensive survey texts on such localization results up to 2016. 

These two survey texts triggered interest in localization of discrete-time random walks among Bernoulli traps, the (time and space) discrete version of Brownian motion among Poisson obstacles. Deep localization properties were derived in \cite{DFSX20a, DFSX20b, DFSX21} in this setting in dimension $d\geq 2$. Similar results for a correlated random potential in $d=1$ (with i.i.d.\ gaps between the obstacles) have been derived recently in \cite{PS22}. 

Earlier work \cite{OR16, OR17, OR18} analysed the strongly related model of a spatial random branching walk in a Pareto-distributed random field of branching rates. For this model, this series of papers derives a description that resembles our model and results quite strongly. It turns out there that the main bulk of the particles is highly concentrated in a number of sites that are defined in terms of a Poisson point process (essentially the same as our $\Pi$); more precisely, the branching process subsequently visits  points of this point process that are step for step extremal with respect to a compromise between high potential values and short distances. This precise mechanism is different from the one that is detected in the parabolic Anderson model (PAM) in \cite{KLMS09}; the main difference to that model being that the branching process is consistent and has no finite time horizon, like the PAM. With respect to our model, an additional difference is the repellent effect from the second part of our Hamiltonian.

The second feature in our model is the Hamiltonian of the famous weakly self-repellent random walk, the negative exponential of the self-intersection local time. It is here only a side-remark that the behaviour of the weakly self-repellent walk is poorly understood in dimensions $d\in\{2,3,4\}$, and it was a substantial challenge to investigate it in the other dimensions. See \cite{MS13, S11} for surveying texts. Generally, it is expected that the typical behaviour is a more or less uniformly spread-out behaviour in space on a scale $t^{\gamma_d}$ that is much larger than the scale $t^{1/2}$ of the free walk (at least in $d\leq 4$), but much less than the scale $t$ of a ballistic walk (at least in $d\geq 2$). However, all these effects will not be seen in our model, because of the presence of the random potential. We will necessarily be working on a much rougher scale than those scales that are believed to be responsible for this spread-out behaviour, and the resulting behaviour will be much more spread-out, but for reasons that have to do with the potential and not with the self-repulsion.

\subsection{Notation}\label{sec:notation}

\noindent We write $\N = \{1,2,\dots\}$ and $\N_0 = \{0,1,2,\dots,\}$. 
For the rest of the paper, we fix $d\in \N$, $\theta, \alpha \in (0,\infty)$. 
We set $Q_R:=[-R,R]^d$ for $R\in (0,\infty)$. 
We write $\Longrightarrow$ for convergence in distribution. 
We abbreviate \lq Poisson point process\rq\ by \lq PPP\rq.
For $y=(y_1,\dots,y_d) \in \R^d$ we write $|y| = \sum_{i=1}^d |y_i|$ for the $\ell^1$ norm  and for $a\in \R^d$, $r>0$ we write
$B(a,r) = \{ y \in \R^d : |y-a| < r\}$ the open $\ell^1$ ball in $\R^d$ around $a$ of radius $r$.

\section{Preparation}\label{sec:VP}

\noindent In the present section, we prepare for the proof of Theorem~\ref{thm:variational_formula1} by  analysing the variational formula $\Xi$ in \eqref{Xidef}. On the way, we need to extract several continuity and compactness properties of the energy and entropy functionals $\Phi_{\mathcal P}$ and $\mathcal D_{\mathcal P}$ as functions of $\mathcal P\in \mathcal M_\rp((0,\infty)\times \R^d)$. For this, we keep $\mathcal P$ deterministic in this section, but restrict to a subclass of such $\mathcal P$'s for which we can prove all needed assertions and for which we can prove that the processes $\Pi_t$ and $\Pi$ satisfy them. 
We define in particular a class of {\em good} point measures, see Definition~\ref{def:good_point_measure},
with the characteristic that if $\mathcal P$ is good then $\Psi_{\mathcal P}$ has at most one maximizer. In Theorem~\ref{thm:variational_formula_deterministic} we formulate all the necessary properties for deterministic good point measures, among other things the uniqueness of the maximizer and its continuous dependence on $\mathcal P$. Furthermore, in Lemma~\ref{l:goodness} we state that $\Pi_t$ and $\Pi$ are almost surely good.  The proofs are deferred to later sections.

It will be convenient for us to compactify   specific subsets  of $(0,\infty) \times \R^d$ as described next. For $\height, \slope>0$ we define the cone-shaped set (see also Figure~\ref{fig:cH}) with height $\height$ and slope $\slope$ by
\begin{equation}\label{e:defcH}
\cH^\slope_\height := \left\{ (f,y) \in (0,\infty) \times \R^d \colon\, f > \slope |y| + \height \right\}.
\end{equation}
\begin{figure}[H]
		\includegraphics[width=0.3\textwidth]{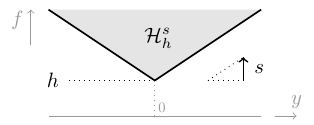}  
	\caption{Illustration of $\cH^s_h$.}
	\label{fig:cH}
\end{figure}

We can embed $(0,\infty) \times \R^d$ continuously and openly into a locally compact Polish space $\mathfrak{E}$ with certain properties, mentioned in the lemma below. 
For a locally compact metric space $E$, we write $\cM_\rp(E)$ for the set of point measures on $E$, i.e., $\N_0\cup \{\infty\}$-valued Radon measures, or equivalently, due to the fact that the support of each such measure is countable and locally finite, the set of Radon measures that can be written as $\sum_{n\in\N} \delta_{x_n}$ for a sequence $(x_n)_{n\in\N}$ in $E$. 
We equip $\cM_\rp(E)$ with the \emph{vague topology}, i.e., $\cP_n \rightarrow \cP$ in $\cM_\rp(E)$ if and only if $\int \phi\, \dd \cP_n \rightarrow \int \phi \,\dd \cP$ for each continuous compactly supported $\phi \colon E \rightarrow \R$.
When $E=\mathfrak{E}$ we will simply write $\cM_\rp = \cM_\rp(\mathfrak{E})$.
We denote by $\cM_\rp^\circ$ the set of point measures in $\cM_\rp$ that are supported in $(0,\infty) \times \R^d$ and equip it with the topology from $\cM_\rp$.

\begin{lemma}
\label{l:about_cM_p}
There exists a locally compact Polish space $\fE$, with $(0,\infty)\times \R^d \subset \fE$, such that 
\begin{enumerate}[label={\normalfont(\roman*)}] 
\item 
\label{item:compactness_in_fE_wrt_cH}
for every $\height,\slope >0$, the
open set $\cH^\slope_\height$ 
is relatively compact in $\fE$, and for every compact subset $K$ in $ \mathfrak{E}$ there exist $\height,\slope>0$ such that $K \cap [(0,\infty) \times \R^d] \subset \cH_\height^\slope$, 
\item 
\label{item:open_and_continuous_embed}
the map $\iota : (0,\infty) \times \R^d \rightarrow \fE$ given by $\iota((f,y))=(f,y)$, $(f,y) \in (0,\infty) \times \R^d$ is open and continuous. In other words, 
$(0,\infty) \times \R^d$ is continuously and openly embedded in $\fE$. 
\end{enumerate}
Moreover, 
\begin{enumerate}
\item 
\label{item:cM_circ_as_subset}
$\cM_\rp^\circ$ can be viewed as a subspace of $ \cM_\rp((0,\infty)\times \R^d)$, in the sense that for $\cP \in \cM_\rp^\circ$, $\cP \circ \iota$ defines a point measure on $(0,\infty) \times \R^d$. 
\item 
\label{item:extension_to_fE}
Let $\cP \in \cM_{\rp}((0,\infty) \times \R^d)$. 
Define $\overline \cP$ on $\fE$ by $\overline \cP(A) = \cP( \iota^{-1}(A))$ for Borel sets $A\subset \fE$. 
Then $\overline \cP$ is an element of $\cM_\rp^\circ$ if and only if $\cP(\cH_h^s)<\infty$ for all $s,h>0$. 
\end{enumerate}
\end{lemma}
\begin{proof}
The proof is given in Appendix~\ref{section:space_fE}, below Lemma~\ref{l:fE}. 
\end{proof}

\begin{remark}
\label{remark:vague_convergence}
Observe that by the Portmanteau theorem \cite[Theorem 13.16]{Kl08}, $\cP_n \rightarrow \cP$ in $\cM_\rp(\fE)$ implies that $\cP_n(A) \rightarrow \cP(A)$ for all measurable relatively compact $A \subset \fE$  with $\cP(\partial A)=0$. And hence by Lemma~\ref{l:about_cM_p}~\ref{item:extension_to_fE}, in particular for all measurable $A$  with $\cP(\partial A)=0$  that are a subset of $\cH_h^s$ for some $h,s>0$. 
\end{remark}

\begin{lemma}
\label{l:Pi_and_Pi_t_as_in_cM_p}
 For all $t>0$, $\Prob (\overline \Pi \in \cM_\rp^\circ)  =\Prob (\overline \Pi_t \in \cM_\rp^\circ) =1$ 
(with $\overline \cP$ as in Lemma~\ref{l:about_cM_p}~\ref{item:extension_to_fE}). 
\end{lemma}
\begin{proof}
Let $h,s >0$. 
We show that 
$\Expec( \Pi(\cH_h^s) )<\infty$ and $\Expec( \Pi_t(\cH_h^s) )<\infty$, so that, e.g., 
$\Prob( \Pi(\cH_h^s)< \infty) =1$, and therefore  $\Prob (\bigcap_{s,h\in (0,\infty) \cap \Q} \{\Pi(\cH_h^s)< \infty\} ) =1$. 
Because $\cH_h^s \subset \cH_j^t$ for $t\le s$ and $j\le h$, this implies $\Prob (\bigcap_{s,h\in (0,\infty)} \{\Pi(\cH_h^s)< \infty\} ) =1$ and thus, by Lemma~\ref{l:about_cM_p}~\ref{item:extension_to_fE} that $\overline \Pi \in \cM_\rp$. 

We have 
\begin{align*}
\Pi_t (\cH_h^s) 
= \sum_{z\in \Z^d} \delta_{\big(\frac{\xi(z)}{r_t^{d/\alpha}}, \frac{z}{r_t} \big)} (\cH_h^s) 
= \sum_{z\in \Z^d} \1 
\Big\{ \frac{\xi(z)}{r_t^{d/\alpha}}> s \Big|\frac{z}{r_t}\Big| +h \Big\}. 
\end{align*}
We calculate 
\begin{align*}
\Prob \Big( 
\frac{\xi(z)}{r_t^{d/\alpha}}> s \Big|\frac{z}{r_t}\Big| +h
\Big)
= \Big(r_t^{d/\alpha} \big(s \Big|\frac{z}{r_t}\Big| +h\big)\Big)^{-\alpha} 
= r_t^{-d} \big(s \Big|\frac{z}{r_t}\Big| +h\big)^{-\alpha}. 
\end{align*}
Therefore, because $\alpha >d$, 
\begin{align*}
\Expec \Big(\Pi_t (\cH_h^s)  \Big)
\le \sum_{z\in \Z^d} r_t^{-d} \big(s \Big|\frac{z}{r_t}\Big| +h\big)^{-\alpha} <\infty. 
\end{align*}
Note that $\Pi(\cH_\height^s)$ is a Poisson distributed random variable with parameter
\begin{equation*}
\int_{\R^d} \int_0^\infty \1_{\cH_\height^s}(f,y)  \frac{\alpha}{f^{\alpha+1}} \dd f \dd y = 
\int_{\R^d} \frac{1}{(s|y| + \height)^\alpha} \dd y,
\end{equation*}
which is finite for $\alpha>d$, so that $\E(\Pi(\cH_h^s)) < \infty$. 
\end{proof}

From here on, we will make abuse of notation and write $\Pi$ also for $\overline \Pi$ and $\Pi_t$ for $\overline \Pi_t$.

In the following lemma we state the convergence of $\Pi_t$ towards $\Pi$, as mentioned between \eqref{Pitdef} and \eqref{PPPdef}: 

\begin{lemma}[$\Pi_t\Longrightarrow\Pi$]\label{l:P(t)toPi}
Let $\alpha\in(d,\infty)$. 
Let $t_1,t_2,\dots \in (0,\infty)$ and $t_n \rightarrow \infty$. 
We may view $\Pi_{t_n}$ and $\Pi$ as elements of $\cM_\rp^\circ$ for all $n$. 
Then $\Pi_{t_n} \rightarrow \Pi$ in $\cM_{\rp}^\circ$ as $n\rightarrow \infty$. 
\end{lemma}
\begin{proof}
That we may view $\Pi_{t_n}$ and $\Pi$ as elements of $\cM_\rp^\circ$ follows by Lemma~\ref{l:Pi_and_Pi_t_as_in_cM_p}. 

The convergence follows by \cite[Lemma~7.4]{BKS18} (the fact that we have $(0,\infty) \times \R^d$ instead of $\R \times \R^d$ does not change the validity of the lemma, as the proof builds on \cite[Proposition~3.21]{Re87} can be carried out in our situation in the same way). For this we have to check the two conditions, namely (7.17) and (7.18) of that lemma (we take the $\widehat N_t$ in that lemma to be equal to zero, furthermore let us mention that in (7.17) there should be ``$\frac{t^d}{(2\widehat N_t +1)^d}$'' instead of ``$\frac{t^d}{(2\widehat N_t)^d}$''). The first condition, (7.17), follows by 
\begin{align*}
\lim_{r \rightarrow \infty} r^d \Prob \Big( \frac{\xi(0)}{r^{d/\alpha}} >s \Big) 
= \lim_{r \rightarrow \infty} r^d (r^{d/\alpha} s)^{-\alpha} = s^{-\alpha}. 
\end{align*}
The second condition, (7.18), follows by the fact that for all $\slope,\height>0$ 
\begin{align*}
&
\sum_{x\in \Z^d : |x|> rn} \Prob \Big( \frac{\xi(0)}{r^{d/\alpha}} > \slope \frac{|x|}{r} + \height \Big) 
 \le 
\sum_{x\in \Z^d : |x|> rn} \Prob \Big( \xi(0) > \slope |x| r^{\frac{d}{\alpha} - 1 }  \Big) \\
& \le \slope^{-\alpha} 
\sum_{x\in \Z^d : |x|> rn} |x|^{-\alpha} r^{\alpha-d} 
 \le \slope^{-\alpha} 
\int_{\frac{rn}{2} }^\infty u^{-\alpha} r^{\alpha-d} u^{d-1} \dd u 
\cnewline 
\cand \begin{calc}
= \slope^{-\alpha} r^{\alpha - d} \frac{u^{d-\alpha}}{d-\alpha} \Big|_{\frac{rn}{2}}^\infty 
= \slope^{-\alpha} r^{\alpha - d} \frac{(\frac{rn}{2})^{d-\alpha}}{\alpha -d} 
\end{calc} 
= \slope^{-\alpha} \frac{2^{\alpha - d}}{\alpha -d} n^{d-\alpha} \xrightarrow{ n \rightarrow \infty} 0. 
\end{align*}
\begin{calc}
So that indeed, 
\begin{align*}
\limn \limsup_{r\rightarrow \infty} 
\sum_{x\in \Z^d : |x|> rn} \Prob \Big( \frac{\xi(0)}{r^{d/\alpha}} > \slope \frac{|x|}{r} + \height \Big) =0. 
\end{align*}
\end{calc}
\end{proof}

For $\cP \in \cM^\circ_\rp$ and $R>0$, define 
\label{e:defMr}
\begin{equation}
M_R(\cP) := \sup \left\{ f \colon\, (f,y) \in \cP \text{ and } y \in Q_R \right\}.
\end{equation}

\begin{lemma}
\label{l:basicgoodprop}
Let $\cP,\cP_1,\cP_2,\dots$ be in $ \cM_{\rp}^\circ$ such that $\cP_n \to \cP$ in $\cM_{\rp}$. 
Then
\begin{equation}
\label{e:condcompact}
\sup_{n \in \N} M_R(\cP_n) <\infty \;\; \mbox{ for all } R>0\qquad \text{ and } \qquad
\lim_{R \to \infty} \sup_{n \in \N} \frac{M_R(\cP_n)}{R}=0.
\end{equation}
In particular, $\lim_{R \to \infty} \frac{M_R(\cP)}{R}=0$.
\end{lemma}

\begin{proof}
Fix $\varepsilon \in (0,1)$. First observe that for any $\cQ \in \cM_\rp$ and any $\height >0$, 
$\cQ(\cH_\height^\varepsilon) < \infty$  since $\cH^\varepsilon_\height$ is relatively compact (in $\fE$),
and thus $V(\cQ) := \sup\{f \colon (f,y) \in \supp \cQ \cap \cH_\height^\epsilon \} < \infty$.
Fix $\height \in (0,\varepsilon)$ such that $\cP(\partial \cH_\height^\varepsilon) = 0$. 
By \cite[Proposition~3.13]{Re87}, there exists an $n_0 \in \N$ such that $V(\cP_n) \leq V(\cP)+1$ for all $n\geq n_0$,
implying $M = \sup_{n \in \N} V(\cP_n) < \infty$. 
For $R>0$, note that $(0,\infty) \times Q_R = A \cup B$ where $A \subset (0,\varepsilon (R+1)] \times \R^d$ and $B \subset \cH_\height^\epsilon$ (see Figure~\ref{fig:cone_and_strip}),
so that 
\begin{figure}[H]
		\includegraphics[width=0.4\textwidth]{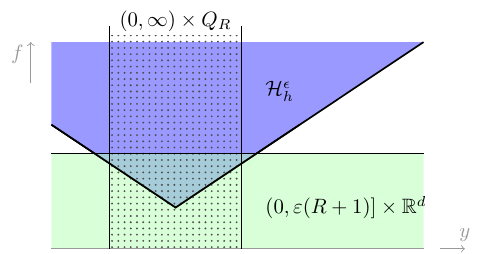}  
	\caption{Illustration $(0,\infty) \times Q_R \subset \cH_h^\epsilon \cup (0,\varepsilon (R+1)] \times \R^d$.}
	\label{fig:cone_and_strip}
\end{figure}

\begin{align*}
\sup_{n \in \N} M_R(\cP_n) \leq \max\{\varepsilon (R+1), M\} < \infty,
\end{align*}
implying the first statement in \eqref{e:condcompact}.
For the second statement, divide the above inequality by $R$, take the $\limsup$ as $R \to \infty$ and then 
the limit as $\varepsilon \to 0$.
\end{proof}

Recall that $\cW$ denotes the set of subprobability measures on $(0,\infty)\times\R^d$ with total mass $\leq 1$.
 For  $R>0$ and $\cP\in \cM_\rp((0,\infty) \times \R^d)$ define the sets 
\begin{align}
\label{e:def_fF}
\fF(\cP) &:= \left\{ \mu \in \cW\colon\, \mu \ll \cP \ \text{and}\  \supp(\mu) \ \text{is finite} \right\}, \\
\label{e:def_fF_1}
\fF_1(\cP) & := \{  \mu \in \fF(\cP) : \mu \mbox{ is a probability measure}\}. 
\end{align}
In the following lemma we show that if a point measure $\cP$ has sufficiently many points, then one may restrict to take the supremum over elements in $\fF_1(\cP)$ in the variational formula for $\Xi(\cP)$ \eqref{Xidef}. 
Then, we introduce the notion of a good point measure, under which we will prove a conditional version of Theorem \ref{thm:variational_formula1}~\ref{item:2d_part}~\ref{i:existence_n_props_of_max}.

\begin{lemma}
\label{l:Xi_as_sup_over_probm_fF_1}
Let $\cP \in \cM_\rp^\circ$. Suppose that 
\begin{align}
\label{eqn:many_points}
\forall \delta> 0 \ 
\exists m\in \N \ 
\exists \mbox{ distinct } (f_1,y_1),\dots (f_m,y_m) \in \supp \cP, \quad 
\sum_{i=1}^m f_i \ge 2\theta, \ D_0(y_1,\dots,y_m) <\delta. 
\end{align}
Then 
\begin{align}
\label{eqn:supremum_over_finite_supports_and_prob_m}
\Xi(\cP) 
\begin{calc} = \sup_{\mu \in \cW} \Psi_\cP(\mu)  \end{calc}
= \sup_{\mu \in \fF_1(\cP) } \Psi_\cP(\mu). 
\end{align}
\end{lemma}

 The proof of lemma is given at the end of Section~\ref{sec:reformulation}.

\begin{definition}[Good point measure]
\label{def:good_point_measure}
We say that a point measure $\cP\in\cM_\rp^\circ$ is \emph{good} if 
$\Psi_\cP$ possesses at most one maximizer in $\fF(\cP)$ in the sense that 
there exists at most one $\nu \in \fF(\cP)$ such that $\sup_{\mu \in \fF(\cP)} \Psi_\cP(\mu) = \Psi_\cP(\nu)$, 
and if it satisfies at least one of the two following conditions: 
\begin{enumerate}[label={\normalfont(\roman*)}]
\item 
\label{item:good3_condition_new_points}
There exists a $\beta>2$ such that for all $R,C>0$ there exists a $\epsilon_{R,C} >0$ such that for $\epsilon\leq \epsilon_{R,C}$ and for all $y\in Q_{R}$  the set $[(C\epsilon, \infty) \times B(y,\epsilon^\beta)] \cap \supp \cP$ is nonempty. 

\item\label{itemgoodfinitelymanypoints}  $\cP((0,\infty)\times Q_R)<\infty$ for every $R>0$.
\end{enumerate}
\end{definition}

Now we can formulate a deterministic version of Theorem~\ref{thm:variational_formula1}~\ref{item:2d_part}~\ref{i:existence_n_props_of_max} (and more) for good processes.

\begin{theorem}[Analysis of $\Psi_\cP$ for good $\cP$]\label{thm:variational_formula_deterministic}
If $\cP,\cP_1,\cP_2,\dots\in \cM_\rp^\circ$ are good in the sense of Definition~\ref{def:good_point_measure}, then the following statements hold.
\begin{enumerate}
\item {\em Maximizer:} 
\label{item:var_f_maximizer}
There exists  a unique $\mu^*\in \cW$ such that
\begin{equation}\label{defxicP}
 \Psi_{\cP}(\mu^*)=\sup_{\mu \in \cW} \Psi_{\cP}(\mu).
\end{equation}
This maximizer $\mu^*$ has finite support, is a probability measure and satisfies  $\mu^* \ll \cP$, i.e., $\mu^*\in \fF_1(\cP)$. 
\item 
 {\em Multisupport maximizer:} 
\label{item:support_more_than_one}
Let $k\in \N$, $\epsilon := \frac{\theta}{4 q k^4}$ and $L> 2 \theta + (q+1)\varepsilon$. 
With 
 $B_\epsilon = \overline {B(0,\epsilon)}$, 
 define the regions of $(0,\infty) \times \R^d$ (see also Figure~\ref{fig:k_strategy})
\begin{align}
\label{eqn:regions_multisupport}
\begin{aligned}
G 
& = \left[ L, L+\frac{2\theta}{k} \right] \times B_{\varepsilon}, 
\quad E^1 = (\varepsilon, L) \times B_{\varepsilon}, 
\quad E^2 = \left(L + \frac{2\theta}{k}, \infty \right) \times B_{\varepsilon}, \\
E^3 & = \big\{(f,y) \in (0,\infty) \times \R^d \colon\, |y| > \varepsilon, f > {\varepsilon} \vee ( |y| - 3 \theta) \big\}.
\end{aligned}
\end{align}
If $\cP$ satisfies 
\begin{align}
\label{eqn:spreading_maxis}
\cP(G)=k \quad \text{ and } \quad \cP(E^1) = \cP(E^2) = \cP(E^3) =0,
\end{align}
then $\# \supp \mu^* = k$. 
\item {\em Stability:} 
\label{item:var_f_stability}
For any open neighbourhood $O\subset \cW$ of $\mu^*$,
\begin{equation}
\sup_{O^{\rm c}} \Psi_{\cP}< \sup_{\cW}\Psi_\cP=\Psi_{\cP}(\mu^*).
\end{equation}

\item {\em Continuity of maximizer:} 
\label{item:var_f_continuity}
If $\cP_n\to \cP$ in $\cM_\rp^\circ$, then the maximizers $\mu^*_n$ of $\Psi_{\cP_n}$ converge towards $\mu^*$ as $n\to\infty$ in the vague topology.
\end{enumerate}
\end{theorem}

\begin{figure}[H]
		\includegraphics[width=0.5\textwidth]{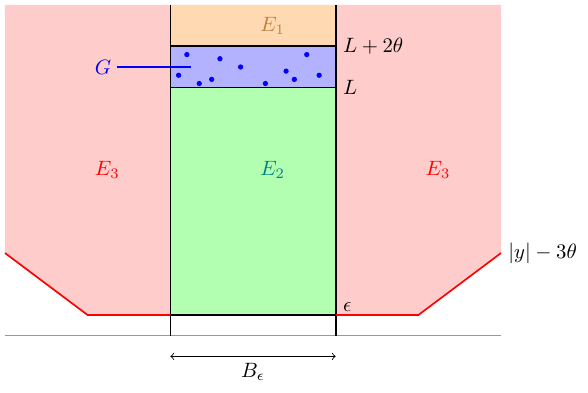}  
	\caption{Illustration of the regions $G$, $E^1,E^2$ and $E^3$ as in \eqref{eqn:regions_multisupport}.}
	\label{fig:k_strategy}
\end{figure}

 The proof of Theorem~\ref{thm:variational_formula_deterministic} is given in Section~\ref{exisuni}.

In order to be able to apply Theorem~\ref{thm:variational_formula_deterministic} to the point processes $\Pi_t$ defined in \eqref{Pitdef} and its limiting PPP $\Pi$ defined in \eqref{PPPdef}, we use the following lemmas, whose proofs are given in Section~\ref{subsec:optimizers_with_finite_support}. 

\begin{lemma}[Goodness of $\Pi$ and $\Pi_t$]\label{l:goodness} 
Fix $\alpha\in(2d,\infty)$. 
Then, for any $t\in (0,\infty)$, with probability one, $\Pi$ and $\Pi_t$ are good. 
\end{lemma}

\begin{lemma}
\label{l:Pi_spreads_maxis}
Let $\alpha \in (d,\infty)$. Let $k\in \N$ and $G, E^1,E^2,E^3$ be as in \eqref{eqn:regions_multisupport}. 
Then 
\begin{align*}
\Prob\big[ \, \Pi(G)=k, \Pi(E^1 \cup E^2 \cup E^3)=0 \, \big] >0. 
\end{align*}
\end{lemma}
\begin{proof}
Since the regions in \eqref{eqn:regions_multisupport} are disjoint and each of them has finite and positive intensity measure,
the random variables $\Pi(G)$, $\Pi(E^1)$, $\Pi(E^2)$, $\Pi(E^3)$ are independent and have non-trivial Poisson distributions, so that $\Prob[ \Pi(G)=k, \Pi(E^1)=\Pi(E^2)=\Pi(E^3)=0]$ has positive probability. 
\end{proof}

It is clear that Theorem~\ref{thm:variational_formula1}~\ref{item:2d_part}~\ref{i:existence_n_props_of_max} directly follows from Theorem~\ref{thm:variational_formula_deterministic}, combined with Lemma~\ref{l:goodness} and Lemma~\ref{l:Pi_spreads_maxis}.

For the proof of the lower bound in Section~\ref{sec:LowBound}, we use the following lemma so that we can apply Lemma~\ref{l:Xi_as_sup_over_probm_fF_1} to $\Pi$. 

\begin{lemma}
\label{l:Pi_satisfies_many_points}
Let $\alpha \in (d,\infty)$. With probability one, $\Pi$ satisfies \eqref{eqn:many_points}. 
\end{lemma}

\section{Analysis of the variational formula}\label{sec:AnaVP}

\noindent Here we give the proof of Theorem~\ref{thm:variational_formula_deterministic}; that is, we analyse the maximum of $\Psi_\cP$ and its maximizer for an arbitrary point measure $\cP$ that is good in the sense of Definition~\ref{def:good_point_measure}. 

Let us first give a short outline of the proof. 
In Section \ref{sec:AnaPhik} we analyse the maximization of the energy functional $ \Phi_\cP(\mu)$ over $\mu$ when the number of points of $\cP$ is fixed; this involves only the maximization over the potential values. 
In Section~\ref{sec:reformulation}, we introduce the crucial tool for handling variational problems, namely the Gamma-convergence, and derive $\Gamma$-continuity properties of $\cP\mapsto \Phi_\cP$ and $\cP\mapsto\cD_\cP$ and consider  the compactness of the objects appearing in the variational formula $\sup_{\mu\in \cW} \Psi_\cP(\mu)$ (the right-hand side of \eqref{defxicP}): if $\cP$ is good, then one can restrict the variational formula to measures in $\cW$ that have a compact support with respect to $\R^d$. 
Then we give the proof of Lemma~\ref{l:Xi_as_sup_over_probm_fF_1}. 
In Section~\ref{sec:SuppFinite} we show that any maximizer of $\Psi_\cP$  is necessarily of finite support. In Section \ref{exisuni} finally we prove Theorem~\ref{thm:variational_formula_deterministic}~\ref{item:var_f_maximizer}, putting together the results derived in the preceding sections, namely the $\Gamma$-continuity of $\cP \mapsto -\Psi_\cP$, and the fact that we need to optimize $\Psi_\cP$ only over compact subsets of $\cW$. Recall that by our definition of ``good'', the uniqueness of the maximizer is guaranteed for good $\cP$.

\subsection{Maximization of $\Phi_\cP$ with fixed number of points}\label{sec:AnaPhik}

\noindent In this section we derive, for a given point measure with finite support, $\cP = \sum_{i=1}^k \delta_{(f_i,y_i)}$, explicit information about the maximization of $\Phi_\cP(\mu)$ over $\mu$. We need  slightly adapted notation. 
Since we optimize here only  $\Phi_\cP(\mu)$ over $\mu$, we can also drop the points $y_1,\dots,y_k$; see Definition~\ref{def-functionals}. 
We obtain explicit information about the maximising vector $w=(w_1,\dots,w_k)=(\mu(f_i,y_i))_{i=1}^k$.

Fix $\theta \in (0,\infty)$ as always. Furthermore, we fix $k\in\N$,  assume that $\cP=\sum_{i=1}^k\delta_{(f_i,y_i)}$ and therefore may restrict our maximization problem to $\mu$ of the form $\mu=\sum_{i=1}^k w_i\delta_{(f_i,y_i)}$. Then a comparison to Definition~\ref{def-functionals} shows that 
\begin{align}
\label{eqn:max_varphi_k_is_max_Phi}
\sup_{\mu\in\cW}\Phi_\cP(\mu) = \varphi_k(f_1,\dots,f_k),
\end{align}
where  $\varphi_k\colon (0,\infty)^k\to[0,\infty)$ is defined as
\begin{equation}
\label{eqn:varphi_k_def}
\varphi_k( f_1,\dots, f_k)  = 
\sup_{\substack{w_1,\dots,w_k\geq 0\\ \sum_{i=1}^k w_i\leq 1}} \sum_{i=1}^k \bigg( w_i f_i-\theta w_i^2 \bigg).
\end{equation}
We are going to analyze the function $\varphi_k$ in this section. 

Since $\varphi_k(f_1,\dots,f_k)$ does not depend on the order of the $f_i$, we may assume them to be ordered in a decreasing way. The following is the main result of this section; it identifies the optimal $w_1,\dots,w_k$ and thus the optimal $\mu$, provides some of its properties and shows its uniqueness.

\begin{prop}[Analysis of $\varphi_k$]\label{p:AnaPhik}
Fix $k\in\N$ and $f_1\geq f_2\geq\dots\geq f_k > 0$. 

{\it \underline{Case 1}: $f_1+\dots+f_k\geq 2\theta$.} 
Let $k_\star = K_\star(f_1,\dots,f_k)$, where
\begin{equation}\label{e:defkstar}
K_\star\big(f_1, \ldots, f_k\big) := 
\inf \Big\{ j \in \{1, \ldots, k-1\} \colon\, j f_{j+1} \leq \sum_{i=1}^j f_i - 2\theta \Big\} \wedge k,
\end{equation}
where we interpret $\inf \emptyset = \infty$.
Then the unique maximizer in \eqref{eqn:varphi_k_def} is given by
\begin{equation}\label{e:wmax}
w_i := 
\begin{cases}
\frac1{2\theta}\big[f_i - \tfrac{1}{k_\star}(\sum_{j=1}^{k_\star} f_j - 2\theta)\big] & \text{if } i \leq k_\star, \\
0 & \text{otherwise.}
\end{cases}
\end{equation}
Moreover, $w_i>0$ for $i\in \{1,\dots,k_\star\}$, $w_1+\dots+w_{k_*}=1$ and 
\begin{align}\label{defphisupt}
\varphi_k(f_1, \ldots, f_k) = \varphi_{k_\star}(f_1, \ldots, f_{k_\star}) 
= \frac{1}{4\theta} \Big(  \sum_{i=1}^{k_\star} f^2_i  - \frac{1}{ k_\star}\Big(\sum_{i=1}^{k_\star}f_i - 2\theta \Big)^2 \Big) .
\end{align}
In particular,
\begin{equation}\label{e:ineqPhik}
\sum_{i=1}^{k_\star - 1} \frac{f_i^2}{4\theta} < \varphi_k(f_1, \ldots, f_k) \leq \sum_{i=1}^{k_\star} \frac{f_i^2}{4 \theta}.
\end{equation}

{\it \underline{Case 2}: $f_1+\dots+f_k< 2\theta$.}
Then the unique maximizer is given by $w_i = f_i/2\theta$, $1\leq i \leq k$. 
Moreover, $w_1+\dots+w_k<1$ and
\begin{equation}\label{defphilesst}
\varphi_k(f_1, \ldots, f_k)  = \sum_{i=1}^k \frac{f_i^2}{4 \theta}.
\end{equation}
\end{prop}

The proof of Proposition \ref{p:AnaPhik} builds on the following lemma, and is given below the proof of Lemma~\ref{l:optimizing_over_w}.

\begin{lemma}
\label{l:optimizing_over_w}
Let $k\in \N$ and $f_1\ge f_2 \ge \cdots f_k \ge 0$. 
\begin{enumerate}
\item The map 
\label{item:over_positive_w}
\begin{align}
\label{eqn:sum_w_i_f_i}
(w_1,\dots,w_k)\mapsto \sum_{i=1}^k [w_i f_i-\theta  w_i^2]
\end{align}
is maximized over $[0,\infty)^k$ precisely for $w_i = \frac{f_i}{2\theta}$. 
\item 
\label{item:over_positive_with_trivial_extra_condition}
Suppose $\sum_{i=1}^k \frac{f_i}{2\theta} \le 1$. Then \eqref{eqn:sum_w_i_f_i} is maximized over $(w_1,\dots,w_k) \in[0,\infty)^k$ under the constraint $\sum_{i=1}^k w_i \le 1$ precisely by $w_i = \frac{f_i}{2\theta}$. 
\item 
\label{item:over_all_w_i_with_sum_gamma}
Let $\gamma \in \R$. 
Then \eqref{eqn:sum_w_i_f_i} (as a function on $\R^k$) is maximized over $(w_1,\dots,w_k)$ in $ \R^k$ under the constraint $\sum_{i=1}^k w_i = \gamma$ by 
\begin{align}
\label{eqn:w_i_with_sum_equal_gamma}
w_j = \frac{f_j}{2\theta} + \frac1k \Big(\gamma - \sum_{i=1}^k \frac{f_i}{2\theta} \Big) \forqq{j\in \{1,\dots,k\}}. 
\end{align}
\item 
\label{item:positive_w_i_with_sum_gamma}
Let $\gamma \in [0,\infty)$ and suppose that 
$k f_k + 2\theta \gamma - \sum_{i=1}^k f_i \ge 0$. 
Then \eqref{eqn:sum_w_i_f_i} is maximized over $(w_1,\dots,w_k) \in[0,\infty)^k$ under the constraint $\sum_{i=1}^k w_i = \gamma$ by 
\eqref{eqn:w_i_with_sum_equal_gamma}, and we have 
\begin{align}
\label{eqn:optimal_sum_with_gamma_sum_of_w}
\sum_{i=1}^k w_i f_i-\theta  w_i^2 
& = \frac{1}{4\theta} \sum_{i=1}^k f_i^2 - \frac{\theta}{k} \Big( \gamma - \sum_{i=1}^k \frac{f_i}{2\theta} \Big)^2 . 
\end{align}
\item 
\label{item:positive_w_i_with_sum_1_by_extra_condition}
Suppose $\sum_{i=1}^k \frac{f_i}{2\theta} \ge 1 $ and  
$k f_k + 2\theta  - \sum_{i=1}^k f_i \ge 0$. 
Then \eqref{eqn:sum_w_i_f_i} is maximized over $[0,1]^k$ under the constraint $\sum_{i=1}^k w_i \le 1$ by \eqref{eqn:w_i_with_sum_equal_gamma} with $\gamma =1$. 
\item 
\label{item:zero_is_better}
Suppose $k f_k + 2\theta - \sum_{i=1}^{k} f_i \le 0$ or equivalently
\begin{align}
\label{eqn:k_is_not_active}
(k-1) f_k + 2\theta - \sum_{i=1}^{k-1} f_i \le 0. 
\end{align}
Then, if \eqref{eqn:sum_w_i_f_i} is maximized by $(w_1,\dots,w_k)\in[0,1]^k$ with $\sum_{i=1}^k w_i \le 1$, then $w_k=0$. 
\end{enumerate}
\end{lemma}

\begin{proof}
\ref{item:over_positive_w} follows by the fact that $w_i \mapsto w_i f_i-\theta  w_i^2$ is concave for all $i$, so that the maximum is attained where its derivative equals zero (or at the boundary, i.e., for $w_i=0$, but this gives an outcome that is clearly less than for $w_i = \frac{f_i}{2\theta}$). 

\ref{item:over_positive_with_trivial_extra_condition} follows immediately from \ref{item:over_positive_w}. 

\ref{item:over_all_w_i_with_sum_gamma} is proved by using the Lagrange multiplier method:
Define $L : \R^{k+1} \rightarrow \R$ by 
\begin{align*}
L(w_1,\dots,w_k,\lambda) := \sum_{i=1}^k w_i f_i -\theta \sum_{i=1}^k w_i^2 - \lambda \Big(\sum_{i=1}^k w_i - \gamma\Big)  \forqq{w_1,\dots,w_k,\lambda  \in [0,\infty)}.
\end{align*}
$(w_1,\dots,w_k,\lambda)$ is the extremal point for $L$ if $\nabla L(w_1,\dots,w_k,\lambda)=0$, which is the case if 
\begin{align*}
f_i - \lambda - 2\theta w_i & = 0 \qquad \mbox{ for all } i, \qquad \mbox{ and } \qquad 
\sum_{i=1}^k w_i  = \gamma. 
\end{align*}
Combining gives $\lambda = \frac 1k \sum_{i=1}^k f_i - 2\theta \gamma$ and \eqref{eqn:w_i_with_sum_equal_gamma}. This extremal point for $L$ is the maximizer for $L$ over $\R^k$ as $L$ is concave and because $\lim_{|x| \rightarrow \infty} x f_i - \theta x^2 = - \infty$ for all $i$. 

\ref{item:positive_w_i_with_sum_gamma} follows from \ref{item:over_all_w_i_with_sum_gamma} as the condition implies  that $k f_j + 2\theta \gamma - \sum_{i=1}^k f_i \ge 0$ (remember $f_j \ge f_k$) for all $j$ and thus $w_j\ge 0$ for $w_j$ as in  \eqref{eqn:w_i_with_sum_equal_gamma}, i.e., 
\begin{align*}
w_j = \frac{f_j}{2\theta} + \frac1k \Big(\gamma - \sum_{i=1}^k \frac{f_i}{2\theta} \Big)  \ge 0. 
\end{align*}
As furthermore, 
\begin{align*}
\frac{f_j}{\theta} -  w_j = \frac{f_j}{2\theta} - \frac1k \Big( \gamma - \sum_{i=1}^k \frac{f_i}{2\theta} \Big),
\end{align*}
we have obtain \eqref{eqn:optimal_sum_with_gamma_sum_of_w} by the following equality: 
\begin{align*}
\sum_{i=1}^k w_i f_i-\theta  w_i^2 
& = \sum_{i=1}^k w_i (f_i-\theta  w_i) 
= \theta \sum_{i=1}^k \Big( \Big(\frac{f_i}{2\theta}\Big)^2 - \frac{1}{k^2} \big( \gamma - \sum_{i=1}^k \frac{f_i}{2\theta} \big)^2 \Big).
\end{align*}

\ref{item:positive_w_i_with_sum_1_by_extra_condition} follows from \ref{item:positive_w_i_with_sum_gamma} as one observes that \eqref{eqn:optimal_sum_with_gamma_sum_of_w} is maximal when $\gamma$ is closest to $\sum_{i=1}^k \frac{f_i}{2\theta}$. 

\ref{item:zero_is_better} 
Suppose $\widetilde w_1,\dots, \widetilde w_k \in [0,1]$ for $k\ge 2$ are such that $\sum_{i=1}^k \widetilde w_i \le 1$ (we may assume $k\ge 2$ as $\theta>0$ so that \eqref{eqn:k_is_not_active} cannot be satisfied for $k=1$). 
Let us define $w_1,\dots, w_k$ by $w_k=0$ and for $i\in \{1,\dots,k-1\}$
\begin{align*}
w_i := \widetilde w_i + \tfrac{1}{k-1} \widetilde w_k. 
\end{align*}
Then by writing $\gamma = \sum_{i=1}^k \widetilde w_k = \sum_{i=1}^k w_k$, we see that
\begin{align*}
\sum_{i=1}^k \widetilde w_i f_i - \theta \widetilde w_i^2  - \left( \sum_{i=1}^k  w_i f_i - \theta  w_i^2  \right) 
& = \widetilde w_k f_k - \theta \widetilde w_k^2 
+ \sum_{i=1}^{k-1} (\widetilde w_i - w_i) (f_i - \theta (\widetilde w_i + w_i ) ) \\
& = \widetilde w_k f_k - \theta \widetilde w_k^2 - \frac{\widetilde w_k }{k-1}  
\Big( \sum_{i=1}^{k-1} f_i - 2 \theta \gamma \Big) \\
& = \frac{ \widetilde w_k}{k-1} \left( (k-1) f_k - \sum_{i=1}^{k-1} f_i  +2\theta - (1-\gamma) 2\theta - (k-1)\theta \widetilde w_k \right) \\
& \le - \frac{\widetilde w_k}{k-1} \left(  (1-\gamma) 2\theta + (k-1) \theta \widetilde w_k  \right) 
\le -  \theta \widetilde w_k^2 . 
\end{align*}
This proves that the maximizer has to satisfy $w_k =0$. 
\end{proof}

\begin{proof}[Proof of Proposition \ref{p:AnaPhik}]
Case 2 follows directly from Lemma \ref{l:optimizing_over_w}~\ref{item:over_positive_with_trivial_extra_condition}. 

In Case 1, observe first that $j f_{j+1} \leq \sum_{i=1}^j f_i - 2 \theta$ for all $j > k_\star$, 
and that $f_1+\cdots+f_{k_\star} \geq 2 \theta$. 
By definition of $K_\star$ one has $(k_\star -1) f_{k_\star} > \sum_{j=1}^{k_\star -1} f_j -2 \theta$ and thus $f_{k_\star} > \frac{1}{k_\star} ( \sum_{j=1}^{k_\star} f_j - 2\theta)$ and so $w_1\ge w_2 \ge \cdots \ge  w_{k_\star}>0$. 
By Lemma \ref{l:optimizing_over_w}~\ref{item:zero_is_better} it follows that $w_i = 0$ for $i >k_\star$ and so by \ref{item:positive_w_i_with_sum_1_by_extra_condition} one completes the proof. 
\end{proof}

\subsection{Some topological properties of the variational formula}\label{sec:reformulation}

\noindent In this section, we prove that the functional $\cP \mapsto -\Psi_\cP$ introduced in Definition~\ref{def-functionals} is Gamma-continuous in the vague topology, which is the crucial property under which we can find later arguments for the existence of maximizers and continuity properties of the maximizers as a function of $\cP$. The main tool of the arguments is a characterization of the vague convergence of point measures in terms of one-by-one convergence of its points.

Let us introduce the crucial sense of convergence for variational formulas.

\begin{definition}[Gamma convergence]\label{def:gamma_convergence}
Let $X$ be a metric space. 
Let $f,f_1,f_2,\dots\colon X\rightarrow [-\infty,\infty]$.  We say that the sequence $(f_n)_{n\in\N}$ \emph{Gamma converges} to $f$, written $f_n \arrowgamman f$, if 
\begin{enumerate}[label={\normalfont(\roman*)}]
\item 
\label{item:gamma_liminf}
for all $x\in X$ and all sequences $(x_n)_{n\in\N}$ in $X$ with $x_n \rightarrow x$, 
\begin{align*}
f(x) \le \liminfn f_n(x_n), 
\end{align*}
\item 
\label{item:gamma_limsup}
for all $x\in X$ there exists a sequence $(x_n)_{n\in\N}$ in $X$ such that $x_n \rightarrow x$ and 
\begin{align*}
f(x) \ge \limsupn f_n(x_n). 
\end{align*}
\end{enumerate}
\end{definition}

\begin{remark}
\label{remark:gamma_and_lsc}
Observe that $f\arrowgamman f$ if and only if $f$ is lower semi-continuous. 
\hfill$\Diamond$ \end{remark}

We use the following statements about $\Gamma$-convergence, which are sometimes referred to as the Fundamental Theorem(s) of Gamma convergence: 

\begin{theorem}
\label{theorem:gamma_convergence}
Let $X$ be a metric space. 
Let $f,f_1,f_2,\dots\colon X\rightarrow [-\infty,\infty]$. Suppose $f_n \arrowgamman f$. 
\begin{enumerate}
\item 
\label{item:gamma_conv_on_cpt}
\ncite[Proposition 1.18]{Bra02} For each compact $K\subset X$ 
\begin{align*}
\inf_{x\in K} f(x) \le \liminfn \inf_{x\in K} f_n(x). 
\end{align*}
\item 
\label{item:gamma_conv_optimizers}
\ncite[Theorem 1.21]{Bra02} 
Suppose there exists a compact set $K\subset X$ such that $\inf_{x\in X} f_n(x) = \inf_{x\in K} f_n(x)$ for all $n\in\N$. 
Suppose that $x_1,x_2,\dots \in X$ are such that $f_n(x_n) = \inf_{x\in X} f_n(x)$ for all $n\in\N$. Then there exists a subsequence of $(x_n)_{n\in\N}$ that converges to an $y\in X$ for which $\inf_{x\in X} f(x) = f(y)$. 
\end{enumerate}
\end{theorem}

The main result of Section~\ref{sec:reformulation} is the following proposition. Part \ref{item:compactness} will allow us to restrict the search for a maximizer $\mu^*$ of $\Psi_\cP$ to those $\mu$ whose $\R^d$-support is within some box in $\R^d$. Recall that $\cW$ is the set of subprobability measures on $(0,\infty)\times \R^d$  with total mass $\leq 1$, and $Q_R=[-R,R]^d$. 
Furthermore, we introduce
\begin{align}\label{eqn:def_cW_R}
\cW_R 
& = \{ \mu \in \cW \colon \supp \mu \subset  (0,\infty) \times Q_R  \},\qquad R>0. 
\end{align}

\begin{proposition}
\label{p:gamma_convergence_Psi}
Let $\cP, \cP_1,\cP_2,\cP_3,\dots$ be in $\cM^\circ_\rp$ such that $\cP_n \to \cP$ in $\cM_\rp$. Then 
\begin{enumerate}
\item {\em Compactness}
\label{item:compactness}
\begin{equation*}
\lim_{R \to \infty}  \sup_{\mu \in \cW \setminus \cW_R}  \sup_{n\in \N} \Psi_{\cP_n} (\mu)  = - \infty. 
\end{equation*}
\item {\em Gamma convergence of $-\Psi$}
\label{item:gamma_convergence_Psi}
\begin{align*}
-\Psi_{\cP_n} \arrowgamman -\Psi_{\cP}.
\end{align*}
\end{enumerate}
\end{proposition}

The proof of this proposition is at the end of this section. We prepare for the proof by citing a well-known result from point-process theory about a characterization of vague convergence by point-wise convergence.
 For $\cP \in \cM_\rp$ and $L>0$, recalling that $Q_L=[-L,L]^d$, we denote by $\cP^{\ssup{L}}$ the point measure $\1_{[L^{-1},\infty) \times Q_L} \cP$, which means $\frac{\dd \cP^{\ssup{L}}}{\dd \cP} =  \1_{[L^{-1},\infty) \times Q_L} $, i.e., 
\begin{align}
\label{eqn:ssup_L}
\cP^{\ssup L}(A)=\cP \Big(A \cap \big[ [L^{-1},\infty)\times Q_L \big] \Big)
\end{align}
for any Borel measurable $A\subset \fE$. 
For $\mu \in \cW$ we also write $\mu^{\ssup{L}} = \1_{[L^{-1},\infty) \times Q_L} \mu$. 
 
 Observe that as $[L^{-1},\infty) \times Q_L \subset \cH_h^s $ for some $h,s>0$ (e.g. $s= \frac14$ and $h= \frac{L}{2}$), and $\cH_h^s$ is relatively compact in $\fE$, $\cP([L^{-1},\infty) \times Q_L) \in \N_0$ for all $\cP \in \cM_\rp$. 

\begin{lemma}
\label{l:convergence_on_compacts}
Let $\cP,\cP_1,\cP_2,\dots \in \cM_\rp$ and $L>0$ be such that $\cP_n \rightarrow \cP$ in $\cM_\rp$ and $\cP(\partial([L^{-1},\infty)\times Q_L))=0$. 
\begin{enumerate}
\item Put $k = \cP([L^{-1}, \infty) \times Q_L)  \in \N_0$. Then there exist $(f_i,y_i), (f_i^n,y_i^n) \in [L^{-1},\infty)\times Q_L$,  for $n\in \N$ and $i\in\{1,\dots,k\}$ 
such that, for all large enough $n\in\N$ (with empty sums interpreted as zero),
\begin{align*}
& \cP_n^{\ssup L} = \sum_{i=1}^k \delta_{(f_i^n,y_i^n)}, \quad 
\cP^{\ssup L} = \sum_{i=1}^k \delta_{(f_i,y_i)}, \quad 
(f_i^n, y_i^n) \arrown (f_i,y_i), \quad i\in \{1,\dots,k\}. 
\end{align*}

\item 
\label{item:conv_measures_LL_pointm_on_compacts}
Suppose $\mu,\mu_1,\mu_2,\dots$ are in  $\cW$ such that $\mu_n \rightarrow \mu$ in $\cW$ and $\mu_n \ll \cP_n$ for all $n\in\N$. 
Then $\mu \ll \cP$ and, with $k$, $(f_i,y_i),(f_i^n,y_i^n)$ as above, there exist
$(w_1^n, \dots, w_k^n),(w_1, \dots, w_k) \in [0,1]^k$ such that, 
for all large enough $n$,
\begin{align*}
& \mu_n^{\ssup L} = \sum_{i=1}^k w_i^n \delta_{(f_i^n,y_i^n)}, \quad 
\mu^{\ssup L} = \sum_{i=1}^k w_i \delta_{(f_i,y_i)}, \quad w_i^n \arrown w_i, \qquad i\in \{1,\dots,k\}.
\end{align*}
\end{enumerate}
\end{lemma}
\begin{proof}
For the first statement, note that $[L^{-1}, \infty) \times Q_L$ is a relatively compact subset of $\mathfrak{E}$
and apply \cite[Proposition~3.13]{Re87}. 
\begin{calc}
(See also Theorem~\ref{theorem:resnick_on_compact}.)
\end{calc}
The second statement is a straightforward consequence of the first.
\begin{calc}
That $\mu \ll \cP$ follows by the fact that from the convergences one obtains $\mu^{\ssup{L}} \ll \cP^{\ssup{L}}$ for all $L>0$: Let $f_L$ be a density function which equals zero outside $[L,\infty) \times Q_L$. Then $\mu = \limL \mu^{\ssup{L}} = \limL f_L \cP^{\ssup{L}}= \limL f_L \cP$. 
Hence $f= \limL f_L = \sup_{L\in\N} f_L$ is the density for $\mu$ with respect to $\cP$. 
\end{calc}
\end{proof}

Here is the main step in the proof of Proposition~\ref{p:gamma_convergence_Psi}.

\begin{lemma}
\label{l:upper_semicontinuity_for_parts_of_tilde_Psi}
Let $\cP , \cP_1,\cP_2,\dots \in \cM_\rp^\circ$ be such that $\cP_n \to \cP$ in $\cM_\rp$, and let $\mu,\mu_1,\mu_2,\dots \in \cW$ be such that $\mu_n \to \mu$ in $\cW$.
Then
\begin{enumerate}
\item 
\label{item:cD_cP_gamma_conv}
$\cD_{\cP}(\mu) \leq \liminf_{n \to \infty} \cD_{\cP_n}(\mu_n)$.
\item 
\label{item:conv_tilde_Phi}
If $\mu_n \ll \cP_n$ and there exists a $R>0$ such that  $\mu_n \in \cW_R$ for all $n \in \N$, then
\begin{align*}
\Phi_{\cP}(\mu) \geq \limsup_{n \to \infty} \Phi_{\cP_n}(\mu_n). 
\end{align*}
\end{enumerate}
\end{lemma}
\begin{proof}
\ref{item:cD_cP_gamma_conv}
If $\mu \not \ll \cP$, then by Lemma~\ref{l:convergence_on_compacts} 
there exists an $N\in \N$ such that $\mu_n \not \ll \cP_n$ for all $n\ge N$,
and the conclusion trivially holds. 
Therefore, we may assume $\mu \ll \cP$ and $\mu_n \ll \cP_n$ for all $n\in\N$. 
Moreover, we may assume that $\mu \ne 0$. 
Let $L>0$ be such that $\cP$ has zero measure on the boundary of $[L^{-1}, \infty) \times Q_L$ and $\cD_\cP(\mu^{\ssup L}) > 0$, which implies $\mu^{\ssup L} \neq 0$. 
Let $k$, $f_i^n,y_i^n,w_i^n$, $f_i,y_i,w_i$ be as in Lemma~\ref{l:convergence_on_compacts}, and note that $k\ge 1$. 
Let $i_1,\dots,i_m \in \{1,\dots,k\}$, $m\in\N$, be the distinct indices such that $w_{i_j} > 0$, $j\in \{1,\dots,m\}$, and $w_\ell=0$ otherwise.
We may assume that, for all $i \leq k$ and all $n$ large enough, $w_i>0$ implies $w_i^n>0$. 
Then
\begin{align*}
\cD_{\cP_n}(\mu_n) 
\ge \cD_{\cP_n}(\mu_n^{\ssup L})
\ge D_0(y_{i_{1}}^n,\dots,y_{i_{m}}^n) \arrown D_0(y_{i_{1}},\dots,y_{i_{m}}) = \cD_\cP(\mu^{\ssup L}). 
\end{align*}
Therefore, for any $L>0$, $\liminfn \cD_{\cP_n} (\mu_n) \ge \cD_\cP (\mu^{\ssup L})$. 
Since $\cD_{\cP}(\mu) = \sup_{L>0} \cD_{\cP}(\mu^{\ssup L})$, the claim follows. 

\ref{item:conv_tilde_Phi} 
Let $R>0$, $\mu_n \in \cW_R$, $\mu_n \ll \cP_n$ for all $n\in\N$ and $\mu_n \rightarrow \mu$ in $\cW$. 
Note that this implies $\mu \in \cW_R$ as well and, by Lemma~\ref{l:convergence_on_compacts}~\ref{item:conv_measures_LL_pointm_on_compacts}, $\mu \ll \cP$.
Let $\epsilon>0$. Let us first show that for large $L>0$ 
\begin{align}
\label{eqn:smallness_difference_box_uniform}
\Phi_{\cP_n}(\mu_n) < \Phi_{\cP_n}(\mu_n^{\ssup L}) + \epsilon \ \mbox{ for all } n\in\N \qquad \text{ and } \qquad
|\Phi_\cP(\mu) - \Phi_\cP(\mu^{\ssup L})| < \epsilon. 
\end{align}
Indeed, take $L>R$ such that $L^{-1} < \epsilon$,
$\mu((0,L^{-1}) \times Q_L) < \epsilon/\theta$ and $\cP(\partial ( [L^{-1}, \infty) \times Q_L)) = 0$.
Then
\begin{align*}
\Phi_{\cP_n}(\mu_n) - \Phi_{\cP_n}(\mu_n^{\ssup L}) = \int_{(0,L^{-1}) \times Q_L} \left[ f - \theta \frac{\dd \mu_n} {\dd \cP_n}(f,y) \right]\dd \mu_n (f,y) \leq L^{-1} < \varepsilon,
\end{align*}
and the same inequality is valid with $\cP_n$, $\mu_n$ replaced by $\cP$, $\mu$, for which also
\begin{align*}
\Phi_{\cP}(\mu) - \Phi_{\cP}(\mu^{\ssup L}) = \int_{(0,L^{-1}) \times Q_L} \left[ f - \theta \frac{\dd \mu}{\dd \cP}(f,y) \right] \dd \mu(f,y) \geq - \theta \mu((0,L^{-1}) \times Q_L) > - \epsilon,
\end{align*}
where we used that $\frac{\dd \mu }{\dd \cP} \leq 1$. This concludes \eqref{eqn:smallness_difference_box_uniform}.
Now it is enough to show that $\Phi_{\cP_n}(\mu_n^{\ssup L}) \to \Phi_{\cP}(\mu^{\ssup L})$, 
but this follows by Lemma~\ref{l:convergence_on_compacts}: with $k$,$f_i,f_i^n,y_i,y_i^n,w_i,w_i^n$ as therein and $n$ large enough,
\begin{align*}
\Phi_{\cP_n}(\mu_n^{\ssup L}) = 
\sum_{i=1}^k \Big( w_i^n f_i^n - \theta (w_i^n)^2 \Big)
\rightarrow 
\sum_{i=1}^k \Big( w_i f_i - \theta w_i^2  \Big)
=
 \Phi_{\cP}(\mu^{\ssup L}). 
\end{align*} 
\end{proof}

As a by-product of the proof, we obtained: 

\begin{lemma}
\label{l:continuity_Psi_in_ssup_L}
Let $\cP \in \cM_\rp^\circ$ and $\mu \in \cW_R$ for some $R\in (0,\infty)$. 
Then 
\begin{align*}
\Phi_\cP(\mu^{\ssup L}) \arrowL \Phi_\cP(\mu),
\qquad 
 \cD_\cP(\mu^{\ssup L}) \arrowL \cD_\cP(\mu), 
\qquad 
\Psi_\cP(\mu^{\ssup{L}}) \xrightarrow{L \rightarrow \infty} \Psi_\cP(\mu). 
\end{align*}
\end{lemma}
\begin{proof}
This follows by definition of $\cD_\cP$ and by \eqref{eqn:smallness_difference_box_uniform}. 
\end{proof}

\begin{proof}[Proof of Proposition~\ref{p:gamma_convergence_Psi}]
\ref{item:compactness}
It suffices to show that given any $A>0$, there exists an $R_0>0$ such that,  
\begin{align*}
\Psi_{\cP_n}(\mu) \le -  A  \qquad  \mbox{ for all $R \ge R_0$, $\mu \in \cW  \setminus \cW_R$ and $n\in\N$}. 
\end{align*}
Recall the definition of $M_R$ from \eqref{e:defMr} and recall that $q = \frac{d}{\alpha -d} >0$.
By Lemma~\ref{l:basicgoodprop}, there exists an $R_0 >0$ such that
\begin{align}
\label{eqn:conditions_R_0_in_terms_of_max}
\max\{A, \sup_{n \in \N }M_R(\cP_n) \} \le \tfrac12 q (R-1) \qquad \mbox{ for all } R \ge R_0. 
\end{align}
Let $R\ge R_0$, $\mu \in \cW \setminus \cW_R$ and $n\in\N$. 
We decompose $(0,\infty)\times \R^d$ into
\begin{align*}
S_R^{\ssup 0}:=(0,\infty) \times Q_R, \qquad S_{R}^{\ssup k}:= (0,\infty)\times \big[Q_{R+k} \setminus Q_{R+k-1}\big] \mbox{ for } k\in\N. 
\end{align*}
Observe that $\zeta:=\mu((0,\infty) \times \R^d)$ is in $(0,1]$.
Write $\zeta = \sum_{k\in\N_0} \zeta_k$ with $\zeta_k := \mu( S_R^{\ssup k})$ for $k\in\N_0$.
Note that $\zeta_k>0$ implies $\cD_{\cP_n}(\mu) \ge R+k-1 $, so that $\zeta_k \cD_{\cP_n}(\mu) \ge \zeta_k (R+k-1)$ for $k\in\N_0$.
Since $\Psi_{\cP_n}(\mu) = -\infty$ if $\mu \not \ll \cP_n$, we may and do assume $\mu \ll \cP_n$. 
Hence we have the lower bound $\zeta \cD_{\cP_n}(\mu)  = \sum_{k\in\N_0} \zeta_k \cD_{\cP_n}(\mu) 
\ge \sum_{k\in\N_0}  \zeta_k (R+k-1)$. Furthermore, we have the upper bound
$$
\Phi_{\cP_n}(\mu)  \le \sum_{k\in\N_0} \int_{S_R^{\ssup k}} f \,{\rm d}\mu(f,y) 
\le \sum_{k\in\N_0} \zeta_k M_{R+k}(\cP_n).
$$
Together with \eqref{eqn:conditions_R_0_in_terms_of_max} and $D_{\cP_n}(\mu) \geq R \geq R_0$, this gives 
\begin{align*}
\Psi_{\cP_n}(\mu)
& = \Phi_{\cP_n}(\mu) -  q \cD_{\cP_n}(\mu)
\leq \sum_{k\in\N_0} \zeta_k \left[ \tfrac12 q (R+k-1) - q (R+k-1) \right] - q(1-\zeta)R \\
& \leq \sum_{k\in\N_0} \zeta_k (-A) - (1-\zeta)A 
= - A. 
\end{align*}

\ref{item:gamma_convergence_Psi}
Let us first show \ref{item:gamma_liminf} of Definition~\ref{def:gamma_convergence}. Pick $\mu,\mu_1,\mu_2,\dots\in\cW$ such that $\mu_n\to\mu$.
We have to show that $-\Psi_\cP(\mu) \le \liminfn -\Psi_{\cP_n}(\mu_n)$, i.e., $\Psi_\cP(\mu) \ge \limsupn \Psi_{\cP_n}(\mu_n)$. 
By \ref{item:compactness} we may assume that there exists an $R>0$ such that $\mu,\mu_n \in \cW_R$ for all $n\in\N$, because if such $R$ does not exist, then $\limsup_{n \to \infty} \Psi_{\cP_n}(\mu_n) = - \infty$.
Passing to subsequences if necessary, we may also assume that $\mu_n \ll \cP_n$ for all $n\in\N$ and, by Lemma~\ref{l:convergence_on_compacts}~\ref{item:conv_measures_LL_pointm_on_compacts}, $\mu \ll \cP$. 
In this case, the desired statement is a direct consequence of Lemma~\ref{l:upper_semicontinuity_for_parts_of_tilde_Psi}.

To verify \ref{item:gamma_limsup} of Definition~\ref{def:gamma_convergence}, let $\mu \in \cW$. 
Assume first that $\Psi_\cP(\mu) = -\infty$. 
Since $\cP_n$ is countable for each $n$, there exists a $(f,y) \in (0,\infty)\times \R^d \setminus \bigcupn \cP_n$. 
Setting $\mu_n = (1-\tfrac1n)\mu + \tfrac1n \delta_{(f,y)}$,
it is clear that $\mu_n \to \mu$ and $\Psi_{\cP_n}(\mu_n) = -\infty$ for every $n$.

Assume now that $\Psi_\cP(\mu) > -\infty$,
which implies $\mu \ll \cP$. 
As we will soon see, it suffices to show the following: for every $L>0$, 
there exists a sequence $\mu_n$ with $\mu_n=\mu_n^{\ssup L}$ such that $\cD_{\cP_n}(\mu_n) \rightarrow \cD_{\cP}( \mu^{\ssup L} )$
and $\Phi_{\cP_n}(\mu_n) \to \Phi_\cP(\mu^{\ssup L})$.
To prove the latter, fix $L>0$ and take $k$, $(f_i,y_i),(f_i^n,y_i^n)$ as in Lemma~\ref{l:convergence_on_compacts}. 
Define $w_i = \mu (f_i,y_i)$ for $i \in \{1,\dots,k\}$ and $\mu_n := \sum_{i=1}^k w_i \delta_{(f_i^n,y_i^n)}$. Then 
\begin{align*}
\cD_{\cP_n}(\mu_n) = D_0 (\{y_j^n \colon w_j>0\}) \rightarrow D_0(\{y_j \colon\, w_j > 0\}) = \cD_\cP(\mu^{\ssup L}),
\end{align*}
$\mu_n \to \mu^{\ssup L}$ and $\Phi_{\cP_n}(\mu_n) \to \Phi_\cP(\mu^{\ssup L})$ as well, 
as shown in the last line of the proof of Lemma~\ref{l:upper_semicontinuity_for_parts_of_tilde_Psi}.

Now, for each $m\in\N$, we can find a sequence $(\nu_{m,n})_{n\in\N}$ in $\cW$ with $\nu_{m,n}^{\ssup{m}}=\nu_{m,n}$ such that $\cD_{\cP_n}(\nu_{m,n}) \rightarrow \cD_{\cP}( \mu^{\ssup{m}} )$ and $\Phi_{\cP_n}(\nu_{m,n}) \rightarrow \Phi_\cP(\mu^{\ssup m})$ as $n\rightarrow \infty$.
Let $(N_m)_{m\in\N}$ be a strictly increasing sequence in $\N$ such that 
\begin{align*}
\Big|\cD_{\cP_n}(\nu_{m,n}) - \cD_{\cP}( \mu^{\ssup{m}} )\Big| \vee \Big|\Phi_{\cP_n}(\nu_{m,n}) -  \Phi_\cP(\mu^{\ssup m})\Big| < \frac1m \qquad \text{ for all } n \geq N_m.
\end{align*}
Define $m_n := \max\{m \in \N: N_m \leq n\}$ and $\mu_n := \nu_{m_n, n}$.
Note that $m_n \to \infty$, $\mu_n \to \mu$ and $n \geq N_{m_n}$, so that by Lemma~\ref{l:continuity_Psi_in_ssup_L},
\begin{align*}
\Big|\Psi_{\cP_n}(\mu_n) - \Psi_{\cP}(\mu)\Big| \leq 
\Big|\Psi_{\cP_n}(\mu_n) - \Psi_{\cP}(\mu^{\ssup {m_n}})\Big|
+ \Big|\Psi_\cP(\mu^{\ssup {m_n}}) - \Psi_\cP(\mu)\Big| \arrown 0. 
\end{align*}
\end{proof}

With the convergence of Lemma~\ref{l:continuity_Psi_in_ssup_L} and the compactness in Proposition~\ref{p:gamma_convergence_Psi}~\ref{item:compactness}, we prove Lemma~\ref{l:Xi_as_sup_over_probm_fF_1}:

\begin{proof}[Proof of Lemma~\ref{l:Xi_as_sup_over_probm_fF_1}]
By Proposition~\ref{p:gamma_convergence_Psi}~\ref{item:compactness} it follows that there exists an $R>0$ such that $\Xi(\cP)$ equals $\sup_{\nu \in \cW_R} \Psi_\cP(\nu)$. 
Then, by Lemma~\ref{l:continuity_Psi_in_ssup_L}, it follows that $\Xi(\cP)$ equals $\sup_{\nu \in \fF(\cP)} \Psi_\cP(\nu)$. 
Let $\nu \in \fF(\cP)$ and $\delta>0$.  
We show $\sup_{\mu \in \fF_1(\cP)} \Psi_\cP(\mu) \ge \Psi_\cP(\nu) - 2\delta$. 
Let $(f_1,y_1),\dots, (f_m,y_m)$ be distinct elements of $\cP$ such that $\sum_{i=1}^m f_i \ge 2\theta$ and $D_0(y_1,\dots,y_m) <\delta$. 
Let $k\in\N_0$ and $(f_{m+1},y_{m+1}),\dots, (f_{m+k},y_{m+k})$ be the distinct elements that form the support of $\nu$ (so possibly $k=0$). 
By Proposition~\ref{p:AnaPhik} there exist $w_i$ for $i\in \{1,\dots,m+k\}$ with $\sum_{i=1}^{m+k} w_i =0$ such that for $\mu = \sum_{i=1}^{m+k} w_i \delta_{(f_i,y_i)}$ one has 
\begin{align*}
\Phi_\cP(\mu) = \varphi_{k+m}(f_1,\dots,f_{m+k}) \ge \varphi_m(f_1,\dots,f_m) \ge \Phi_\cP(\nu), \quad 
\cD_\cP(\mu) \le D_0(y_1,\dots,y_{m+k}) \le 2\delta + \cD_\cP(\nu),
\end{align*}
and thus $\Psi_\cP(\mu) \ge \Psi_\cP(\nu) - 2\delta$. 
\end{proof}


\subsection{Maximizers have finite support}\label{sec:SuppFinite}

\noindent In  this section we prove that, if $\cP$ is a good point  measure  in $\cM_\rp^\circ$, then every maximizer $\mu^*$ of $\Psi_\cP$ has a finite support. It is this result that needs one of the two conditions \ref{item:good3_condition_new_points} or \ref{itemgoodfinitelymanypoints} of Definition \ref{def:good_point_measure}. Indeed, we will use \ref{item:good3_condition_new_points} to construct, from a maximization candidate with infinitely many points, a better one with only finitely many points, and we will use \ref{itemgoodfinitelymanypoints} for a simple argument that the maximizer has only finitely many points.

\begin{proposition}[Maximizers have finite support]
\label{p:exist_n_finiten_of_maxis}
Let $\cP$ be a point measure in $\cM_{\rm p}^\circ$ that is good in the sense of Definition~\ref{def:good_point_measure}. 
Then 
\begin{enumerate}
\item 
\label{item:at_least_one_maximizer}
$\Psi_\cP$ has at least one maximizer,
\item 
\label{item:existence_radius_in_which_max_lie}
there exists a $R>0$ such that every maximizer of $\Psi_\cP$ lies in $\cW_R$, 
\item 
\label{item:maximizers_have_finite_support}
every maximizer has finite support, 
\end{enumerate}
and, if $\cP$ satisfies \ref{item:good3_condition_new_points} of Definition~\ref{def:good_point_measure}, then 
\begin{enumerate}[resume]
\item 
\label{item:max_are_prob_and_sum_f_i_larger_2theta}
every maximizer $\nu$ is a probability measure, i.e.,  $\nu = \sum_{i=1}^k w_i \delta_{(f_i,y_i)}$ for some $k\in\N$, $w\in [0,1]^k$ $\sum_{i=1}^k w_i=1$, $f_i \in (0,\infty)$, $y_i \in \R^d$ for $i\in \{1,\dots,k\}$. 
Moreover, $\sum_{i=1}^k f_i > 2\theta$. 
\end{enumerate}
\end{proposition}

\begin{proof} 
\ref{item:at_least_one_maximizer}
By Proposition~\ref{p:gamma_convergence_Psi}, see also Remark~\ref{remark:gamma_and_lsc}, $\Psi_\cP$ is upper semicontinuous. 
$\cW$ is sequentially compact by \cite[Corollary 13.31]{Kl08}. 
Therefore $\Psi_\cP$ has at least one maximizer.

\ref{item:existence_radius_in_which_max_lie}
By Proposition~\ref{p:gamma_convergence_Psi}~\ref{item:compactness}, we may pick $R>0$ so large that 
$\sup_{\mu \in \cW\setminus\cW_R}\Psi_\cP(\mu) <0 \le \sup_{\mu \in \cW_R}\Psi_\cP(\mu)$. 
This implies that every maximizer lies in $\cW_R$. 

\ref{item:maximizers_have_finite_support}
Let $\nu\in\cW_R$ be a maximizer of $\Psi_\cP$. 
Clearly, $\nu\ll\cP$ (otherwise $\Psi_\cP(\nu) = -\infty < \Psi_\cP(0)$).  
Under \ref{itemgoodfinitelymanypoints} of Definition~\ref{def:good_point_measure}, $\nu$ has finite support. Therefore we assume instead that  \ref{item:good3_condition_new_points} of Definition~\ref{def:good_point_measure} holds. 
Moreover, without loss of generality we may assume that the supports of $\cP$ and $\nu$ are infinite. 
We are going to show that there exists a $\mu\in \cW_{R+1}$ such that $\supp \mu$ is a finite set 
and  $\Psi_\cP(\mu) > \Psi_\cP(\nu)$, which implies the claim. 

Let $(f_i,y_i)\in (0,\infty) \times \R^d$ for $i\in\N$ be distinct and such that $\{(f_i,y_i) \colon i\in\N\} = \supp \nu $. 
We may assume that $f_1\geq f_2\geq f_3\geq \dots$ and $f_k \rightarrow 0$ as $k\to \infty$ (due to the fact that $[\epsilon,\infty) \times Q_R$ is relatively compact in $\fE$, because it is a subset of $\cH_h^s$ for some $s,h>0$, and so there are only finitely many $i$ such that $f_i \ge \epsilon$ for all $\epsilon>0$). 
We separate the proof in two cases, depending on $\sum_{i\in\N} f_i$: Case 1: $\sum_{i\in\N} f_i \in (2\theta,\infty]$, Case 2: $\sum_{i\in\N} f_i \in [0,2\theta]$.

\textbf{\underline{Case 1}} $\sum_{i\in\N} f_i \in (2\theta,\infty]$.
The idea is that Proposition~\ref{p:AnaPhik} tells us that $\mathcal D_\cP$ is maximized using a finite number of points such that adding points the $\Phi$ part will not enlarge, but the $\cD$ part will increase. 

Let $\delta>0$ be such that $\sum_{i\in\N}   f_i \ge  2\theta + 2\delta$. 
Let $K_1$ be such that $\sum_{i=1}^{K_1} f_i > 2\theta + \delta$. 
Let $K_2 \ge K_1$ be such that $f_k < \frac{\delta}{K_1}$ for all $k> K_2$. 
Then for $k\ge  K_2$ we have 
\begin{align}
\label{eqn:Phi_cP_le_varphi_finite}
\Phi_\cP(\nu) & \le \varphi_k(f_1,\dots,f_k) + \delta, \\
\label{eqn:condition_for_k_star_finite}
\sum_{i=1}^k f_i - k f_{k+1} 
& = \sum_{i=1}^k (f_i - f_{k+1}) 
\ge \sum_{i=1}^{K_1} (f_i - \frac{\delta}{K_1}) 
\ge 2\theta + \delta - \delta = 2\theta.
\end{align}
By \eqref{eqn:Phi_cP_le_varphi_finite} it follows that (as the above can be done for any $\delta>0$), for $\varphi_k$ as in \eqref{eqn:varphi_k_def},
\begin{align*}
\Phi_\cP(\nu) \le  \supk \varphi_k(f_1,\dots,f_k). 
\end{align*}
\begin{calc}
By \eqref{eqn:Phi_cP_le_varphi_finite} it follows that for all $\delta>0$ there exists a $K_2>0$ such that $\Phi_\cP(\nu) \le \varphi_k(f_1,\dots,f_k) + \delta$ for all $k\ge K_2$ and thus $\Phi_\cP(\nu) \le \supk \varphi_k(f_1,\dots,f_k) +\delta$ for all $\delta >0$. 
\end{calc}
By \eqref{eqn:condition_for_k_star_finite} it follows that there exists a $\ell \in \N$ such that $K_\star(f_1,\dots,f_k) = \ell$ for all $k \ge K_2$, where $K_\star$ is as in \eqref{e:defkstar}. 
Therefore, by Proposition~\ref{p:AnaPhik}, we have $\varphi_{\ell}(f_1,\dots,f_k) =  \varphi_m(f_1,\dots,f_m)$ for all $m\ge K_2$ and thus, with $w_i$ as in \eqref{e:wmax}, for $w$ given by 
\begin{align*}
w(f,y) = \begin{cases}
w_i & \mbox{if } i\in\{1,\dots,\ell\} \mbox{ and }  (f,y) = (f_i,y_i), \\
0 & \mbox{otherwise},
\end{cases}
\end{align*}
we have for $\mu = w \cP$ that $\Phi_\cP(\mu) = \varphi_{\ell}(f_1,\dots, f_{\ell}) = \supk \varphi_k(f_1,\dots,f_k)$ and thus 
\begin{align*}
\Phi_\cP(\mu) \ge  \Phi_\cP(\nu), 
\qquad  \cD_\cP(\mu) <  \cD_\cP(\nu)
\qquad \mbox{and therefore}\qquad \Psi_\cP(\mu) > \Psi_\cP(\nu). 
\end{align*}

 \textbf{\underline{Case 2}} $\sum_{i\in\N} f_i \in [0,2\theta]$.
Let us first  introduce some objects. 
For $k\in\N$, let 
\begin{align*}
a_k 
:= \frac{\sum_{i=k+1}^\infty f_i}{2\theta}.
\end{align*}
Then $a_k \rightarrow 0 $ as $k\rightarrow \infty$. 
Let $R_1:=\inf\{s>0\colon  \supp \nu\subset (0,\infty)\times Q_{s}\}$ the smallest $r$ such that $\nu\in\cW_r$. 
Then $R_1\leq R$ since $\nu\in\cW_R$. 
Then there exists a function $\varphi \colon \N \rightarrow \N$  such that $y_{\varphi(n)}$ converges to $z=(z_1,\dots, z_d)\in\partial Q_{R_1}$ as $n\rightarrow \infty$. 
Assume, without loss of generality, that  $z_1=R_1$. 
Let $\beta>2$ and  $\epsilon = \epsilon_{R_1+1,3\theta}$ as in the condition \ref{item:good3_condition_new_points} of Definition~\ref{def:good_point_measure} on $\cP$. 
Pick $k\in \N$ such that $\theta a_k^2- 8 q a_k^\beta>0$ and $3 a_k^\beta\leq 1$ and  $a_k <\epsilon \wedge 1$. 
Then define 
\begin{align*}
a = a_k \qquad \mbox{and}\qquad
\bar z = (R_1+ 2a^\beta, z_2, \dots, z_d). 
\end{align*}
Observe that $\bar z \in Q_{R_1+1}$. 
By the assumption  \ref{item:good3_condition_new_points} on $\cP$ from Definition~\ref{def:good_point_measure} there exists a 
\begin{align*}
(\widetilde f, \widetilde y)\in [(3\theta a,\infty)\times B(\bar z,a^\beta)] \cap \supp \cP
\end{align*}
 We observe that $(\widetilde f,\widetilde y) \notin \supp \nu$ since $\widetilde y\notin Q_{R_1}$, but $\widetilde y \in Q_{R_1 + 3a^\beta} \subset Q_{R+1}$. 
 Since $\lim_{n\to \infty} y_{\phi(n)}=z$, 
 there exists a $(\hat f,\hat y)\in \supp \nu $ such that  $\hat y\in B(z, a^\beta)$ and therefore $(\widetilde f,  \widetilde y)\in (3\theta a,\infty) \times B(\hat y, 4 a^\beta)$. 
 \begin{figure}[H]
		\includegraphics[width=0.8\textwidth]{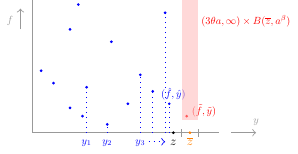}  
	\caption{Visualisation of $z$, $\overline z$ and $(\tilde f, \tilde y)$ (for $d=1$).}
	\label{fig:goodness}
\end{figure}

We recall Proposition~\ref{p:AnaPhik}, Case 2, which tells us that the unique maximizer of $\varphi_k$ as in \eqref{eqn:varphi_k_def}, is given by $(\frac{f_1}{2\theta}, \dots, \frac{f_k}{2\theta})$. 
Define $\mu = w \cP$ (i.e., $\radon{\mu}{\cP} = w$), where 
\begin{equation}\label{wcandidatedef}
w(f,y) = 
\begin{cases}
\frac{f_i}{2\theta} &\mbox{if } (f,y) = (f_i,y_i) \mbox{ for some } i \in \{1,\dots, k\}, \\
a 
&\mbox{if } (f,y) = (\widetilde f, \widetilde y), \\
0 & \mbox{otherwise}. 
\end{cases}
\end{equation}
Because
\begin{align*}
\sum_{i=1}^k \frac{f_i}{2\theta} + a =  \frac{\sum_{i=1}^\infty f_i}{2\theta}  \in [0,1],
\end{align*}
it is clear that $\mu\in\cW_{R+1}$ and that $\mu$ has only finite support. 
We are going to show that $\Psi_\cP(\mu)>\Psi_\cP(\nu)$. 

Observe that
\begin{align*}
\sum_{i=k+1}^\infty \frac{f_i^2}{4\theta} 
\le \frac{1}{4\theta} \left( \sum_{i=k+1}^\infty f_i \right)^2 
= \frac{1}{4\theta} \left(2\theta a  \right)^2 = \theta a^2. 
\end{align*}
Therefore, by using that $\sum_{i=1}^k \frac{f_i^2}{4\theta}= \sum_{i=1}^\infty \frac{f_i^2}{4\theta} 
- \sum_{i=k+1}^\infty \frac{f_i^2}{4\theta}$, we see that 
\begin{equation}\label{lowerphiw}
\Phi_\cP(\mu) 
 \ge 
\sum_{i=1}^{k}  w_i  f_i +  a \widetilde f - \theta \Big(\sum_{i=1}^k  w_i^2 + a^2 \Big) 
 = \sum_{i=1}^k \frac{f_i^2}{4\theta}  + a \widetilde f - \theta a^2
 \ge  \sum_{i=1}^\infty \frac{f_i^2}{4\theta}  + a \widetilde f . 
\end{equation}
By  Proposition~\ref{p:AnaPhik} it follows that 
\begin{equation}\label{lowerphiv}
\Phi_\cP(\nu) \le \sum_{i=1}^\infty \frac{f_i^2}{4\theta}.
\end{equation}
Moreover, we recall the definition of $\cD_{\cP}$ in \eqref{defdist} and see that
\begin{align}
\label{eqn:distance_w_compared_with_v}
\cD_{\cP}(\nu)\geq D_0(\{y_1,\dots,y_k,\hat y\}) \geq D_0(\{y_1,\dots,y_k,\hat y,\widetilde y\})- 8 a^\beta,
\end{align}
where the last inequality holds true since any path starting from $0$ and visiting all points in $\{y_1,\dots,y_k,\hat y\}$ 
can be extended to visit $\widetilde y$ as well by traveling back and forth along the straight line linking $\hat y$ and $\widetilde y$ (which are at most $4a^\beta$ apart from each other). Since $\supp_{\R^d} \mu = \{y_1,\dots,y_k,\widetilde y\}$, 
\begin{align*}
D_0(\{y_1,\dots,y_k,\hat y,\widetilde y\})\geq D_0(\{y_1,\dots,y_k,\widetilde y\}) =\cD_{\cP}(\mu).
\end{align*} 
Thus, we deduce from 
\eqref{eqn:distance_w_compared_with_v} that 
\begin{equation}\label{compdist}
\cD_\cP(\mu)\le \cD_\cP(\nu) + 8a^\beta,
\end{equation}
and therefore, using \eqref{lowerphiw}, \eqref{lowerphiv} and \eqref{compdist} in combination with the fact that
$\widetilde f\geq 3\theta a$ we obtain 
\begin{equation}\label{psiwlarpsiv}
\Psi_\cP(\mu) - \Psi_\cP(\nu)\ge - \sum_{i=k+1}^\infty \frac{f_i^2}{4\theta} + a \widetilde f - \theta a^2 -8 q a^\beta\ge a \widetilde f - 2\theta a^2 - 8q a^\beta  > \theta a^2 - 8q a^\beta >0. 
\end{equation}

\ref{item:max_are_prob_and_sum_f_i_larger_2theta}
Suppose that $\mu = \sum_{i=1}^k w_i \delta_{(f_i,y_i)}$. 
If $\mu$ is a maximizer and $\sum_{i=1}^k f_i \ge  2\theta$, then $\sum_{i=1}^k w_i =1$ because of Proposition~\ref{p:AnaPhik} (as, like in Case 1, $\Phi_\cP(\mu) = \varphi_k(f_1,\dots,f_k)$), so that $\mu$ is a probability measure. 

If $\sum_{i=1}^k f_i < 2\theta$ and $\mu$ maximizes $\Phi_\cP$ (observe $\Phi_\cP$, not $\Psi_\cP$), then by  Proposition~\ref{p:AnaPhik}, $\sum_{i=1}^k w_i <1$, i.e.,  $\mu$ is not a probability measure. 
Moreover, like in Case 2 above, one can show that $\mu$ is not a maximizer of $\Psi_\cP$: 
Indeed, one chooses an $a\in (0,\epsilon \wedge 1)$ with $a \le 1 - \sum_{i=1}^k w_i$,  $\theta a^2 - 8 q a^\beta>0$ and $3a^\beta \le 1$ and follows the same lines as in Case 2 to find a $(\tilde f, \tilde y)$ such that $\Psi_\cP(\mu + a \delta_{(\tilde f, \tilde y)}) > \Psi_\cP(\mu)$. 
\end{proof}


\subsection{Proof of Theorem \ref{thm:variational_formula_deterministic}}\label{exisuni}

\noindent In this section, we prove Theorem \ref{thm:variational_formula_deterministic} subject to 
Proposition~\ref{p:gamma_convergence_Psi}~\ref{item:gamma_convergence_Psi}
 (Gamma-convergence of $-\Psi$) and Proposition~\ref{p:exist_n_finiten_of_maxis} (finiteness of support of maximizers).

\begin{proof}[Proof of Theorem~\ref{thm:variational_formula_deterministic}]
The existence for \ref{item:var_f_maximizer} follows directly by Proposition~\ref{p:exist_n_finiten_of_maxis}, whereas the uniqueness is imposed by the fact that $\cP$ is assumed to be good. 
\ref{item:var_f_stability} follows by the fact that $\Psi_\cP$ is upper-semicontinuous and $\cW$ is sequentially compact (see the beginning of the proof of Proposition~\ref{p:exist_n_finiten_of_maxis}), so that, in particular, $O^c$ is sequentially compact for any open set $O\subset \cW$. Therefore there exists a maximizer of $\Psi_\cP$ on $O^c$, which by the uniqueness cannot be equal to the maximizer over $\cW$, therefore proving the desired inequality. 
\ref{item:var_f_continuity} follows from Theorem~\ref{theorem:gamma_convergence}~\ref{item:gamma_conv_optimizers} and  Proposition~\ref{p:gamma_convergence_Psi}~\ref{item:gamma_convergence_Psi}. 

Let us now prove \ref{item:support_more_than_one}. 
The idea is that the points in $G$ are all worth visiting because of their large energy values, but at the same time they are not distinct enough so as to give preference to only a couple of them. 
Having no points in $E^1,E^2$ and $E^3$ contributes to make points outside of $G$ not worth visiting, because either their energy values are too low or their distance too large. 

Fix $\cP \in \cM_\rp^\circ$ with $\cP(G)=k$ and $\cP(E^i) = 0$, $i=1, 2, 3$.
Denote by $(f_1,y_1),\dots,(f_k,y_k)$ the $k$ points of $\cP$ in $G$, with $f_1\ge f_2 \ge \cdots \geq f_k$. 
Take $w_i$ as in \eqref{e:wmax} of Proposition~\ref{p:AnaPhik} and let $\nu^*:= \sum_{i=1}^k w_i \delta_{(f_i, y_i)}$.
Note that, since $f_i \geq L > 2\theta$ for $1 \leq i \leq k$, the relevant formulas from Proposition~\ref{p:AnaPhik}
will (mostly) be \eqref{e:wmax} and \eqref{defphisupt}.
We divide the proof into the following steps:
\begin{enumerate}[label={\normalfont(Step \arabic*)}, labelsep=0.8em, leftmargin=*] 
\item $w_i > 0$ for all $i \in \{1,\dots,k\}$. 
\item If $\nu \in \cW$ and $\supp \nu \subset G$ then 
$\Psi_\cP(\nu) \leq \Psi_\cP(\nu^*)$;
\item If $\nu \in \cW$ and $\supp \nu \not \subset G$ then 
$\Psi_\cP(\nu) \leq \Psi_\cP(\nu^*)$.
\end{enumerate}
Steps~2--3 together with \ref{item:var_f_maximizer} will then show that $\mu^*=\nu^*$, and this together with Step~1 implies \ref{item:support_more_than_one}.

\smallskip

\noindent
\underline{Step 1}
By Proposition~\ref{p:AnaPhik} it suffices to check that $k = K_\star(f_1,\dots,f_k)$,  where $K_\star$ is as in \eqref{e:defkstar}. 
This follows as for any $j\in \{1,\dots,k-1\}$ we have 
\begin{align}
\label{eqn:estimate_of_sum_f_i_min_2theta}
\sum_{i=1}^j f_i - 2\theta \le j (L+\frac{2\theta}{k}) - 2\theta = jL - \frac{(k-j)2\theta}{k} \le j f_{j+1} - \frac{2\theta}{k} < j f_{j+1} .
\end{align}

\noindent
\underline{Step 2}
We can assume that $\nu \ll \cP$. 
We will first show the statement for $\nu \in \cW$ which are nonzero. 
Observe that by \eqref{eqn:max_varphi_k_is_max_Phi}, for any $\nu \in \cW$ with $\emptyset \ne \supp \nu \subset G$, 
\begin{align*}
\Psi_\cP(\nu) = \Phi_\cP(\nu) - \cD_\cP(\nu) \le \varphi_{|J|}( (f_j)_{j\in J} ) - q D_0 ( (y_j)_{j\in J}) ,
\end{align*}
where $J\subset \{1,\dots,k\}$ is such that $\supp \nu = \{ (f_j,y_j) : j\in J\}$. 
Therefore  it suffices to show 
that for all $J\subset \{1,\dots,k\}$, $J\ne \emptyset$, one has 
\begin{align*}
\varphi_{|J|} ( (f_j)_{j\in J} - q D_0 ( (y_j)_{j\in J} ) \le \Psi_\cP(\nu^*) = \varphi_k(f_1,\dots,f_k) - q D_0(y_1,\dots, y_k). 
\end{align*}
Of course for $k=1$ the above is clear. 
Suppose that $k\ge 2$. Let $J\subset \{1,\dots,k\}$ with $m:=|J|\le k-1$ and $\ell \in \{1,\dots,k\} \setminus J$. 
By an inductive argument on $m$, it suffices to show that
\begin{align}
\label{eqn:increase_by_including_one_point}
\varphi_m( (f_i)_{i\in J} ) - q D_0((y_i)_{i\in J}) 
\leq
\varphi_{m+1}( (f_i)_{i\in J\cup \{\ell\} } ) - q D_0((y_i)_{i\in J \cup \{\ell\} }). 
\end{align}
First of all, a  straightforward computation using \eqref{defphisupt} gives
\begin{calc}
setting $S= \sum_{i\in J} f_i - 2\theta$ and using that $\frac{1}{m} - \frac{1}{m+1} = \frac{1}{m(m+1)}$, 
\end{calc}
\begin{align}
\cand 
\varphi_{j+1}((f_i)_{i \in J\cup \{\ell\}}) - \varphi_j((f_i)_{i \in J}) 
\cnewline
\cand \begin{calc}
= 
\frac{1}{4\theta} \Big( \sum_{i\in J\cup \{\ell\}} f^2_i  - \frac{1}{ m+1 }\Big(\sum_{i\in J\cup \{\ell\}} f_i - 2\theta \Big)^2\Big)
- 
\frac{1}{4\theta} \Big(  \sum_{i\in J} f^2_i  - \frac{1}{ m }\Big(\sum_{i\in J} f_i - 2\theta \Big)^2 \Big)
\end{calc} \cnewline
\cand \begin{calc}
= 
\frac{1}{4\theta} \bigg[
f^2_\ell  - \frac{1}{ m+1 }\Big(f_\ell + \sum_{i\in J} f_i - 2\theta \Big)^2\Big)
+ \frac{1}{ m }\Big(\sum_{i\in J} f_i - 2\theta \Big)^2 \Big)
	\bigg]
\end{calc} \cnewline
\cand \begin{calc}
= 
\frac{1}{4\theta} \bigg[
f^2_\ell  - \frac{1}{ m+1 }\Big(f_\ell^2 + 2f_\ell S + S^2 \Big)
+ \frac{1}{ m }S^2 \Big)
	\bigg]
\end{calc} \cnewline
\cand \begin{calc}
= 
\frac{1}{4\theta} \bigg[
\frac{m}{m+1}
f^2_\ell  - \frac{2f_\ell S}{ m+1 }
+ \frac{1}{ m (m+1) } S^2 \Big)
	\bigg]
= 
\frac{1}{4\theta m (m+1)} \bigg[ m f_\ell - S\bigg]^2
\end{calc} \cnewline
& = \frac{1}{ 4 \theta m(m+1)} \left[ m f_\ell - \left(\sum_{i \in J}f_i - 2\theta\right)\right]^2 
\begin{calc}
\geq \frac{1}{4 \theta m(m+1)} \big(\frac{2 \theta}{k}\big)^2 
\end{calc}
\geq \frac{\theta}{k^4} = 4 q \varepsilon,
\end{align}
where in the last line we used that
\begin{align}
\label{eqn:sum_f_i_min_2theta_le_m_f_ell}
\sum_{i\in J} f_i - 2\theta \le m f_\ell - \frac{2\theta(k-m)}{k}, 
\end{align}
which follows similarly as the estimate in \eqref{eqn:estimate_of_sum_f_i_min_2theta}. 
From this and the observations 
\begin{align*}
D_0((y_i)_{i\in J \cup \{\ell\} }) - D_0((y_i)_{i\in J}) \le 4\epsilon, \quad 4q\epsilon - 4\epsilon>0, 
\end{align*}
we deduce \eqref{eqn:increase_by_including_one_point}. 
\begin{calc}
This basically follows from the observation that 
\begin{align*}
|y_1 - y_2| + |y_2 - y_3| \le |y_1 - y_3| + 2 |y_2 - y_3|\le |y_1 - y_3| + 2|y_2| + 2|y_3|, \\
|y_1 - y_2| + |y_2 - y_3| \le |y_1 - y_2| + |y_2| + |y_3. 
\end{align*}
\end{calc}
In order to finish this step, it suffices to observe that 
(because $L>2\theta + (q+1)\epsilon$), 
\begin{align}
\label{eqn:lower_bound_nu_supp_in_G}
\Psi_\cP( \nu^*) \ge 
 \varphi_1(f_1) - q D_0(y_1)
= f_1 - \theta -q|y_1| 
\ge L - \theta - q \epsilon > \epsilon >0. 
\end{align}

\noindent
\underline{Step 3}
By \ref{item:var_f_maximizer} it suffices to show this step for $\nu$ with $\supp \nu \subset \cP$ finite support. 
Let $(f,y) \in \supp \nu \setminus G$ be such that $f$ is maximal among such points.
Assume first that $f>\varepsilon$. In this case, our assumptions on $\cP$ imply that $|y|\geq f+ 3\theta$. 
If $f \geq L+ \frac{2\theta}{k}$, then $\Phi_{\cP}(\nu) \le f$, $\cD_\cP(\nu) \ge |y| \ge f + 3\theta \ge f$, and thus 
\begin{align*}
\Psi_\cP(\nu) = \Phi_\cP(\nu) - q \cD_\cP(\nu) 
\le f - q f <0 < \Psi_\cP(\nu^*). 
\end{align*}
If instead $\varepsilon < f \leq L+2\theta/k$, then $\Phi_\cP(\nu) \le L+\frac{2\theta}{k}$ and $|y| \ge f + 3 \theta > \varepsilon+\theta + \frac{2\theta}{k}$, so (remember \eqref{eqn:lower_bound_nu_supp_in_G})
\begin{align*}
\Psi_\cP(\nu) \leq L+\frac{2\theta}{k} - q |y|
\le L -(q-1)\frac{2\theta}{k} -q \theta - q\varepsilon < L -\theta-q\varepsilon \leq \Psi_\cP(\nu^*). 
\end{align*}
Assume now that $f \leq \varepsilon$. If $\supp \nu \cap G = \emptyset$ then
$\Psi_\cP(\nu) \leq \varepsilon  \le \Psi_\cP(\nu^*)$ (because of \eqref{eqn:lower_bound_nu_supp_in_G} again).
Lastly, suppose $\supp \nu \cap G = \{(f_i,y_i)_{i \in J}\} \neq \emptyset$ where $J \subset \{1,\ldots, k\}$ with $|J|=m \geq 1$. 
Let $N \in \N$ be the number of points in $\supp \nu \setminus G$. 
Denote by $(f_{k+1},y_{k+1}), \dots, (f_{k+N},y_{k+N})$ the points in $\supp \nu \setminus G$ with $f= f_{k+1} \geq f_{k+2} \geq \cdots \geq f_{k+N}$.
Observe that 
\begin{align*}
m f \le m \epsilon < mL -2\theta \le \sum_{j\in J} f_j - 2\theta, 
\end{align*} 
so that for $I = \{k+1,\dots,k+N\}$, $K_\star( (f_i)_{i\in J \cup I} ) = K_\star( (f_i)_{i \in J} )$.
The latter equals $ |J|$ due to \eqref{eqn:sum_f_i_min_2theta_le_m_f_ell} and thus $\Phi_\cP(\nu) \le \varphi_{|J\cup I|}( (f_i)_{i\in J \cup I} ) \le \varphi_m( (f_i)_{i\in J}) \le  \Phi_\cP(\nu^*)$ and so $\Psi_\cP(\nu) \le \Psi_\cP(\nu^*)$. 
\end{proof}

\section{Goodness of \texorpdfstring{$\Pi$}{} and of \texorpdfstring{$\Pi_t$}{} }
\label{subsec:optimizers_with_finite_support}

\noindent  In order that we can apply Theorem \ref{thm:variational_formula_deterministic} to the rescaled process $\Pi_t$ (for all $t>0$)  defined in \eqref{Pitdef} and to the PPP $\Pi$ defined in \eqref{PPPdef}, we show in this section that they are good  in the sense of Definition~\ref{def:good_point_measure} for any $\alpha\in(2d,\infty)$. 
Moreover, we prove Lemma~\ref{l:Pi_satisfies_many_points} for $\alpha \in (d,\infty)$. 
That is, we prove Lemma~\ref{l:goodness}, see Lemma~\ref{l:part_goodness} and Lemma~\ref{l:at_most_one_finite_max}. 
That both $\Pi$ and $\Pi_t$ can be viewed as elements of $\cM_\rp^\circ$, has been shown in Lemma~\ref{l:Pi_and_Pi_t_as_in_cM_p}.

\begin{lemma}
\label{l:prob_Pi_on_halfopen_times_ball}
Let $s,r>0$ and $x\in \R^d$. 
Let $V_d\in (0,\infty)$ be the volume of the unit ball in $\R^d$. 
Then 
\begin{align*}
\Prob\Big( \Pi\big( [s,\infty) \times B(x,r) \big) =0 \Big) = \ee^{- V_d s^{-\alpha} r^d}. 
\end{align*}
\end{lemma}
\begin{proof}
The random variable  $\Pi([s,\infty)\times B(x,r))$ is Poisson distributed with parameter 
\begin{equation}
\int_{[s,\infty)\times B(x,r)}\alpha y^{-(1+\alpha)} \dd y \otimes \dd z
=\int_{[s,\infty)}\frac{\alpha}{y^{1+\alpha}} \dd y \ V_d r^d = V_d s^{-\alpha} r^d.
\end{equation}
\end{proof}

\begin{lemma}
\label{l:part_goodness}
\ 
\begin{enumerate}
\item 
\label{item:Pi_t_satisfies_part_goodness}
 For $t\in (0,\infty)$,  with probability one, $\Pi_t$ satisfies \ref{itemgoodfinitelymanypoints} of Definition \ref{def:good_point_measure}. 
\item 
\label{item:Pi_satisfies_part_goodness}
 If $\alpha \in (2d,\infty)$, then  with probability one, $\Pi$ satisfies \ref{item:good3_condition_new_points} of Definition \ref{def:good_point_measure}. 
\end{enumerate}
\end{lemma}
\begin{proof}
\ref{item:Pi_t_satisfies_part_goodness}
By construction $\Pi_t$ satisfies \ref{itemgoodfinitelymanypoints} since $\Pi_t\big( (0,\infty)\times Q_R\big)$ is simply the cardinality of
$\{x\in \Z^d\colon |x|\leq R r_t\}$, which is $R^d r_t^d$. 

\ref{item:Pi_satisfies_part_goodness}
Let $\beta \in (2,\frac{\alpha}{d})$. 
First observe that $Q_R$ can be covered by balls $B(z,\frac{1}{k})$ with $z\in \frac1k \Z^d$ and $z\in Q_{R+1}$, i.e., 
\begin{align*}
Q_R \subset \bigcup_{z\in (\frac1k \Z^d) \cap Q_{R+1}} B(z, \tfrac{1}{k}). 
\end{align*}
Then, observe that therefore, with probability one $\Pi$ satisfies \ref{item:good3_condition_new_points} of Definition \ref{def:good_point_measure} if 
\begin{calc}
(replace $\epsilon$ by $k^{-\frac{1}{\beta}}$)
\end{calc}
\begin{align*}
\Prob 
\bigg( 
\bigcup_{R,C \in (0,\infty) \cap \Q} 
\bigcap_{N\in\N} 
\bigcup_{k \ge N} 
\bigcup_{z\in (\frac1k \Z^d) \cap Q_{R+1} }
\, \Big\{ \Pi \cap \big[ [C k^{-\frac{1}{\beta}}, \infty) \times B(z,\tfrac1k) \big] = \emptyset \Big\} \bigg) 
=0. 
\end{align*}
By the Borel--Cantelli Lemma, the above holds if we can show that for any $R,C \in (0,\infty)$, 
\begin{align*}
\sumk \Prob 
\bigg( 
\bigcup_{z\in (\frac1k \Z^d) \cap  Q_{R+1} }
\, \Big\{ \supp \Pi \cap \big[ [C k^{-\frac{1}{\beta}}, \infty) \times B(z,\tfrac1k) \big] = \emptyset \Big\} \bigg) <\infty. 
\end{align*}
Let $R,C>0$. We may assume $R>2$. 
By estimating the probability of the union by the sum of the probabilities, and observing that $\# (\frac1k \Z^d) \cap Q_{R+1} \le (2 k (R+1) +1)^d \le (5 k R)^d$, \begin{calc}
(indeed, use that 
$R+1 \le 2R$ and $4k R +1 \le 5k R$)
\end{calc}
by Lemma~\ref{l:prob_Pi_on_halfopen_times_ball} (with $s= Ck^{-\frac{1}{\beta}}$, $r= \frac1k$, so that $s^{-\alpha} r^d = C^{-\alpha} k^{\frac{\alpha}{\beta} -d}$), 
\begin{align*}
\sumk \Prob 
\bigg( 
\bigcup_{z\in (\frac1k \Z^d) \cap  Q_{R+1} }
\, \Big\{ \supp \Pi \cap \big[ [C k^{-\frac{1}{\beta}}, \infty) \times B(z,\tfrac1k) \big] = \emptyset \Big\} \bigg)
\le \sumk (5kR)^d  \ee^{- V_d C^{-\alpha} k^{\frac{\alpha}{\beta} -d} }. 
\end{align*}
Because $\beta< \frac{\alpha}{d}$, we have $\frac{\alpha}{\beta}>d$ and therefore the above sum is finite. 
\end{proof}

\begin{lemma}
\label{l:formula_Phi_nu_if_finite_optimizer}
Let $\cP \in \cM_\rp^\circ$. Suppose $\nu \in \fF(\cP)$, $\nu \ne 0$ and  $\Psi_\Pi(\nu) = \sup_{\mu \in \fF(\Pi)} \Psi_\cP(\mu)$. 
Let $k\in \N$, $(f_1,y_1),\dots,(f_k,y_k) \in \cP$ be such that $\supp \nu = \{(f_1,y_1),\dots,(f_k,y_k)\}$. 
Then $\Phi_\cP(\nu) = \widetilde \varphi_k(f_1,\dots,f_k)$, where 
\begin{align}
\label{eqn:def_widetilde_varphi}
\widetilde \varphi_k(f_1,\dots,f_k) = \frac{1}{4\theta} \Big(  \sum_{i=1}^{k} f^2_i  - \frac{1}{ k}\Big(\sum_{i=1}^{k}f_i - 2\theta \Big)^2 \Big) . 
\end{align}
\end{lemma}
\begin{proof}
Observe that (for $\varphi_k$ as in \eqref{eqn:varphi_k_def})
\begin{align*}
\Phi_\cP(\nu) = \varphi_k( f_1,\dots,f_k), \qquad 
\cD_\cP(\nu) = D_0 (y_1,\dots,y_k). 
\end{align*}
By the definition of the support, we have $\nu = \sum_{i=1}^k w_i \delta_{(f_i,y_i)}$ for some $w_1,\dots,w_k \in (0,1]$. 
By Lemma~\ref{l:optimizing_over_w} and~\ref{item:zero_is_better} we may assume that 
$k \min\{f_1,\dots,f_k\} + 2\theta - \sum_{i=1}^k f_i  >0$ (otherwise $w_i=0$ for some $i$). 
Therefore, by ~\ref{item:positive_w_i_with_sum_1_by_extra_condition} of that lemma, it follows that $\Phi_\cP(\nu) = \widetilde \varphi_k(f_1,\dots,f_k)$ (see also \eqref{e:wmax} and \eqref{defphisupt}). 
\end{proof}

\begin{lemma}
\label{l:at_most_one_finite_max}
Let $\alpha \in (0,\infty)$ and $t\in (0,\infty)$.  Recall the definition of $\fF(\cP)$ in \eqref{e:def_fF}. 
\begin{enumerate}
\item 
\label{item:Pi_at_most_one_finite_max}
With probability one, 
$\Psi_\Pi$ possesses at most one maximizer in $\fF(\Pi)$.
\item 
\label{item:Pi_t_at_most_one_finite_max}
With probability one, 
$\Psi_{\Pi_t}$ possesses at most one maximizer in $\fF(\Pi_t)$. 
Moreover, for $L>0$ and $\Pi_t^{\ssup{L}} =\1_{[L^{-1},\infty) \times Q_L} \Pi_t $ (see also \eqref{eqn:ssup_L}), the function 
 $\Psi_{\Pi_t^{\ssup{L}}}$ possesses at most one maximizer in $\fF(\Pi_t)$.
\end{enumerate}
\end{lemma}
\begin{proof}
\ref{item:Pi_at_most_one_finite_max}
We show that the event that there exist $\mu_1,\mu_2 \in \fF(\Pi)$ with $\mu_1 \ne \mu_2$ and $ \Psi_{\Pi}(\mu_1)=\Psi_{\Pi}(\mu_2)=\sup_{\mu\in \fF(\Pi)} \Psi_\Pi(\mu) $, has probability zero. 
For this it suffices to show that $\Prob(\cN_L)=0$ for any $L>0$, where, with $S_L:= [L^{-1},\infty) \times Q_L$, 
\begin{align*}
\cN_L = \Big\{
\exists \mu_1,\mu_2\in \fF(\Pi) 
& \colon\, 
\mu_1\neq \mu_2 , 
\ \supp \mu_i \subset S_L, 
\Psi_{\Pi}(\mu_1)=\Psi_{\Pi}(\mu_2)=\sup_{\mu\in \fF(\Pi)} \Psi_\Pi(\mu) 
\Big\}. 
\end{align*}
Let $L>0$. 
We give an explicit almost sure description of $\Pi$ on $S_L$. 
Let us write $\Theta$ for the intensity measure of $\Pi$, i.e., $\Theta({\rm d}(f,y))=\alpha f^{-(1+\alpha)}\,{\rm d} f\otimes{\rm d}y$. 
Let $N$ be a Poisson distributed variable with parameter $m_L$, where $m_L = \Theta(S_L)$. 
Let $(F_j,Y_j)_{j\in\N}$ be i.i.d.\ random variables that are independent from $N$ and whose law is given by $\frac{1}{m_L} \1_{S_L} \Theta$. 
Then, 
\begin{align*}
\1_{S_L} \Pi \ \mbox{ is equal in distribution to } \  \sum_{j=1}^{N} \delta_{(F_j,Y_j)} . 
\end{align*}
Without loss of generality, we may assume $\1_{S_L} \Pi = \sum_{j=1}^{N} \delta_{(F_j,Y_j)}$. 
Then, by Lemma~\ref{l:formula_Phi_nu_if_finite_optimizer} $\cN_L$ is included in the event 
(we make abuse of notation and for $m=0$ we understand $ \widetilde \varphi_m(F_{j_1},\dots, F_{j_m})$ and $D_0( Y_{j_1},\dots,  Y_{j_m})$ to be equal to $0$)
\begin{calc}
\begin{align*}
\Big\{ \exists k,m\in \N_0 \
& \exists (f_1,y_1),\dots,(f_k,y_k), (g_1,z_1),\dots, (g_m,z_m) \in S_L\cap \supp \Pi \colon \\
& \{(f_1,y_1),\dots,(f_k,y_m)\} \ne \{(g_1,z_1),\dots, (g_m,z_m)\} , \\
& \widetilde \varphi_k(f_1,\dots,f_k) = \widetilde \varphi_m(g_1,\dots,g_m) + D_0(y_1,\dots,y_k) - D_0(z_1,\dots,z_m) \Big\},
\end{align*}
 therefore in the event
\end{calc}
\begin{align}
\label{eqn:event_equality_maxim}
\left\{
\begin{aligned}
&  \mbox{There exist } k,m\in \N_0, \mbox{ distinct } i_1,\dots,i_k,i_* \mbox{ and distinct } j_1,\dots,j_m \mbox{ in } \{1,\dots,N\} \\
& \quad \mbox{such that } i_* \notin \{j_1,\dots,j_m\} \mbox{ and } \\
& \quad 
 \widetilde \varphi_{k+1}(F_{i_1},\dots,F_{i_k},F_{i_*}) = \widetilde \varphi_m(F_{j_1},\dots, F_{j_m}) + D_0( Y_{i_1},\dots, Y_{i_k},Y_{i_*}) - D_0( Y_{j_1},\dots,  Y_{j_m}) 
 \end{aligned} \right\}. 
\end{align}
The above event is included in the one where we replace $\{1,\dots,N\}$ by $\N$. 
\begin{calc}
Therefore, it suffices to let $k,m\in \N_0$, take distinct $i_1,\dots,i_k,i_*$ and distinct $j_1,\dots,j_m$ in $\N$ such that $i_* \notin \{j_1,\dots,j_m\}$ and show that 
\begin{align*}
\Prob \Big(  \widetilde \varphi_{k+1}(F_{i_1},\dots,F_{i_k},F_{i_*}) = \widetilde \varphi_m(F_{j_1},\dots, F_{j_m}) + D_0( Y_{i_1},\dots, Y_{i_k},Y_{i_*}) - D_0( Y_{j_1},\dots,  Y_{j_m})  \Big)=0. 
\end{align*}
\end{calc}
Then, it follows that \eqref{eqn:event_equality_maxim} has probability zero by Lemma~\ref{l:prob_tilde_varphi_and_dist_zero}. 

\ref{item:Pi_t_at_most_one_finite_max} Follows similar as the above argument: Besides replacing $\Pi$ by $\Pi_t$, replace  $\Theta$ by 
the product measure of $\alpha f^{-(1+\alpha)} \1_{[r_t^{-d/\alpha},\infty)}(f) \dd f$ and $\sum_{z\in r_t^{-1} \Z^d} \delta_z$, and $N$ by $\#(Q_L \cap \Z^d)$. Then again, one can show that the event \eqref{eqn:event_equality_maxim} has zero probability by applying Lemma~\ref{l:prob_tilde_varphi_and_dist_zero}. From this, the ``moreover'' part immediately follows too. 
\end{proof}

\begin{lemma}
\label{l:prob_tilde_varphi_and_dist_zero}
Suppose that $F_1,F_2,\dots$ are i.i.d.\ random variables with values in $(0,\infty)$ whose law has a density with respect to the Lebesgue measure. 
Let $Y_1,Y_2,\dots$ be i.i.d.\ random variables with values in $\R^d$. 
Let $k,m\in \N_0$. 
Suppose that $i_1,\dots,i_k,i_*$ are distinct element of $\N$ and $j_1,\dots,j_m$ are distinct elements in $\N$ such that $i_* \notin \{j_1,\dots,j_m\}$. Then
\begin{align}
\label{eqn:probability_of_equality_eq_zero}
\Prob \Big(  \widetilde \varphi_{k+1}(F_{i_1},\dots,F_{i_k},F_{i_*}) = \widetilde \varphi_m(F_{j_1},\dots, F_{j_m}) + D_0( Y_{i_1},\dots, Y_{i_k},Y_{i_*}) - D_0( Y_{j_1},\dots,  Y_{j_m})  \Big)=0. 
\end{align}
\end{lemma}
\begin{proof}
We explain the following argument in more detail below. 
If we condition the above event in \eqref{eqn:probability_of_equality_eq_zero} on all variables except $F_{i_*}$, that is, on  $F_{i_1},\dots, F_{i_k}, F_{j_1},\dots, F_{j_m}$, $Y_{i_1},\dots,Y_{i_k}, Y_{j_1},\dots,Y_{j_m}$ and $Y_{i_*}$, then by the formula for $\widetilde \varphi_k$ \eqref{eqn:def_widetilde_varphi},  there exist $C_1,C_2,C_3\in \R$, $C_1 \ne 0$ or $C_2 \ne 0$ such that the event in the probability of \eqref{eqn:probability_of_equality_eq_zero} becomes
\begin{align*}
C_1 F_{i_*}^2 + C_2 F_{i_*} + C_3 =0. 
\end{align*}
The probability of such event is equal to zero as $F_{i_*}$ has a density with respect to the Lebesgue measure. 

Indeed, observe that $ \widetilde \varphi_1(f_{i_*}) 
\begin{calc}
= \frac{1}{4\theta} (4\theta f_{i_*} - 4 \theta^2) 
\end{calc}
= f_{i_*} -  \theta = A_0 f_{i_*}^2 + B_0 f_{i_*} - C_0$, for $A_0=0$, $B_0=1$ and $C_0 = -\theta$, and for $k\in \N$, 
\begin{align*}
A_k& = \frac{1}{4\theta} (1- \frac{4\theta^2}{k+1}), \qquad 
B_k = B_k(f_1,\dots,f_k) = \frac{2}{4\theta(k+1)}\Big( \sum_{i=1}^k f_i - 2\theta \Big), \\
C_k& = C_k(f_1,\dots,f_k) = \frac{1}{4\theta} \Big( \sum_{i=1}^k f_i^2 - \frac{1}{k+1} \Big(\sum_{i=1}^k f_i - 2\theta\Big)^2 \Big),
\end{align*}
that 
\begin{align*}
\widetilde \varphi_{k+1}(f_1,\dots,f_{k},f_{i_*}) 
= 
A_k f_{i_*}^2 + B_k f_{i_*} + C_k. 
\end{align*}
So that for $\tilde C_k = C_k - ( \widetilde \varphi_m( f_{j_1},\dots,  f_{j_m}) + D_0(  y_{i_1},\dots,  y_{i_k}, y_{i_*}) - D_0(  y_{j_1},\dots,   y_{j_m})  \Big))$, we have 
\begin{align}
& \Prob
\left( 
\begin{aligned}
 \widetilde \varphi_{k+1}(F_{i_1},\dots,F_{i_k},F_{i_*}) = \widetilde \varphi_m(F_{j_1},\dots, F_{j_m}) \\
+ D_0( Y_{i_1},\dots, Y_{i_k},Y_{i_*}) - D_0( Y_{j_1},\dots,  Y_{j_m})  
\end{aligned} \ 
 \bigg|  \overline{F} = \overline{f}, \overline{Y} = \overline{y} 
 \right)
 = \Prob( A_k F_{i_*}^2 + B_k F_{i_*} + \tilde C_k =0),
 \label{eqn:prob_parabolic_eq}
\end{align}
where 
\begin{align*}
\overline F
&= (F_{i_1}, \dots, F_{i_k}, F_{j_1}, \dots, F_{j_m}) , 
& & 
\overline{f} 
=(f_{i_1}, \dots, f_{i_k}, f_{j_1}, \dots, f_{j_m}), \\
\overline{Y} 
&=
( Y_{i_1}, \dots,  Y_{i_k},Y_{i_*},  Y_{j_1}, \dots,  Y_{j_m}), 
& &
\overline{y} 
=(y_{i_1}, \dots, y_{i_k},y_{i_*}, y_{j_1}, \dots, y_{j_m}). 
\end{align*}
As the law of $F_{i_*}$ has a density with respect to the Lebesgue measure, the right-hand side (and thus the left-hand side) of \eqref{eqn:prob_parabolic_eq} equals zero. 
\end{proof}

For the proof of Lemma~\ref{l:Pi_satisfies_many_points}, we use the following lemma. 

\begin{lemma}\label{p:TSPinPPP}
Let $\lambda \in (0,\infty)$. 
Let $\zeta$ be a PPP on $(0,1)^d$ with intensity $\lambda$. 
 If $k \leq (\lambda/4)^{1/d}$, then
 \begin{align*}
 \Prob\Big(\exists \text{ distinct } Z_1, \ldots, Z_k \in \zeta   \colon D_0(Z_1, \ldots, Z_k) < d \Big) 
 \geq 1- \exp \left( - \Big(\frac{\lambda}{4} \Big)^{\frac1d} \right). 
 \end{align*}
 \end{lemma}
 \begin{proof}
 Let $\omega$ be a PPP on $\R^d$ with intensity $\lambda$.
In an almost sure and inductive sense we define sequences $(R_i)_{i\in\N}$ in $(0,\infty)$ and $(Y_i)_{i\in\N}$ in $[0,\infty)^d$ by setting $R_0=0$ and $Y_0=0$ and (on the probability one set such that the following infima are finite)
\begin{align*}
 R_{i+1} := \inf \left\{r> 0\colon\, \omega( Y_i + (0,r]^d ) >0 \right\} \;, 
\end{align*}
and by letting $Y_{i+1}\in [0,\infty)^d$ be the unique point in $\supp \omega \cap  \left( Y_i + (0,R_{i+1}]^d \right)$ (see also Figure~\ref{fig:choices_Y_and_R}). 
\begin{figure}[H]
		\includegraphics[width=0.3\textwidth]{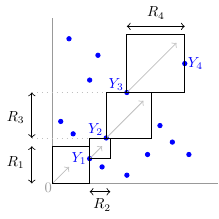}  
	\caption{Illustration of choosing $R_1,R_2,R_3,R_4$ and $Y_1,Y_2,Y_3,Y_4$.}
	\label{fig:choices_Y_and_R}
\end{figure}
Observe that 
\begin{align*}
 \Prob\Big(\exists \text{ distinct } Z_1, \ldots, Z_k \in \zeta   \colon D_0(Z_1, \ldots, Z_k) < d \Big) 
 \ge 
 \Prob\Big( 
 \{Y_1,\dots,Y_k\} \subset (0,1)^d , \sum_{i=1}^k |Y_i - Y_{i-1} | \le d
 \Big) 
\end{align*}
Note that  $Y_i \in Q_{R_1 + \cdots + R_i}$,
 $|Y_i - Y_{i-1}| \leq d R_i$ and $(R_1 + \cdots + R_k)^d \leq k^{d-1} (R_1^d + \cdots + R_k^d)$.
 Thus
 \begin{align*}
 \Prob \left( \{Y_1,\dots,Y_k\} \not \subset (0,1)^d \text{ or } \sum_{i=1}^k |Y_i - Y_{i-1} | \geq d \right)
 \leq \Prob \left( \sum_{i=1}^k R_i \geq 1\right)
 \leq \Prob \left( \sum_{i=1}^k \lambda R_i^d \geq \frac{\lambda}{k^{d-1}} \right).
 \end{align*}
 On the other hand, the random variables $\lambda R_i^d$ are i.i.d.\ Exp($1$). 
 \begin{calc}
 Indeed, we have $\Prob(\lambda R_1^d < t) = \Prob (\omega ( (0, (\frac{t}{\lambda})^{\frac1d} )^d )=0) = \exp ( - \lambda \frac{t}{\lambda}) = \exp(-t)$ for all $t\ge 0$. 
 \end{calc}
 Using $\E[ \ee^{\frac12 (\lambda R_1^d-2)} ] = 2\ee^{-1} < 1$,
  $\lambda k^{1-d} - 2 k \ge \frac{\lambda}{2} k^{1-d}$, 
$\frac{\lambda}{4 } k^{1-d} \ge (\frac{\lambda}{4})^{\frac1d}$
  and the Markov inequality, we obtain
 \begin{align*}
 \Prob \left( \sum_{i=1}^k \lambda R^d_i \geq \lambda k^{1-d} \right) \leq 
 \Prob \left( \sum_{i=1}^k (\lambda R^d_i-2) \geq \frac{\lambda}{2} k^{1-d} \right)
 \leq \ee^{ -\frac{\lambda}{4 }k^{1-d} } \E[\ee^{\frac12(\lambda R_1^d-2)}]^k
 < \ee^{ -\frac{\lambda}{4 } k^{1-d} }
 \le \ee^{- (\frac{\lambda}{4})^{1/d} }. 
 \end{align*}
\end{proof}

\begin{proof}[Proof of Lemma~\ref{l:Pi_satisfies_many_points}]
Let $\delta>0$. 
For $n,k\in\N$ let 
\begin{align*}
\cQ_n & := \left[\tfrac{1}{n},\infty \right) \times Q_{\frac{\delta}{d}}, \qquad \lambda_n := (\tfrac{\delta}{d})^d n^\alpha, \\
\cE_{n,k} & := 
\Big\{\exists \text{ distinct } (f_1, y_1),\dots, (f_k,y_k) \in \supp \Pi \cap \cQ_n	 \text{ such that }D_0(y_1, \ldots, y_{k}) <  \delta \Big\}.
\end{align*}
Then $\Pi(\cQ_n)$ is Poisson distributed with parameter $\lambda_n$. 
For $n\in\N$ let $k_n := \lceil \lambda_n^{\frac1d}  \rceil$. 
Then on the event $\cE_{n,k_n}$ we have 
\begin{align*}
\sum_{i=1}^{k_n} f_i \geq \frac{k_n}{n} \ge \frac{\delta}{d} n^{\frac{\alpha}{d}-1},
\end{align*} 
which is larger than $2\theta$ for sufficiently large $n$. 

Hence, for such large $n$ we have $\Prob( \Pi \mbox{ satisfies } \eqref{eqn:many_points} ) \ge \Prob(\cE_{n,k_n})$ and 
so it suffices to show 
\begin{align}
\label{eqn:prob_E_n_k_n_to_1}
\Prob(\mathcal E_{n,k_n}) \xrightarrow{n\to \infty}  1. 
\end{align}
Note that the projection of $\1_{\cQ_n} \Pi $ onto $(0, \delta/d)^d$ is a PPP on $(0,\delta/d)^d$ with intensity $\lambda_n (\delta/d)^{-d}$ (because $\Pi(\cQ_n)$ is  Poisson distributed with parameter $\lambda_n$). 
Therefore, for a PPP $\zeta$ on $(0,1)^d$ with intensity $\lambda_n$, we have 
\begin{align*}
\Prob(\mathcal E_{n,k_n}) 
= \P\Big(\exists \text{ distinct } Z_1, \ldots, Z_{k_\delta} \in \zeta   \colon D_0(Z_1, \ldots, Z_{k_\delta}) < d \Big) . 
\end{align*}
Therefore, by applying Lemma~\ref{p:TSPinPPP}, we conclude \eqref{eqn:prob_E_n_k_n_to_1}. 
\end{proof}


\section{Proof of Theorem \ref{thm:variational_formula1}}
\label{prooofthvarfor}

\noindent In the present section we will prove Theorem \ref{thm:variational_formula1}
subject to Proposition \ref{lowerboundrand} and Proposition \ref{p:uppbound} below, whose proofs are postponed to Sections \ref{sec:LowBound} and \ref{sec:UppBound}, respectively.

In some sense, Proposition \ref{lowerboundrand} gives us the lower bound of Theorem \ref{thm:variational_formula1}~\ref{item:2d_part}~\ref{i:law_W_convergence} whereas Proposition \ref{p:uppbound} gives us the corresponding  upper bound, as well as Theorem~\ref{thm:variational_formula1}~\ref{item:2d_part}~\ref{i:law_W_convergence}.
 
Our strategy is the following. We first need to \lq compactify\rq\ the partition function, i.e., to show that the random walk in the partition function can be restricted to some large box with a diameter on the scale $r_t$. This is done in Proposition~\ref{p:uppbound}~\ref{item:compactification} in the sense of a convergence in distribution. Furthermore, we derive upper (in Proposition~\ref{p:uppbound}~\ref{item:upper_bound_compactified_Z_n}) and lower  (in Proposition~\ref{lowerboundrand}) bounds for the compactified partition function that lead to the right limit, the variational formula $\Xi$. Finally we need to upper bound the compactified partition function with $W_t$ outside a neighbourhood of the maximizer against something that has a strictly smaller exponential rate. Here the stability of the variational formula from Theorem~\ref{thm:variational_formula_deterministic}(c) will be crucial.
 
The upper and lower bounds for the compactified partition function are proved even in the almost-sure sense with respect to $\xi$, using the Skorohod embedding. That is, we do not  work with a fixed trajectory $t\mapsto Z_t^{\xi,\beta}$ for a given realization of $\xi$, but  with a sequence of realizations that are constructed jointly on one probability space. For this, we fix a sequence of times $(t_n)_{n\in\N}$. Since this construction is used several times in the paper, we state it in the following remark.

\begin{remark}[Skorohod embedding]\label{skohorod-embedding}
Let $(t_n)_{n\in \N}$ be a strictly increasing sequence in $(0,\infty)$ such that $ t_n \to \infty$. 
By Skorohod's representation theorem (see, e.g.,  \cite[Theorem 1.6.7]{Bi99}) and Lemma~\ref{l:P(t)toPi} we can define on the same probability space
a sequence $(\bPi_n)_{n\in\N}$ of point processes and a Poisson point process $\bPi$ on $(0,\infty)\times \R^d$ of intensity  $\alpha f^{-1-\alpha}  \dd f \otimes \dd y$ 
such that $\bPi_n$ is the same in distribution as $\Pi_{t_n}$ (see \eqref{Pitdef}) for every $n\in\N$, and $\bPi$ is the same in distribution as $\Pi$, and 
\begin{align*}
\bPi_n \rightarrow \bPi \quad \mbox{ almost surely in } \cM_\rp^\circ. 
\end{align*}
Without loss of generality we may assume that
\begin{equation}
\label{eqn:Pin}
\bPi_n = \sum_{z\in \Z^d} \delta_{\big(\frac{\xi_n(z)}{r_{t_n}^{d/\alpha}}, \frac{z}{r_{t_n}} \big)}, 
\end{equation}
for some random variables $\xi_n(z)$ which are the same in distribution as $\xi(z)$, for any $n\in\N$ and $z\in \Z^d$. 
Fix a metric $\fd$ on $\cW$ that is compatible with the vague topology and write $B(\nu,\delta) =\{ \mu \in \cW : \fd(\nu,\mu) <\delta\}$ for $\nu \in \cW$ and $\delta>0$. 
We introduce the following notation: 
\begin{equation}
\label{simpli}
\begin{aligned}
& \br_n = r_{t_n}, 
\qquad
\bgamma_n=\br_n \log t_n, \\
& \bZ_n=Z_{t_n}^{\xi_n,\beta_{t_n}},
\qquad \bH_n(X)=H_{t_n}^{\xi_n,\beta_{t_n}} (X), 
 \qquad \pmb{\P}_n = \P_{t_n}^{\ssup{\xi_n}}, 
\qquad \bW_n = W_{t_n}^{\xi_n, X}, 
\end{aligned}
\end{equation}
and for $R,\delta>0$ and $n\in\N$ 
\begin{align}
\label{eqn:def_bZ_n_R_+_-}
\bZ_n^{R,-}
&=\E\Big[\ee^{\bH_n(X)}\ \1\Big\{ \max_{s\in[0,t_n]}|X_s|\leq R \br_n\Big\}\Big]\qquad\mbox{and}\qquad \bZ_n^{R,+}=\bZ_n-\bZ_n^{R,-}, \\
\label{eqn:def_bZ_n_R_-_delta}
\bZ_n^{R,-,\delta}
& = \E\Big[\ee^{\bH_n(X)}\, \1\{\fd(\bW_n,\mu^*) \ge \delta \}\, \1\Big\{ \max_{s\in[0,t_n]}|X_s|\leq R \br_n\Big\} \Big]. 
\end{align}
Recall \eqref{Xidef}. For a good point measure $\cP$ (see Definition~\ref{def:good_point_measure}) on $(0,\infty)\times \R^d$ and $\nu^* \in \fF(\cP)$ the unique maximizer of $\Psi_\cP$ (the existence is shown in Theorem~\ref{thm:variational_formula_deterministic}) so that $\Xi(\cP) = \Psi_\cP(\nu^*)$, we define
\begin{align}
\label{eqn:def_Xi_delta}
\Xi^\delta(\cP) := \sup_{\nu \in \cW : \fd(\nu,\nu^*) \ge \delta } \Psi_\cP(\nu). 
\end{align}
For the probability measure on the space where the $\bPi$ and $\bPi_n$'s live, we make abuse of notation and write $\Prob$ and assume this does not lead to confusion. 
\hfill$\Diamond$ 
\end{remark}
  
We can now formulate the lower bound for the partition function. 

\begin{proposition}[Lower bound]\label{lowerboundrand} Fix $\alpha\in(d,\infty)$. Then, with $\Prob$-probability $1$, 
\begin{align}
\label{eqn:lower_bound_liminf}
\liminf_{n\to \infty} \frac{1}{\bgamma_n} \log \bZ_n\geq  \Xi(\bPi).
\end{align}
\end{proposition}

The proof of Proposition~\ref{lowerboundrand} is given in Section~\ref{sec:LowBound}. Now we formulate the appropriate upper bounds for both assertions of Theorem \ref{thm:variational_formula1}. What will be crucial for the proof is the following observation:

\begin{lemma}
\label{l:maximizer_under_PPP}
Let $\alpha \in (2d,\infty)$. 
There exists a
random variable $\mu^*$ with values in $\fF(\bPi)$, 
such that $\Prob$-almost surely $\mu^*$ is the unique element of $\cW$ such that 
\begin{equation}\label{defxi}
 \Psi_{\bPi}(\mu^*)=\Xi(\bPi).
\end{equation}
Moreover, almost surely 
\begin{align}
\label{eqn:Xi_delta_less_Xi}
\Xi^\delta(\bPi) < \Xi(\bPi). 
\end{align}
\end{lemma}
\begin{proof}
Because $\alpha>2d$, $\bPi$ is almost surely  good by Lemma~\ref{l:goodness}. 
Therefore by Theorem~\ref{thm:variational_formula_deterministic}~\ref{item:var_f_maximizer} such $\mu^*$ exists (that it is random in the sense that it is a measurable function on the probability space, is not completely trivial; see Appendix~\ref{sec:measurability_of_maximizer}) and is unique almost surely and by Theorem~\ref{thm:variational_formula_deterministic}~\ref{item:var_f_stability}, \eqref{eqn:Xi_delta_less_Xi} holds. 
\end{proof}

\begin{proposition}[Upper bounds]\label{p:uppbound} Fix $\alpha\in (d,\infty)$. 
\begin{enumerate}
\item 
\label{item:upper_bound_compactified_Z_n}
{\em Upper bound for compactified $\bZ_n$:} 
\begin{align}
\label{eqn:limsup_Z_R_min}
& \Prob \Big[ \, 
\limsup_{R\to\infty}\limsup_{n\to\infty}\frac{1}{\bgamma_n}\log \bZ_n^{R,-} \leq \Xi(\bPi) \, \Big] =1.
\end{align}


\item 
\label{item:upper_bound_compactified_Z_n_away_max}
{\em Upper bound for compactified $\bZ_n$ away from maximizer:} 
Assume that $\alpha>2d$. Then 
\begin{align}
\label{eqn:limsup_Z_R_min_delta}
& \Prob \Big[ \, 
\limsup_{R\to\infty}\limsup_{n\to\infty}\frac{1}{\bgamma_n}\log \bZ_n^{R,-,\delta} \le \Xi^\delta(\bPi)
\, \Big] = 1 \forqq{\delta>0}
\end{align}


\item 
\label{item:compactification}
{\em Compactification:} 
\begin{align}
\label{eqn:liminf_Z_R_plus}
& \lim_{R\to \infty} \liminf_{n\to\infty}\  \Prob\bigg [\frac 1{\bgamma_n}\log \bZ_n^{R,+} \leq -A\bigg]=1 \forqq{A>0}. 
\end{align}
\end{enumerate}
\end{proposition}

The proof of Proposition~\ref{p:uppbound} is given in Section~\ref{sec:UppBound}. 
We extract the following lemma from Proposition~\ref{lowerboundrand} and Proposition~\ref{p:uppbound}~\ref{item:compactification} which will be used for the proof of Theorem~\ref{thm:variational_formula1}. 

\begin{lemma}
\label{l:prob_convergence_frac_Z_n_plus_min}
For all $\epsilon, \eta >0$ there exist an $R>0$ and an $N\in\N$ such that for all $n\ge N$ 
\begin{align*}
\Prob \left[ \frac{\bZ_n^{R,+}}{\bZ_n^{R,-}} < \epsilon \right] 
\ge 1- \eta. 
\end{align*}
\end{lemma}
\begin{proof}
First, we bound $\bZ_n^{R,-}$ from below by restricting the expectation to the trajectory that remains at the origin up to time $t_n$ and obtain (because by for example \eqref{Hamil_Pi_t}, for $Y_s=0$ for all $s\in [0,t_n]$,  $\bH_n(Y) \ge - \theta \bgamma_n $)
\begin{calc}
Indeed, regarding \eqref{eqn:local_times}, $\ell^{\ssup{Y}}_{t_n} = t_n$ and thus 
\begin{align*}
\bH_n(Y) 
= \sum_{z\in\Z^d}\xi(z)\ell_{t_n}^{\ssup{Y}}(z)-\beta_{t_n} \sum_{z\in\Z^d}\ell_{t_n}^{\ssup{Y}}(z)^2
\ge - \beta_{t_n} t_n^2 
= \theta \bgamma_n,
\end{align*}
where the latter equality can be found in \eqref{r_tpropertis}. 
\end{calc}
\begin{equation}\label{borninfeasy}
\bZ_n^{R,-}\geq \P(\ell_{t_n}(0)=t_n) \ee^{-\theta \bgamma_n}=\ee^{-2d t_n-\theta \bgamma_n},\qquad n\in\N, R>0.
\end{equation}
Because $\bgamma_n = \br_n \log t_n = t_n^{1+q} (\log t_n)^{-q}$ (see \eqref{defr}), we have $\frac{t_n}{\bgamma_n} \rightarrow 0$ as $n\rightarrow \infty$. 
Let $\epsilon,\eta>0$. 
Then, by also using \eqref{eqn:liminf_Z_R_plus} with $A=3\theta$ there exists an $R>0$ and an $N\in\N$ such that for all $n\ge N$, 
\begin{align*}
\Prob \Big[ 
	\bZ_n^{R,+} \le \ee^{-3\theta \bgamma_n}
	\Big] \ge 1- \eta, \qquad 
	\bZ_n^{R,-} \ge \ee^{-2\theta \bgamma_n}, \qquad
	\ee^{-\theta \bgamma_n} <\epsilon
\end{align*} 
so that 
\begin{align*}
\Prob \left[ \frac{\bZ_n^{R,+}}{\bZ_n^{R,-}} < \epsilon \right] 
\ge 
\Prob \left[ \frac{\bZ_n^{R,+}}{\bZ_n^{R,-}} < \ee^{-\theta \bgamma_n} \right] 
\ge 1-\eta. 
\end{align*}
\end{proof}

Now we prove Theorem~\ref{thm:variational_formula1} subject to the above propositions and lemma. 

\begin{proof}[Proof of Theorem~\ref{thm:variational_formula1}~\ref{item:partition_convergence}] 
Let $\cD$ be the continuity set of the distribution function for $\Xi(\bPi)$, i.e., the subset of $\R$ containing every continuity point of  $x\mapsto \Prob(\Xi(\bPi)\leq x)$. 
We will prove 
\begin{align}
\label{eqn:limitfreeener_with_Xi_Pi}
 \frac{1}{\bgamma_n} \log \bZ_n
\ \overset{t\to\infty}{\Longrightarrow} \ \Xi(\bPi),  
\end{align}
by showing the
following two inequalities:
\begin{eqnarray}
\limsup_{n\to \infty} \Prob\left[  \frac{1}{\bgamma_n} \log \bZ_n \leq h\right]&\leq& \Prob\left (\Xi(\bPi)\leq h\right ), \qquad h\in \R,\label{convcontpartfun}\\
\liminf_{n\to \infty} \Prob\left[  \frac{1}{\bgamma_n} \log \bZ_n \leq h \right]&\geq& \Prob \left (\Xi(\bPi)\leq h\right ), \qquad h\in \cD.\label{convcontpartfunlower}
\end{eqnarray}
The proof of \eqref{convcontpartfunlower} is more involved. Therefore we focus on \eqref{convcontpartfunlower}, because \eqref{convcontpartfun} follows in a similar fashion from Proposition~\ref{lowerboundrand}. 
\begin{calc}
Indeed, by Proposition~\ref{lowerboundrand} it follows that for all $\eta>0$ there exists an $N\in\N$ such that for $\cB_N = \{\forall n\ge N : \frac{1}{\bgamma_n} \log \bZ_n \ge \Xi(\Pi)\}$, 
\begin{align*}
\Prob [ 
	\cB_N
	] \ge 1- \eta. 
\end{align*}
For $\eta>0$ and $N$ as such, we have
\begin{align*}
\sup_{n\ge N} \Prob \Big[
	\frac{1}{\bgamma_n} \log \bZ_n \le h
	\Big]
\le \sup_{n\ge N} \Prob \Big[
	\frac{1}{\bgamma_n} \log \bZ_n \le h, \cB_N
	\Big]
+ \Prob [
	\cB_N^c 
	] 
\le \Prob [
	\Xi(\bPi) \le h 
	]
+ \eta. 
\end{align*}
Therefore $\limsupn \Prob \Big[
	\frac{1}{\bgamma_n} \log \bZ_n \le h, \cB_N
	\Big] \le \Prob \Big[
	\Xi(\bPi) \le h 
	\Big]
+ \eta$ for any $\eta>0$. 
\end{calc}
Observe that for any $R>0$, because $\bZ_n = \bZ_n^{R,-} + \bZ_n^{R,+}$ (see \eqref{eqn:def_bZ_n_R_+_-}), 
\begin{equation}\label{Znesti}
\frac{1}{\bgamma_n} \log \bZ_n
=\frac{1}{\bgamma_n} \log \bZ_n^{R,-}+\frac{1}{\bgamma_n} \log \Big(1+\frac{\bZ_n^{R,+}}{\bZ_n^{R,-}}\Big)
\leq \frac{1}{\bgamma_n} \log \bZ_n^{R,-}+ \frac{1}{\bgamma_n}\frac{\bZ_n^{R,+}}{\bZ_n^{R,-}}.
\end{equation}
Pick an $\eta>0$. 
Let $\cA_{n,R}:= \{ \bZ_n^{R,+} \le \bZ_n^{R,-} \}$. By Lemma~\ref{l:prob_convergence_frac_Z_n_plus_min} there exist an $R>0$ and an $N\in\N$ such that $\Prob(\cA_{n,R}) \ge 1- \eta$ for all $n\ge N$. 

Fix $h\in\mathcal D$ and pick $\epsilon>0$. As $\frac{1}{\bgamma_n} \le \epsilon$ for large $n$, we have for sufficiently large $n$ that 
\begin{equation}\label{mastereq}
\begin{aligned}
\Prob\left[  \frac{1}{\bgamma_n} \log \bZ_n \leq h \right]&\geq \Prob\left[\Big \{ \frac{1}{\bgamma_n} \log \bZ_n \leq h\Big\} \cap \cA_{n,R}   \right]\\
&\geq \Prob\left[ \Big \{ \frac{1}{\bgamma_n} \log \bZ_n^{R,-} \leq h-\epsilon \Big\} \cap \cA_{n,R}   \right]\\
 &\geq  \Prob\left[  \frac{1}{\bgamma_n} \log \bZ_n^{R,-} \leq h-\epsilon  \right] -\Prob( \cA_{n,R}^{\rm c}) \\
&\geq  \Prob\left[  \frac{1}{\bgamma_n} \log \bZ_n^{R,-} \leq h-\epsilon  \right] -\eta.
\end{aligned}
\end{equation}
At this stage,  we use Proposition~\ref{p:uppbound}~\ref{item:upper_bound_compactified_Z_n}, i.e., we use \eqref{eqn:limsup_Z_R_min}, 
to infer that (possibly by choosing $R$ larger)
\begin{align}
\label{partfunrest}
\Prob\Big( \frac{1}{\bgamma_n} \log \bZ_n^{R,-}\leq \Xi(\bPi)+\epsilon\Big) = 1 \forqq{\mbox{ for large } n}. 
\end{align}
\begin{calc}
Indeed, if we have $\Prob( \limsupn Y_n \le A )=1$, then $\Prob( \bigcap_{m\in\N} \bigcup_{N\in\N} \bigcap_{n\ge N} \{Y_n \le A+\frac1m\} ) =1$, so that for all $\epsilon>0$ one has $\Prob( \bigcup_{N\in\N} \bigcap_{n\ge N} \{Y_n \le A+\epsilon\} ) =1$, i.e., by monotonicity, $\limN \Prob(  \{Y_n \le A+\epsilon\} ) \ge  \limN \Prob(  \bigcap_{n\ge N} \{Y_n \le A+\epsilon\} ) = \Prob( \bigcup_{N\in\N} \bigcap_{n\ge N} \{Y_n \le A+\epsilon\} ) =1$. 
\end{calc}
Combining \eqref{mastereq} and \eqref{partfunrest} gives 
\begin{equation}\label{finaleq}
\liminf_{n\to \infty}\Prob\left[  \frac{1}{\bgamma_n} \log \bZ_n \leq h \right]\geq \Prob\left( \Xi(\bPi)\leq h-2 \epsilon\right) -\eta \forqq{\mbox{for any } \eta,\epsilon>0}. 
\end{equation}
By letting $\eta$ and $\epsilon$ converge to zero and by using the continuity at $h$ of $x \mapsto \Prob\left( \Xi(\bPi)\leq x\right)$, this completes the proof of \eqref{convcontpartfunlower}.

From \eqref{eqn:limitfreeener_with_Xi_Pi} we deduce that $\frac{1}{r_t \log t} \log Z_t^{\ssup{\xi}} \Longrightarrow \Xi(\Pi)$ as we obtained the convergence along diverging sequences of $(t_n)_{n\in\N}$ in $(0,\infty)$.
We are left to show that $\Xi(\bPi)$ is almost surely in $[0,\infty)$. 
That $\Xi(\bPi) \ge 0$ follows directly by the fact that it is larger than $\Psi_\bPi$ evaluated in the zero measure. 
That it is finite follows by the fact that $H_t^{\ssup{\xi,\beta_t}} \le H_t^{\ssup{\xi,0}}$ and thus $Z_t^{\ssup{\xi}} = Z_t^{\ssup{\xi,\beta_t}} \le Z_t^{\ssup{\xi,0}}$. 
As the limit of $\frac{1}{r_t\log t} \log Z_t^{\ssup{\xi,0}}$ is almost surely finite (by e.g. (1.6) in \cite{KLMS09}), so is the limit of $\frac{1}{r_t \log t} \log Z_t^{\ssup{\xi}}$, which is $\Xi(\Pi)$. 
\end{proof}

\begin{proof}[Proof of Theorem~\ref{thm:variational_formula1}~\ref{item:2d_part}~\ref{i:law_W_convergence}]
Let $\delta>0$. 
First we show that $\pmb{\P}_n[ \fd(\bW_n,\mu^*) > \delta ]$ converges to zero in $\Prob$-probability, i.e., for all $\kappa>0$, 
\begin{align}
\label{eqn:convergence_in_prob_entangled}
 \Prob\Big[
 	\pmb{\P}_n[ \fd(\bW_n,\mu^*) > \delta ] > \kappa 
 \Big]  \xrightarrow{t\rightarrow \infty} 0. 
\end{align}
Observe that for any $R>0$ 
\begin{align*}
 \pmb{\P}_n[ \fd(\bW_n,\mu^*) > \delta ] 
\le \frac{\bZ_n^{R,-,\delta} + \bZ_n^{R,+} }{\bZ_n} 
\le \frac{\bZ_n^{R,-,\delta} }{\bZ_n} 
+ \frac{\bZ_n^{R,+} }{\bZ_n^{R,-}}.  
\end{align*}
Let $\kappa>0$. 
By Lemma~\ref{l:prob_convergence_frac_Z_n_plus_min} 
it is sufficient to show that there exists an $R>0$ such that 
$\Prob[ \frac{\bZ_n^{R,-,\delta} }{\bZ_n} \le \kappa] \arrown 1$. 
Let $\epsilon>0$. 
By \eqref{eqn:Xi_delta_less_Xi} there exists an  $m\in\N$ such that $\ee^{-\frac{1}{3m}} <\kappa$, and 
\begin{align*}
\Prob [ \cB_{m,\delta} ]
 \ge 1- \epsilon, \quad \mbox{ where } \cB_{m,\delta} =\Big\{ \Xi^\delta(\bPi) - \Xi(\bPi) < - \frac1m \Big\}. 
\end{align*}
By Proposition~\ref{p:uppbound}~\ref{item:upper_bound_compactified_Z_n_away_max} and Proposition~\ref{lowerboundrand} there exists an $R>0$ and an $N\in\N$ such that for all $n\ge N$ 
\begin{align*}
\Prob \Big[\frac{1}{\bgamma_n} \log \bZ_n^{R,-,\delta} \le \Xi^\delta(\bPi) + \frac{1}{3m} \Big] \ge 1- \epsilon, 
\qquad 
\Prob \Big[\frac{1}{\bgamma_n} \log \bZ_n^{R} \ge \Xi(\bPi) - \frac{1}{3m} \Big] \ge 1- \epsilon, 
\end{align*}
so that 
\begin{align*}
1-3\epsilon \le 
\Prob 
	\Big[ 
	\frac{\bZ_n^{R,-,\delta} }{\bZ_n} \le \exp{\Big(- \frac{1}{3m}\Big)}
	\Big] 
	\le 
	\Prob 
	\Big[ 
	\frac{\bZ_n^{R,-,\delta} }{\bZ_n} \le \kappa
	\Big]. 
\end{align*}
From this we conclude \eqref{eqn:convergence_in_prob_entangled}. 
From the convergence in probability we deduce the existence of a strictly increasing $\varphi : \N \rightarrow \N$ such that $\pmb{\P}_{\varphi(n)}[ \fd(\bW_{\varphi(n)},\mu^*) > \delta ]  \rightarrow 0$ $\ \Prob$-almost surely. 
This implies $\Prob$-almost surely that $\bW_{\varphi(n)} \Longrightarrow \mu^*$ in $\cW$, more precisely, $\pmb{\E}_{\varphi(n)}[g(\bW_{\varphi(n)})] \rightarrow g(\mu^*)$ for any $g\in C_\rb(\cW)$. Therefore, 
\begin{align*}
\Expec\big[\pmb{\E}_{\varphi(n)}[g(\bW_{\varphi(n)})]\big]
\rightarrow \Expec[ g(\mu^*)] \forqq{g\in C_\rb(\cW)}. 
\end{align*}
Therefore, as for each sequence $(t_n)_{n\in\N}$ with $t_n \to \infty$ there exists a  strictly increasing $\varphi: \N \rightarrow \N$ such that 
\begin{align*}
\Expec\big[\ee^{\ssup{\xi}}_{t_{\varphi(n)}}[g(W_{t_{\varphi(n)}})]\big]
\rightarrow \Expec[ g(\mu^*)] \forqq{g\in C_\rb(\cW)}, 
\end{align*}
it follows that for any sequence $(t_n)_{n\in\N}$ with $t_n \to \infty$ 
\begin{align*}
\Expec\big[\ee^{\ssup{\xi}}_{t_n}[g(W_{t_n})]\big]
\rightarrow \Expec[ g(\mu^*)] \forqq{g\in C_\rb(\cW)}, 
\end{align*}
and therefore \eqref{eqn:weak_convergence_W_t_expressed}. 
\end{proof}

\begin{remark}
\label{remark:proof_more_general_conv}
The proof of the more general convergence in distribution 
$\cL^{\ssup{\xi}}_t \ \overset{t\to\infty}{\Longrightarrow} \ \delta_{\mu^*}$ as in \eqref{eqn:law_converges_to_delta_mu_star} of Remark~\ref{remark:more_general_statement} can be deduced from the first part of the proof of Theorem~\ref{thm:variational_formula1}~\ref{item:2d_part}~\ref{i:law_W_convergence} as follows. 

First we observe that by Portmanteau's theorem, for probability measures $\rho_1,\rho_2,\dots$ on $\cW$ and $\mu \in \cW$, one has $\rho_n \rightarrow \delta_\mu$ weakly if and only if for all closed sets $C\subset \cW$ one has $\limsupn \rho_n(C) \le \delta_\mu(C)$, which in turns holds if and only if  $\limn \rho_n (\cC^\delta(\mu)) \rightarrow 0$ for all $\delta>0$, where $\cC^\delta(\mu) = \{ \nu \in \cW : \fd(\nu,\mu) >\delta\}$.

Let $\delta>0$. Let $\cC^\delta = \{ \nu \in \cW : \fd(\nu, \mu^*) \ge \delta\}$, i.e., $\cC^\delta = B(\mu^*,\delta)^c$. We write $\bcL_n = \cL_{t_n}^{\ssup{\xi_n}}$. 
From the fact that 
\begin{align*}
\bcL_n(\cC^\delta) 
= \pmb{\P}_n[ \fd(\bW_n,\mu^*) > \delta ], 
\end{align*}
we deduce from \eqref{eqn:convergence_in_prob_entangled} that $\bcL_n(\cC^\delta)$ converges to zero in $\Prob$-probability, i.e., for all $\kappa>0$, $\Prob [	\bcL_n( \cC^\delta ) > \kappa ] \arrown 0$. 
From this we infer the existence of a strictly increasing $\varphi : \N \rightarrow \N$ such that $\bcL_{\varphi(n)}(\cC^\delta) \rightarrow 0$ almost surely. 

Therefore, by the above observation, it follows that 
$\bcL_{\varphi(n)} \rightarrow \delta_{\mu^*}$ almost surely, and thus $\cL_{t_{\varphi(n)}}^{\ssup{\xi}} \Longrightarrow \delta_{\mu^*}$.  
As for each sequence $(t_n)_{n\in\N}$ with $t_n \to \infty$ there exists a  strictly increasing $\varphi: \N \rightarrow \N$ such that $\cL^{\ssup{\xi}}_{t_{\varphi(n)}} \Longrightarrow  \delta_{\mu^*}$, 
it follows that $\cL^{\ssup{\xi}}_{t_n} \Longrightarrow  \delta_{\mu^*}$ for any sequence $(t_n)_{n\in\N}$ with $t_n \to \infty$, implying \eqref{eqn:law_converges_to_delta_mu_star}. 
\end{remark}


\section{Lower bound: proof of Proposition \ref{lowerboundrand}}
\label{sec:LowBound}

\noindent Our strategy follows the heuristics described in Section \ref{sec:heuristics}. 

Recall the setting introduced at the beginning of Section \ref{prooofthvarfor}, in particular Remark~\ref{skohorod-embedding} on the Skorohod embedding and the notations in \eqref{simpli}. 
Let $\mu \in \fF_1(\Pi)$, i.e., $\mu \in \cW$, $\mu \ll \Pi$ and $\mu$ be a probability measure (in case $\alpha \in (2d,\infty)$ one may take $\mu = \mu^*$ as in Lemma~\ref{l:maximizer_under_PPP}). By Lemma~\ref{l:Xi_as_sup_over_probm_fF_1} 
it is sufficient to show that,
with $\Prob$-probability $1$, 
\begin{align}
\label{eqn:sufficient_lower_bound_skorohod}
\liminf_{n\to \infty} \frac{1}{\bgamma_n} \log \bZ_n\geq \Psi_{\bPi}(\mu).
\end{align}
Our approach to do this is to choose a specific path event $\bcA_n$ and use the trivial estimate 
\begin{align}
\label{eqn:easy_bound_Z_t_theta}
\bZ_n \ge \E[ \ee^{\bH_n (X)} \1_{\bcA_n} ].
\end{align}  
We describe the event $\bcA_n$ in Section~\ref{Entroest}, but first give an idea here after introducing the following objects. 
Since $\mu$ is in $\fF_1(\bPi)$, 
there exist (``$\Prob$-''random) 
$k\in \N$ and $(f_1,y_1), \dots, (f_{k},y_{k})\in \supp(\bPi)$ and
$w_1, \dots, w_{k}\in (0,1]$ with $\sum_{i=1}^{k}w_i=1$ such that 
\begin{align*}
\mu=\sum_{i=1}^{k} w_i \, \delta_{(f_i,y_i)}. 
\end{align*}
We may assume that the order of the $(f_1,y_1), \dots, (f_{k},y_{k})$ is such that the minimal distance between the $y_1,\dots,y_{k}$ points is given by $\sum_{i=1}^{k} |y_i-y_{i-1}|$, where here and in the following we take $y_0=0$. 
Hence, 
\begin{align*}
\Xi(\bPi) = 
\Psi_{\bPi}(\mu)=\sum_{i=1}^k \Big(f_i w_i-\theta (w_i)^2\Big)-q \sum_{i=1}^{k} |y_i-y_{i-1}|.
\end{align*}

Because $\bPi_n \rightarrow \bPi$ in $\cM_\rp^\circ$ almost surely, there exists an $N \in \N$ such that for every
$n\geq N$ there exist distinct $(f_1^n, y_1^{n}), \dots, (f_{k}^n, y_{k}^{n})\in \supp \bPi_n$ such that almost surely 
\begin{equation}\label{defpontimp}
(f_i^n, y_i^n)  \xrightarrow{n\to \infty} (f_i,y_i), \qquad i\in \{1,\dots,k\}.
\end{equation}
Observe that by \eqref{eqn:Pin}, 
\begin{align}
\label{eqn:f_i_n_equal_to_frac_xi_etc}
f_i^n = \frac{\xi_n(y_i^n)}{\br_n^{d/\alpha}} \forqq{i\in \{1,\dots,k\}, n\ge N}. 
\end{align}
We will define the event $\bcA_n$, such that on this event the path visits the sites $\br_ny_1^n,\dots,\br_ny_{k}^n$ in this order, staying $\approx w_i t_n$ time units in $\br_n y_i^n$ for any $i\in\{1,\dots,k\}$.
We define $\bcA_n$ and estimate its probability from below in Section \ref{Entroest}. Then, we bound $\bH_{n}(X)$ from below on $\bcA_n$ 
in Section \ref{Energ}. Finally, in Section \ref{compfin} we combine these bounds and apply them
in the framework of the Skorokhod embedding defined above to finish the proof.

\subsection{The path event} \label{Entroest}

\noindent Let us introduce some useful notation involving paths that will be used to define the set $\bcA_n$.

\begin{definition}
\label{def:event_A}
 For $x\in\Z^d$ and $t\in [0,\infty)$ we define 
the \emph{entry time at $x$ after time $t$}, $\tau_x(t)$, and the \emph{exit time from $x$ after time $t$}, $\sigma_x(t)$, by 
\begin{align*}
\tau_x(t):=\inf\{s>t\colon X_s=x\}, \qquad \sigma_x(t):=\inf\{s>t \colon X_s\not=x\}. 
\end{align*} 
Let $t\in (0,\infty)$, 
\begin{align}
\label{eqn:parameters_event_A}
\begin{aligned}
\delta,s \in (0,1), 
\quad k\in \N, 
\quad y_0 :=0, 
\quad y_1,\dots,y_k \in \Z^d, 
\quad y = (y_1,\dots,y_k), \\
w_1,\dots,w_k \in [0,1] \mbox{ with } \sum_{i=1}^k w_i = 1-s, \quad w= (w_1,\dots,w_k). 
\end{aligned}
\end{align}
We define $\cA_{t,k}^{\delta,s}(y,w)$ to be the event where the random walk $X$ walks from $0$ to $y_1$ and then to $y_2$ etcetera. 
It takes at most $\frac{st}{k}$ time to reach $y_1$, then it spends at least $(1-\delta)tw_1$ and at most $tw_1$ time at $y_1$ before it jumps, then it spends at most $\frac{st}{k}$ time to reach $y_2$, waits at least
$(1-\delta)tw_2$ and at most $tw_2$ time at $y_2$ before it jumps, etc. 
More precisely, first we define inductively the entry $\tau^i_y$ and exit times $\sigma^i_y$ of the $y_i$, after the time that $y_{i-1}$ and thus all of $0,y_1,\dots,y_{i-1}$ are visited
\begin{align*}
\btau_y^0 :=0 , 
\qquad 
\btau_y^1 := \tau_{y_1}(0), 
\qquad
\btau_y^i := \tau_{y_i}(\btau_y^{i-1}) \forqq{ i \in \{1,\dots,k\}}, \\
\bsigma_y^i := \sigma_{y_i}(\btau_y^i) \forqq{ i \in \{0,1,\dots,k\}}, 
\end{align*}
so that (by definition $\btau_y^0 = 0  \le \btau_y^1 \le \bsigma_y^1 \le \cdots \btau_y^k \le \bsigma_y^k$ and) 
\begin{align*}
\cA_{t,k}^{\delta,s} (y, w) 
= \bigcap_{i=1}^k \Big\{\, \btau_y^i - \bsigma_y^{i-1} \1_{i\ge 2} \le \frac{st}{k} , \ \bsigma_y^i - \btau_y^i \in [1-\delta,1]tw_i \Big\}. 
\end{align*}
Observe that for $i=1$ we have $\btau_y^i = \btau_y^i - \bsigma_y^{i-1} \1_{i\ge 2} \le \frac{st}{k}$, so that  the waiting time at $0$ plus the ``walking time'' to $y_1$ is less or equal to $\frac{st}{k}$. 
Furthermore, observe that $y_k$ is reached before $t$, i.e., $\bsigma_y^k \le t$, because 
\begin{align*}
\bsigma_y^k 
& = (\bsigma_y^{k} - \btau_y^{k}) +
(\btau_y^{k} - \bsigma_y^{{k-1}}) 
+ \cdots 
+ (\bsigma_y^{1} - \btau_y^{1} )
+ (\btau_y^{1} - \bsigma_y^0)
+ (\bsigma_y^0 - \btau_y^0) \\
& \le tw_k + \frac{st}{k} + tw_{k-1} + \frac{st}{k} 
+ \cdots 
+ tw_1 + \frac{st}{k}  
= t \Big(s+ \sum_{i=1}^k w_i \Big) = t. 
\end{align*}
\end{definition}

\begin{lemma}
\label{l:lower_bound_prob_A}
For any $t\in(0,\infty)$, $\delta,s \in (0,1)$,  $k\in\N$ and $y$ and $w$ as in \eqref{eqn:parameters_event_A}
\begin{equation}\label{ProbA}
\P\Big(\cA_{t,k}^{\delta,s}(y,w) \Big)
\ge 
\prod_{i=1}^k \Big[{\rm Poi}_{\frac{2dst}{k}}\big(|y_i - y_{i-1}| \big)
 {\rm e}^{-2dt w_i (1-\delta)}[1-{\rm e}^{-2dt\delta w_i}]\Big(\frac{1}{2d}\Big)^{|y_i - y_{i-1} |}\Big].
\end{equation} 
\end{lemma}
\begin{proof}
By independence we have 
\begin{align*}
\P\Big(\cA_{t,k}^{\delta,s}(y,w \big) \Big)
= 
\prod_{i=1}^k \P \Big( 0 \le \btau_y^i - \bsigma_y^{i-1} \1_{i\ge 2} \le \frac{st}{k} \Big) \P \Big( \bsigma_y^i - \btau_y^i \in [1-\delta,1]tw_i \} \Big) . 
\end{align*}
By the strong Markov property and the fact that each jump occurs according to an $\Exp(2d)$ random variable (as we assumed our continuous time random walk to have generator $\Delta$),  we have 
\begin{align*}
\bsigma_y^i - \btau_y^i = \sigma_{y_i}(\btau_y^i) - \btau_y^i
= \sigma_{y_i}(\tau_{y_i}(\btau_y^{i-1})) - \tau_{y_i}(\btau_y^{i-1}) 
\overset{(\rm d)}{=} \sigma_{0}(0) - \tau_0(0), \\
\btau_y^i - \bsigma_y^{i-1} \1_{i\ge 2}
= \tau_{y_i}(\btau_y^{i-1}) - \sigma_{y_{i-1}}(\btau_y^{i-1})\1_{i\ge 2}
\overset{(\rm d)}{=} \tau_{y_i-y_{i-1}}(0) - \sigma_0(0) \1_{i\ge 2}, 
\end{align*}
and thus 
\begin{align*}
\P \Big( \bsigma_y^i - \btau_y^i \in [1-\delta,1]tw_i \} \Big)
& = 
\P \Big( \sigma_{0}(0) - \tau_{0}(0) \in [1-\delta,1]tw_i \} \Big)
= \ee^{-2dt w_i (1-\delta)} - \ee^{- 2d t w_i}, \\
 \P \Big( 0 \le \btau_y^i - \bsigma_y^{i-1}\1_{i\ge 2} \le \frac{st}{k} \Big)
&  =  \P \Big( 0 \le \tau_{y_i - y_{i-1}}(0) - \sigma_{0}(0)\1_{i\ge 2} \le \frac{st}{k} \Big)
 \ge \P \Big( \tau_{y_i - y_{i-1}}(0) \le \frac{st}{k} \Big). 
\end{align*}
To estimate the latter probability, we use the following estimate for $\rho \in (0,\infty)$ and $z\in \Z^d$, where $N(z)$ denotes the number of direct paths (i.e., of length $|z|$) from $0$ to $z$: 
\begin{align*}
\P \Big( \tau_{z}(0) \le \rho \Big)
& \ge \P \Big ( X \mbox{ makes } |z| \mbox{ jumps within } \rho \mbox{ time, from } 0 \mbox{ to } z \Big) \\
& = \Poi_{2d \rho}(|z|) (2d)^{-|z|} N(z) \ge \Poi_{2d \rho}(|z|) (2d)^{-|z|}. 
\end{align*}
\end{proof}

\subsection{Energetic lower bound}\label{Energ}

 \noindent  Now we derive a lower bound of $H_{t}^{\ssup{\xi}}$ on 
 $\cA_{t,k}^{\delta,s}(y,w)$.

\begin{lemma}[Lower bound for $H_{t}^{\ssup{\xi}}$]
\label{l:bound_H_t_theta}
Let $t\in(0,\infty)$, $\delta,s \in (0,1)$,  $k\in\N$ and $y$ and $w$ as in \eqref{eqn:parameters_event_A}.  
Then, on the event $\cA_{t,k}^{\delta,s}(y,w)$, 
\begin{align}
\frac{1}{r_t \log t}\, H_{t}^{\ssup{\xi}}(X)
\label{eqn:ham_bound_below}
&\geq  (1-\delta)\sum_{i=1}^k \frac{\xi(y_i)}{r_t^{d/\alpha}} \, w_i -\theta \sum_{i=1}^k w_i^2-(k+5) \theta(\delta+s).
\end{align}
\end{lemma}

\begin{proof} 
We have (see also \eqref{Hamil_Pi_t})
\begin{align}\label{eqn:ham_bound}
\frac{1}{r_t \log t}\, H_{t}^{\ssup{\xi}}(X)=
&\sum_{z\in \Z^d} \Big( \frac{\xi(z)}{r_t^{d/\alpha}} \frac{\ell_t(z)}{t} -\theta \big(\frac{\ell_t(z )}{t} \big)^2\Big).
\end{align}
Using that $\xi \ge 0$ and the basic estimate $\sum_{z\in \Z^d} a_z^2 \le \sum_{z\in A} a_z^2 + (\sum_{z\notin A} a_z)^2$, which can be used to show that the total normalized self-intersection local time (SILT) is not larger than the sum of the normalized SILTs in the $y_1,\dots,y_k$ plus the square of the remaining total local time, we obtain that
\begin{align}
\label{eqn:easy_H_t_theta_bound}
\frac{1}{r_t \log t}\, H_{t}^{\ssup{\xi}}(X) &\geq \sum_{i=1}^k\Big( \frac{\xi(y_i)}{r_t^{d/\alpha}}\, \frac{\ell_t(y_i)}{t}-\theta \big( \frac{\ell_t(y_i)}{t} \big) ^2\Big)-\theta\Big(1-\sum_{i=1}^k \frac{\ell_t(y_i)}{t}\Big)^2.
\end{align}
Observe that the local time at each $y_i$ is at least the time the random walk waits before jumping away, i.e., 
\begin{align}
\label{eqn:lower_bound_local_time_y_i}
\ell_t(y_i) \ge \bsigma_y^{i} - \btau_y^{i} \ge (1-\delta) t w_i. 
\end{align}
On $\cA_{t,k}^{\delta,s}(y,w)$, in between the times $\bsigma_y^{i-1}$ and $\btau_y^i$ the walker is allowed to visit sites $y_j$ for $j\ne i$. 
Moreover, after $\sigma_k t$ time, each of the $y_i$ may be revisited. 
We let $m_i \in [0,1]$ be such that $m_i t$ is the amount of time the path visits $y_i$ after $\sigma_k t$, for all $i \in \{1,\dots,k\}$. 
In particular, 
\begin{equation}\label{encadloctimes}
(1-\delta) w_i\leq \frac{\ell_t(y_i)}{t} \leq w_i + s + m_i, \qquad   i\in\{1,\dots,k\},
\end{equation}
and consequently, 
\begin{align}
\label{eqn:upper_bound_one_min_sum_local_times_sq}
& \Big(1- \sum_{i=1}^k \frac{\ell_t(y_i)}{t}\Big)^2 \le \Big(1-(1-\delta)\sum_{i=1}^k w_i\Big)^2 \leq  (s+\delta w)^2 \le  (s+\delta)^2.
\end{align}
Since $\sigma_{y_k}$ is both bounded from below by $ (1-\delta) t w = t(1-\delta)(1-s) \ge t - (s+\delta) t$, we infer
 $ \sum_{i=1}^k m_i  t\le t- \sigma_{y_k}  \le (s+\delta) t $ and therefore  deduce from the upper bound in \eqref{encadloctimes} that
\begin{align}
 \sum_{i=1}^k \frac{\ell_t(y_i)}{t}^2
\notag 
& \leq \sum_{i=1}^k (w_i +s)^2+2 \sum_{i=1}^k m_i (w_i +s)+\sum_{i=1}^k m_i^2 \\
& \le \sum_{i=1}^k w_i^2 + 2sw + ks^2 + 3\sum_{i=1}^k m_i 
\cnewline
\cand 
\begin{calc}
\notag 
\le \sum_{i=1}^k w_i^2 + 2s + ks^2 + 3(s+\delta) 
\end{calc} 
\label{eqn:upper_bound_one_min_sum_local_times_sq2}
 \leq  \sum_{i=1}^k w_i ^2 + (k+5)  (s+\delta).
\end{align}
Substituting these bounds \eqref{eqn:upper_bound_one_min_sum_local_times_sq} and \eqref{eqn:upper_bound_one_min_sum_local_times_sq2} in \eqref{eqn:easy_H_t_theta_bound} leads to \eqref{eqn:ham_bound_below}.
\end{proof}

\subsection{Conclusion}\label{compfin}

\noindent We now prove Proposition \ref{lowerboundrand}, by proving that  \eqref{eqn:sufficient_lower_bound_skorohod} holds $\bPi$-almost surely. 

\begin{proof}[Proof of Proposition~\ref{lowerboundrand}]
Recall the definition of the approximating sequence of vectors in \eqref{defpontimp}. 
We will assume that $n\geq N$ (where $N$ is as mentioned before \eqref{defpontimp}). 
We set (assuming $N$ is large enough)
\begin{align*}
\delta_n=
s_n = \frac{1}{\log t_n} , \quad
\bcA_n := \cA_{t_n,k}^{\delta_n,s_n} 
\Big( 
\, (\br_n y_1^n,\dots,\br_n y_k^n) \, , 
\, (w_1 - \frac{s_n}{k},\dots,w_k - \frac{s_n}{k}) 
\Big).
\end{align*}
By \eqref{eqn:easy_bound_Z_t_theta} we find a lower estimate for $\bZ_n$ by estimating $\bH_n(X)$ on $\bcA_n$ from below and by estimating $\P(\bcA_n)$ from below.
Recalling \eqref{simpli} and using Lemma \ref{l:bound_H_t_theta} we see that on $\bcA_n$, by using \eqref{eqn:f_i_n_equal_to_frac_xi_etc}, 
\begin{equation}\label{inegham2}
\begin{aligned}
\frac{1}{\bgamma_n}\, \bH_{n}(X)
\cand \begin{calc}
\geq  (1-\delta_n)\sum_{i=1}^k 
\frac{\xi_n(y_i^n)}{\br_n^{d/\alpha}} 
\, w_i^n -\theta \sum_{i=1}^k (w_{i}^n)^2-(k+5) \theta(\delta_n+s_n)
\end{calc} \cnewline
&\ge  (1-\delta_n)\sum_{i=1}^k f_i^n
\, w_i^n -\theta \sum_{i=1}^k (w_{i}^n)^2-(k+5) \theta(\delta_n+s_n)\\
& \xrightarrow{n\to \infty}
\sum_{i=1}^k( f_i \, w_{i} -\theta (w_{i})^2) .
\end{aligned}
\end{equation}
Due to the above limit, for \eqref{eqn:sufficient_lower_bound_skorohod} we are left to show
\begin{equation}\label{ProbAfin}
\liminf_{n\to\infty}\frac 1{\bgamma_n}\log \P\big(
\bcA_n \big)\geq -q D_0(y_1, \dots,y_k).
\end{equation}
With 
$y_0^n :=0$, put
\begin{align*}
\fd_i^n = \br_n |y_i^n - y_{i-1}^n| \forqq{i\in \{1,\dots,k\} }. 
\end{align*}
From  Lemma \ref{l:lower_bound_prob_A} we obtain 
 \begin{equation}\label{ProbA2}
\P\big( \bcA_n \big)
  \ge 
\prod_{i=1}^k \Big[{\rm Poi}_{\frac{2d s_n t_n}{k}}\big(\fd_i^n\big)
 {\rm e}^{-2dt_n w_i^n (1-\delta_n)}[1-{\rm e}^{-2dt_n\delta_n w_{i}^n}]\Big(\frac{1}{2d}\Big)^{\fd_i^n}\Big].
\end{equation} 
Since  for $n$ large enough $w_i^n = w_i - \frac{s_n }{k} = w_i - \frac{1}{k \log t_n }$ is bounded away from $0$ for all $i\in \{1,\dots,k\}$, since $\lim_{n\to \infty} t_n \delta_n=\infty$ and since  $\lim_{x\to \infty}\ee^x (1-\ee^{-x})=\infty$, we have that 
for $n$ large enough - by using that $j^j\geq j !$ for $j\in \N$ and  that $\sum_{i=1}^k (w_i^n + \frac{s_n}{k}) \leq 1$, 
 \begin{align*}
\nonumber \P\big(
\bcA_n
\big)
 & \geq   
\prod_{i=1}^k \Big[{\rm Poi}_{\frac{2d s_n t_n}{k}}\big(\fd_i^n \big)
 {\rm e}^{-2dt_n w_i^n }\Big(\frac{1}{2d}\Big)^{\fd_i^n}\Big]\\
\cand \begin{calc}
\nonumber 
= \prod_{i=1}^k 
\Big[
\frac{[ \frac{2d s_n t_n}{k} ]^{\fd_i^n } \ee^{ - \frac{2d s_n t_n}{k}} }{ \fd_i^n !} 
 {\rm e}^{-2dt_n w_i^n }\Big(\frac{1}{2d}\Big)^{\fd_i^n}\Big]
\end{calc}\cnewline
&\geq\prod_{i=1}^k \Big( 
 \bigg[\frac{\frac{2d s_n t_n}{k}}{\fd_i^n}\bigg]^{\fd_i^n}
 {\rm e}^{-2dt_n (w_i^n+\frac{s_n}{k}) }\Big(\frac{1}{2d}\Big)^{\fd_i^n}
 \Big)
 \geq \ee^{-2dt_n} \   \prod_{i=1}^k \bigg[\frac{s_n t_n}{k\fd_i^n}\bigg]^{\fd_i^n}. 
\end{align*}  
Clearly, as $\bgamma_n = \br_n \log t_n = t_n^{1+q} (\log t_n)^q$ (see \eqref{defr})
we have $\bgamma_n^{-1} \log (\ee^{-2dt_n}) \rightarrow 0$. 
Therefore, because 
$\frac{\fd_i^n}{\br_n} =|y_i^n - y_{i-1}^n|  \xrightarrow{n\rightarrow \infty} |y_i - y_{i-1}|$ for all $i\in \{1,\dots,k\}$ and $n\in \N$ and  because $\bgamma_n=\br_n\log t_n$ and $\log \frac{\br_n}{t_n s_n} \begin{calc}
= \log r_{t_n} - \log t_n - \log s_n = (1+q) \log t_n - (1+q) \log \log t_n - \log t_n + \log \log t_n 
\end{calc} = q \log t_n- q \log \log t_n$,
\begin{align*}
\liminfn \frac{1}{\bgamma_n}
\log \P(\bcA_n) 
&
 \ge 
 \liminfn \frac{1}{\bgamma_n} \log \bigg( \prod_{i=1}^k \bigg[\frac{s_n t_n}{k\fd_i^n}\bigg]^{\fd_i^n}
 \bigg) \\
 \cand \begin{calc}
\ge 
\liminfn \frac{1}{\bgamma_n}
 \Big( 
-\sum_{i=1}^k \Big[ \fd_i^n  \log \frac{ k\fd_i^n }{s_n} 
+ \fd_i^n \log t_n \Big]
 \Big)
\end{calc} \cnewline
&
 \ge 
- \limsupn \sum_{i=1}^k \frac{\fd_i^n}{\br_n} \frac{\br_n}{\bgamma_n}  \Big( \log \frac{\fd_i^n}{\br_n } +\log \frac{k \br_n}{t_n s_n} \Big) \\
\cand \begin{calc}
\geq - \sum_{i=1}^k |y_i - y_{i-1}|  \limsupn \frac{1}{\log t_n}  \Big(
  \log |y_i - y_{i-1}| +
  \log k  - q\log \log t_n + q \log t_n   \Big) 
\end{calc} \cnewline 
 &\geq - q \sum_{i=1}^k  |y_i - y_{i-1}| 
= -q D_0(y_1, \dots,y_k), 
\end{align*}
\end{proof}


\section{Upper bounds: proof of Proposition \ref{p:uppbound}} 
\label{sec:UppBound}

Part \ref{item:compactification} of Proposition \ref{p:uppbound} is a kind of \lq compactification\rq, which we will prove in Section  \ref{sec:compactification}. Part \ref{item:upper_bound_compactified_Z_n} is proved in Section~\ref{subsec:proof_prop_upp_bound_compactified} (using the Skorohod embedding of 
Remark \ref{skohorod-embedding}), and Part \ref{item:upper_bound_compactified_Z_n_away_max} in Section \ref{subsec:proof_prop_upp_bound_cpt_away}.

\subsection{Compactification}
\label{sec:compactification}

\noindent
In this section, we prove Proposition~\ref{p:uppbound}~\ref{item:compactification}. 
For this we actually do not need to consider a subsequence and the objects considered as in Remark~\ref{skohorod-embedding}. That is, we prove the following in this section, from which one directly derives Proposition~\ref{p:uppbound}~\ref{item:compactification}: 

\begin{proposition}
\label{p:compactification_general}
Let $\alpha \in (d,\infty)$ and $\theta \in (0,\infty)$. For any $A>0$, 
\begin{align}
\label{eqn:compactification_general}
\lim_{R\to \infty} \liminf_{t\to\infty}\  \Prob\bigg [\frac 1{r_t \log t }\log\E\Big[\ee^{H_t^{\ssup{\xi}}(X)}\1\Big\{ \max_{s\in[0,t]}|X_s|>R r_t \Big\}\Big]\leq -A\bigg]=1.
\end{align}
\end{proposition}

Let us first state three auxiliary lemmas. 
In the first one we estimate the $\P$-probability that the random walk $X$ takes too many jumps before time $t$ (Lemma \ref{boudfar})
and in the second one we estimate the $\Prob$-probability of the maximum of a modified version of the field $\xi$ outside a big box centered at the origin (Lemma \ref{field}). 
The third one (Lemma~\ref{l:order_statistics_n_pareto}) is a classical representation of the joint distribution of leading values in the $\xi$-field. 
The latter is used because it suffices to prove the estimate \eqref{eqn:compactification_general} but with  $H_t^{\ssup{\xi}}$ replaced by the maximum over the $\xi$ values in a box of radius $M_t$, where 
\begin{align}
\label{def:M_t}
M_t:= \max_{s\in[0,t]} |X_s| \forqq{t \in [0,\infty)}. 
\end{align}

\begin{lemma}\label{boudfar}
For every $R\ge 1$ and all sufficiently large $t$,
\begin{equation} \label{eq:boundfar}
 \P [ M_t \ge R r_t ] 
 \leq \exp\Big(-\frac q2 R r_t \log t\Big) .
 \end{equation}
\end{lemma}
\begin{proof}
The number of jumps taken by $X$ on the time interval $[0,t]$ is in distribution equal to a Poisson random variable $Z$ with parameter 
$2dt$. Therefore, $ \P [ M_t \ge R r_t ] \le P [Z\geq R r_t]$.  
By using Stirling's inequality $n! \ge (\frac{n}{e})^n$, 
the crude bound  $P[Z\geq n]\leq (2dt)^{n}/ n !\leq (2dt\e /n)^n$ for $n\in\N$ implies (recall that $r_t=(\frac t{\log t})^{q+1}$)
\begin{equation}\label{boundpoiss}
P[Z\geq R r_t] \leq \Big( \frac{4d\e t}{R r_t} \Big)^{R r_t}\leq \exp\Big(-\frac q2 R r_t \log t\Big). 
\end{equation}
for any large $t$ (we took an additional factor $2$ to cover up that $Rr_t$ might not be in $\N$).
\end{proof}

\begin{lemma}\label{field}
Let $A,c>0$ and $\epsilon \in (0,1)$. Then there exists an $R>0$ such that for all $r \ge 1$ 
\begin{align}
\label{eqn:prob_max_minus_rel_distance_small}
\Prob \left( \max_{x\in \Z^d \setminus Q_{Rr}}\Big( \frac{\xi(x)}{r^{d/\alpha}} - c\frac{|x|}{r}\Big) \le - A \right) \ge 1- \epsilon. 
\end{align}
\end{lemma}
\begin{proof}
Recall that $Q_R= [-R,R]^d$. We write 
\begin{align*}
 \fQ_R = Q_R \cap \Z^d.
\end{align*}
Pick a $C>0$ such that for all $r\ge 1$
\begin{align*} 
\#(\fQ_{(n+1)r}  \setminus \fQ_{nr})
\le  C n^{d-1} r^d,\qquad n\in\N. 
\end{align*}
Note that $1-x \ge \ee^{-2x}$ for $x\in [0,\frac{\log2}{2}]$. 
\begin{calc}
It will be clear that $1-x - \ee^{-2x} \ge 0$ for $x=0$. 
Furthermore, 
\begin{align*}
\frac{\rm d}{{\rm d} x} 1- x - \ee^{-2x} 
= 2\ee^{-2x} -1, 
\end{align*}
which is $\ge 0$ if $\ee^{-2x} \ge \frac12$, i.e., if $-2x \ge -\log 2$, i.e., $x\le \frac{\log 2}{2}$. 
\end{calc}
Pick $R\in \N$ large enough such that  $CR >A$ and $\frac{1}{(cR-A)^\alpha r^d} < \frac{\log 2}{2}$. Then 
\begin{align*}
\Prob \Big( &\max_{x\in \Z^d \setminus  \fQ_{Rr}}\Big( \frac{\xi(x)}{r^{d/\alpha}} - c \frac{|x|}{r}\Big) \le - A \Big)
= \prod_{n=R}^\infty 
\Prob \Big( \max_{x\in \fQ_{(n+1) r}  \setminus \fQ_{n r} } \Big(\frac{\xi(x)}{r^{d/\alpha}} - c \frac{|x|}{r}\Big) \le - A \Big) \\
& \ge  \prod_{n=R}^\infty 
\Prob \Big(  \frac{\xi(0)}{r^{d/\alpha}} - c n \le - A \Big)^{C n^{d-1} r^d} 
 = \prod_{n=R}^\infty 
 \Big( 1- \frac{1}{(cn-A)^{\alpha} r^d}  \Big)^{C n^{d-1} r^d}  \\
& \ge  \exp \Big(  - 2 C \sum_{n=R}^\infty  n^{d-1 -\alpha } \Big(\frac{n}{cn-A} \Big)^{\alpha} \Big).
\end{align*}
The exponential term on the right-hand side does not depend on $r$ and converges to $1$ as $R \rightarrow \infty$ as $d-1 -\alpha <- 1$ and as $\frac{n}{cn-A}$ is bounded from above.
\end{proof}

We will use the following classical representation of the distribution of the maximum over the $\xi$-values. 

\begin{lemma}
\label{l:order_statistics_n_pareto}
Let $\alpha \in (0,\infty)$, $n\in\N$ and $Z_1,\dots, Z_n$ be i.i.d.\ random variables that are Pareto distributed with parameter $\alpha$. Then the order statistics $Z_{1:n} \ge \cdots \ge Z_{n:n}$ of $Z_1, \dots, Z_n$ is given by 
\begin{align}
\label{eqn:max_pareto}
\Big( Z_{1:n} ,\dots, Z_{n:n} \Big) 
\overset{(\rm d)}{=}
\bigg( \Big(\frac{\Gamma_{n+1}}{\Gamma_1}\Big)^{1/\alpha},  \Big(\frac{\Gamma_{n+1}}{\Gamma_2}\Big)^{1/\alpha}, \dots, \Big(\frac{\Gamma_{n+1}}{\Gamma_n}\Big)^{1/\alpha}\bigg),
\end{align}
where $\Gamma_i = E_1 + \cdots + E_i$ and $(E_i)_{i\in\N}$ is a sequence of i.i.d.\ exponentially distributed random variables with parameter one.
\end{lemma}
\begin{proof}
It is a standard exercise, see for example \cite[Exercise 2.1.2]{DaVe03}, to show that the order statistics of  $n$ i.i.d.\ uniformly distributed random variables in distribution equals $\frac{E_1}{E_1+\cdots + E_{n+1}}, \dots, \frac{E_1+\cdots + E_{n}}{E_1+\cdots + E_{n+1}}$. By using that the Pareto distribution function 
\begin{calc}
$\Prob(Z \le r) =  r\mapsto 1- r^{-\alpha}$
\end{calc} is the composition of the uniform distribution function with $\Phi(s) := (1-s)^{-\frac{1}{\alpha}}$ one finds the order statistics of Pareto distributions and in particular \eqref{eqn:max_pareto}.
\begin{calc}
Indeed, the function $\Phi$ is increasing. Therefore if $U_1,\dots,U_n$ are i.i.d.\ uniformly distributed random variables and $U_{1:n}\ge \cdots \ge U_{n:n}$ their order statistics, then $\Phi(U_i) \overset{(\rm d)}{=} Z_i$ and thus the order statistics of the Pareto variables $Z_1,\dots, Z_n$, say $Z_{1:n} \ge \cdots \ge Z_{n:n}$, satisfies $(Z_{1:n},\dots, Z_{n:n}) = (\Phi(U_{1:n}), \dots, \Phi(U_{n:n}) )$, where 
\begin{align*}
Z_{i:n} \overset{(\rm d)}{=} \Phi(U_{i:n}) 
\overset{(\rm d)}{=} 
\Big( 1 - \frac{ E_1 + \cdots + E_i }{E_1+\cdots + E_{n+1}} \Big)^{-\frac{1}{\alpha}}
= \Big(\frac{ E_{i+1} + \cdots E_{n+1} }{E_1+\cdots + E_{n+1}} \Big)^{-\frac{1}{\alpha}}
\overset{(\rm d)}{=} 
\Big(\frac{E_1+\cdots + E_{n+1}}{ E_1 + \cdots E_n } \Big)^{\frac{1}{\alpha}}. 
\end{align*}
\end{calc}
\end{proof}

\begin{proof}[Proof of Proposition~\ref{p:compactification_general}]
We write $\xi^*(B) = \max_{x\in B} \xi(x)$ here. 
We have (recall \eqref{Hamilton})
\begin{equation}\label{boundH}
H_{t}^{\ssup{\xi}}\leq t  \xi^*( \fQ_{M_t}).
\end{equation}
It is then sufficient to prove \eqref{eqn:compactification_general} with  $t \xi^*( \fQ_{M_t})$ instead of $H_{t}^{\ssup{\xi}}$. Thus, we set 
\begin{align*}
C_{t,R}&:=\E\Big[\ee^{t  \xi^*( \fQ_{R r_t})}\1\big\{M_t>R r_t\big\}\Big] \qquad
\mbox{and} \qquad
D_{t,R}:=\E\Big[\ee^{t  \xi^*( \fQ_{M_t} \setminus \fQ_{R r_t})}\1\big\{M_t>R r_t\big\}\Big].
\end{align*}
Because $\xi^*(\fQ_{M_t}) = \xi^*(\fQ_{M_t} \setminus \fQ_{Rr_t}) \vee \xi^*(\fQ_{Rr_t})$, the proof of \eqref{eqn:compactification_general} is complete once we show that for every $A>0$  
\begin{equation}\label{eq:boundDaux1}
\lim_{R\to \infty} \liminf_{t\to\infty}\  \Prob\bigg [\frac 1{r_t \log t}\log C_{t,R}\leq -A\bigg]= 1,
\end{equation}
and that 
\begin{equation}\label{eq:boundDaux2}
\lim_{R\to \infty} \liminf_{t\to\infty}\  \Prob\bigg [\frac 1{r_t \log t}\log D_{t,R}\leq -A\bigg]= 1.
\end{equation}
\begin{calc}
Indeed, by writing 
\begin{align*}
E_{t,R}&:=\E\Big[\ee^{t  \xi^*( \fQ_{M_t})}\1\big\{M_t>R r_t\big\}\Big], 
\end{align*}
we have $E_{t,R} \le C_{t,R} + D_{t,R}$ and so 
\begin{align*}
\{C_{t,R} \le \exp ( -A r_t \log t) \} 
\cap \{D_{t,R} \le \exp ( -A r_t \log t) \} 
\subset \{ E_{t,R} \le 2 \exp ( -A r_t \log t) \}
. 
\end{align*}
Because $\fQ_{M_t}\setminus \fQ_{Rr_t}$ and $\fQ_{Rr_t}$ are disjoint, $ \xi^*(\fQ_{M_t} \setminus\fQ_{Rr_t})$ and $ \xi^*(\fQ_{Rr_t})$ are independent, and thus so are $C_{t,R}$ and $D_{t,R}$. Therefore 
\begin{align*}
\Prob[E_{t,R} \le 2 \exp ( -A r_t \log t)]
\ge \Prob[C_{t,R} \le \exp ( -A r_t \log t)]
\Prob[D_{t,R} \le \exp ( -A r_t \log t)]
\end{align*}
It remains to observe that 
\begin{align*}
\lim_{R\to \infty} \liminf_{t\to\infty}\ 
\Prob[E_{t,R} \le 2 \exp ( -A r_t \log t)]= 1 \forqq{A>0}, 
\end{align*}
if
\begin{align*}
\lim_{R\to \infty} \liminf_{t\to\infty}\ 
\Prob[E_{t,R} \le \exp ( - B r_t \log t)]= 1 \forqq{B>0}. 
\end{align*}
\end{calc}

\medskip

Let us begin with proving \eqref{eq:boundDaux1}. Pick $A>0$. Recall that $t r_t^{d/\alpha}=r_t \log t$. 
As $C_{t,R}= \ee^{t\xi^*(\fQ_{R r_t})} \P [ M_t > R r_t]$, from Lemma~\ref{boudfar} we obtain a $T>0$
such that for every $R\geq 1$ and $t\geq T$
\begin{align}
\frac{1}{r_t \log t} \log C_{t,R}
&\leq R \Big( \frac{1}{R^{1-d/\alpha}} \frac{\xi^*(\fQ_{R r_t} )}{(Rr_t)^{d/\alpha}}-\frac{q}{2}\Big).
\end{align}
Pick $R_0$ such that $\frac{q R_0}{4}> A$ and thus, 
\begin{calc}
for $t \ge T$ 
\begin{align*}
\Prob \Big[ R \Big( \frac{1}{R^{1-d/\alpha}} \frac{\xi^*(\fQ_{R r_t} )}{(Rr_t)^{d/\alpha}}-\frac{q}{2}\Big) \le - \frac{qR_0}{4} \Big]
\le 
\Prob\bigg [\frac 1{r_t \log t}\log C_{t,R}\leq -A\bigg]
\end{align*}
and so for $R \ge R_0$, 
\begin{align*}
\Prob \Big[ R \Big( \frac{1}{R^{1-d/\alpha}} \frac{\xi^*(\fQ_{R r_t} )}{(Rr_t)^{d/\alpha}}-\frac{q}{2}\Big) \le - \frac{qR_0}{4} \Big]
& \ge 
\Prob \Big[ R \Big( \frac{1}{R^{1-d/\alpha}} \frac{\xi^*(\fQ_{R r_t} )}{(Rr_t)^{d/\alpha}}-\frac{q}{2}\Big) \le - \frac{qR}{4} \Big] \\
& = 
\Prob \Big[  \frac{1}{R^{1-d/\alpha}} \frac{\xi^*(\fQ_{R r_t} )}{(Rr_t)^{d/\alpha}}  \le \frac{q}{4} \Big], 
\end{align*}
so that
\end{calc}
for every $t\ge T$ and $R\geq R_0$ one has 
\begin{equation}\label{refin1}
\Prob\bigg [\frac 1{r_t \log t}\log C_{t,R}\leq -A\bigg]
\geq  \Prob\bigg [\frac{\xi^*(\fQ_{R r_t} ) }{(Rr_t)^{d/\alpha}}
\leq \frac{q}{4} R^{1-d/\alpha} \bigg].
\end{equation}
We apply Lemma \ref{l:order_statistics_n_pareto} to $\xi^*(\fQ_{R r_t})$ to obtain
\begin{align*}
 \Prob\bigg [\frac{\xi^*(\fQ_{R r_t} ) }{(Rr_t)^{d/\alpha}}\leq \frac{q}{4} R^{1-d/\alpha} \bigg]
=
 \Prob\bigg [ \frac{ \Gamma_{\#(\fQ_{Rr_t})} }{(Rr_t)^{d}} \leq \big(\frac{q}{4}\big)^\alpha R^{\alpha-d} \Gamma_1 \bigg].
\end{align*}
By the weak law of large numbers $(Rr_t)^{-d} \Gamma_{\#(\fQ_{Rr_t})}  \Longrightarrow 2^{d}$ as $t\rightarrow \infty$, so that 
\begin{align*}
\liminf_{t \rightarrow \infty} 
 \Prob\bigg [ \frac{ \Gamma_{\#(\fQ_{Rr_t})} }{(Rr_t)^{d}} \leq \big(\frac{q}{4}\big)^\alpha R^{\alpha-d} \Gamma_1 \bigg]
=  \Prob\bigg [  \Gamma_1 \ge  2^d \big(\frac{q}{4}\big)^{-\alpha} R^{d- \alpha} \bigg]. 
\end{align*}
Then, by letting $R\rightarrow \infty$ we conclude  \eqref{eq:boundDaux1}. 

It remains to prove \eqref{eq:boundDaux2}. Since $D_{t,R}$ is decreasing in $R$,  it is sufficient to show that
for every $\epsilon \in (0,1)$, there exists a $R\in \N$ such that 
\begin{equation}\label{limDred}
 \liminf_{t\to\infty}\  \Prob\bigg [\frac 1{r_t \log t}\log D_{t,R}\leq -A\bigg]\geq  1-\epsilon.
 \end{equation}
We set 
\begin{equation}\label{defxx}
x^*_t:=\argmax\{\xi(x) \colon x\in \fQ_{M_t}\setminus \fQ_{R r_t}\}.
\end{equation}
Set 
\begin{equation}\label{defB}
\mathcal{B}_{ R,r}:=\Big\{\max_{x\in \Z^d \setminus \fQ_{ Rr}}\Big( \frac{\xi(x)}{r^{d/\alpha}} - \frac{q}{4}\frac{|x|}{r}\Big) \le - A\Big\}.
\end{equation}
We pick $\epsilon>0$ and we use Lemma \ref{field} with  $c=\frac{q}{4}$ to obtain the existence of an $ R >1$ such that for every  $r \ge 1$, 
\begin{equation}\label{contrfield}
\Prob(\mathcal{B}_{ R,r})\geq 1-\epsilon. 
\end{equation}
As $x_t^* \notin \fQ_{R r_t}$, on $\mathcal{B}_{ R,r_t}$ we derive the following estimates
\begin{align}\label{bounddt}
\nonumber 
	D_{t,R}&\leq \E\Big[\ee^{t  \xi(x^*_t)}\1\big\{M_t> R r_t\big\}\Big] \\
\nonumber 
	&\leq \E\bigg[\exp \Big( t r_t^{d/\alpha}  \frac{\xi(x^*_t)}{r_t^{d/\alpha}}-\frac{q}{4} (r_t \log t)  \frac{|x_t^*|}{r_t} \Big) \  \exp \Big( \frac{q}{4} |x_t^*| \log t \Big)\  \1\big\{M_t> R r_t\big\}\bigg] \\
	\nonumber  &\leq \exp \bigg( r_t \log t 
	\max_{x\in \Z^d \setminus \fQ_{ Rr_t}}\Big( \frac{\xi(x)}{r_t^{d/\alpha}} - \frac{q}{4}\frac{|x|}{r_t}\Big) 
	\bigg)\  \E\Big[\ee^{ \frac{q}{4} |x_t^*| \log t} \1\big\{M_t> R r_t\big\}\Big]\\
&\leq \ee^{-A r_t \log t}\ \   \E\Big[\ee^{ \frac{q}{4} M_t \log t} \, \1\big\{M_t> R r_t\big\}\Big]. 
\end{align}
Furthermore, by Lemma \ref{boudfar}, for large $t$ we have $ \P\big[M_t> j r_t\big] \le \exp (-\frac{q}{2} j  r_t \log t )$ for all $j\in\N$ and so, 
\begin{align*}
 \E\Big[\ee^{ \frac{q}{4} M_t \log t} \, \1\big\{M_t> R r_t\big\}\Big]
&\leq \sum_{j =  R}^\infty \E\Big[\ee^{ \frac{q}{4} M_t \log t} \, \1\big\{j r_t<M_t\leq (j+1) r_t\big\}\Big]\\
 &\leq  \sum_{j =  R}^\infty  \ee^{\frac{q}{4}   (j+1) r_t \log t} \,   \P\big[M_t> j r_t\big] \\
&\leq  \ee^{\frac{q}{4} r_t \log t } \sum_{j =  R}^\infty \exp \left(-\frac{q}{4} j  r_t \log t \right)
\\
& 
= 
\frac{ \exp \left(\frac{q}{4} (1-  R  ) r_t  \log t \right) }
{1 - \exp \left(-\frac{q}{4} r_t \log t  \right)}.
\end{align*}
As $R>1$, the latter converges to $0$ as $t\to \infty$. Therefore, by combining this with \eqref{bounddt}, we assert that for $t$ large enough, we have
\begin{equation}\label{contD}
\cB_{R,r_t} \subset \bigg\{ \frac{1}{r_t\log t} \log D_{t,R}\leq -A \bigg\}.
\end{equation}
\begin{calc}
Indeed, let 
\begin{align*}
s_t = 
\frac{ \exp \left(\frac{q}{4} (1-  R  ) r_t  \log t \right) }
{1 - \exp \left(-\frac{q}{4} r_t \log t  \right)}. 
\end{align*}
Suppose that $s_t <1$ for $t \ge T'$ for some $T' \ge T$. Then for $t \ge T'$, 
\begin{align*}
\cB_{R,r_t} \subset 
\bigg\{  D_{t,R}\leq \exp ( -A r_t\log t) s_t \bigg\} 
\subset
\bigg\{ D_{t,R}\leq \exp ( -A r_t\log t) s_t \bigg\}.
\end{align*}
\end{calc}
It remains to combine  \eqref{contrfield}  with \eqref{contD} to derive \eqref{limDred}. 
\end{proof}

\subsection{Proof of Proposition~\ref{p:uppbound}~\ref{item:upper_bound_compactified_Z_n}}
\label{subsec:proof_prop_upp_bound_compactified}

\noindent We adopt the setting introduced in Remark \ref{skohorod-embedding}; see also \eqref{simpli} for abbreviations. 
Before we start the proof and state a lemma that we will use for it, let us make the following observations. First observe that by \eqref{eqn:Hamiltonian_in_terms_of_Phi} $\bH_n(X) = \bgamma_n \Phi_{\bPi_n}( \bW_n)$. Let us write $\bW_n^\epsilon$ for the restriction of $\bW_n$ to $[\epsilon,\infty) \times \R^d$. Because for $w = \frac{\dd \bW_n}{\dd \bPi_n}$ one has  $\int_{(0,\epsilon) \times \R^d }  f w(f,y) - \theta w(f,y)^2 \dd \bPi_n(f,y) \le \epsilon$, we have 
\begin{align*}
\bH_n(X) \le \epsilon \bgamma_n + \Phi_{\bPi_n} (\bW_n^\epsilon). 
\end{align*}
Let $\bPi_n^\epsilon$ also be the restriction of $\bPi_n$ to $[\epsilon, \infty) \times \R^d$. 
As on the event $\{ \max_{s\in[0,t_n]}|X_s|\leq R \br_n \}$ the support of $\bW_n^\epsilon$ is a subset of $ E_n := \supp_{\R^d} \bPi_n^\epsilon \cap Q_R$, we have 
\begin{equation}
\begin{aligned}
\bZ_n^{R,-} = 
\E\Big[\ee^{\bH_n(X)} & \1\Big\{ \max_{s\in[0,t_n]}|X_s|\leq R \br_n \Big\}\Big]
 \le 
\sum_{A\subset E_n} \ee^{\epsilon \bgamma_n} 
\E\Big[\ee^{\bgamma_n \Phi_{\bPi_n}(\bW_n^\epsilon)} \1\Big\{ \supp_{\R^d} \bW_n^\epsilon =  A \Big\}\Big] \\
& \le 
\sum_{A\subset E_n} \ee^{\epsilon \bgamma_n} \exp \Big(\sup_{\substack{\mu \in \cW_R \\ \supp_{\R^d} \mu = A}}  \bgamma_n \Phi_{\bPi_n^\epsilon}(\mu) \Big)
\P\Big[ A =  \supp_{\R^d} \bW_n^\epsilon \Big]. 
\end{aligned}
\label{eqn:estimate_Z_n_R_min}
\end{equation}
In the proof we will provide a probabilistic argument that allows us to restrict the $A$ in the summand to those which do not contain elements around zero, i.e., of $A$ that are subsets of $\supp_{\R^d} \bPi_n^\epsilon \cap Q_R \setminus Q_\delta$. 
Then we will use that $\{A =  \supp_{\R^d} \bW_n^\epsilon\} \subset \{A \subset  \supp_{\R^d} \bW_n\}$ and the following lemma (for which we do not need the Skorohod setting, i.e., we do not need to restrict to a sequence of times). To motivate the condition of the lemma, observe that $A\subset E_n$ implies that $A\subset \supp_{\R^d} \bPi_n^\epsilon \subset \supp_{\R^d} \bPi_n \subset \br_n^{-1} \Z^d$. 

\begin{lemma}
\label{l:ineq_prob_supp_W_contains_A}
Let $R>\delta>0$. 
There exists a function $\gamma : (0,\infty) \rightarrow \R$ such that $\lim_{t\rightarrow \infty} \gamma(t) =0$, and such that for all $t\in (1,\infty)$ and all $A \subset Q_R \setminus Q_\delta$ with $r_t A \subset \Z^d$, 
\begin{align}
\label{eqn:inequality_prob_supp_W_contains_finite_set}
\P \Big[ 
A \subset \supp_{\R^d} W_{t} 
\Big]
\le \exp \Big( 
- q D_0(A) r_t \log t (1+\gamma(t)). 
\Big) .
\end{align}
\end{lemma}
\begin{proof}
Let $t\in (0,\infty)$ and $A$ be as mentioned.
Without loss of generality we may assume that $A$ is nonempty (because $D_0(\emptyset)=0$). 
Write $\widetilde A = r_t A$. 
Observe that $\supp_{\R^d} \Pi_t = r_t^{-1} \Z^d$, so that by definition of $W_t$, see \eqref{Wtdensity}, 
\begin{align*}
\Big\{ 
A \subset \supp_{\R^d} W_t  
\Big\}
= \Big\{
\ell_t(z) >0 \mbox{ for all } z\in \widetilde A 
\Big\}.
\end{align*}
The paths that realise the above event, i.e., that have a strict positive local time at all points of $A$, they make at least $n= D_0(\widetilde A)$ jumps. 
By using Stirling's inequality $n! \ge (\frac{n}{e})^n$  and that $\frac{n!}{(n+m)!}\le \frac{1}{m!}$, we obtain 
\begin{equation}
\label{eqn:upper_bound_prop_local_times}
\begin{aligned}
\P [ \ell_{t}(z)>0\ \mbox{ for all } z\in A ] 
&\leq \sum_{m=n}^\infty {\rm Poi}_{2dt}(m)=\sum_{m=n}^\infty\ee^{-2dt}\frac{(2dt)^{m}}{m!}\\
&=\frac{(2dt)^{n}}{n!}\sum_{m=0}^\infty\ee^{-2dt}\frac{(2dt)^{m}}{(m+n)!}n!
\leq \left(\frac{2dt e}{n} \right)^{n}. 
\end{aligned}
\end{equation}
Now we use that $n= D_0(\widetilde A) = r_t D_0(A)$, that $n > \delta r_t$ (because $A$ is nonempty and a subset of $Q_R \setminus Q_\delta$) and use that $r_t = t^{1+q} (\log t)^{-(1+q)}$ so that $\log t - \log r_t = - q \log t + (1+q) \log \log t$ to obtain 
\begin{align*}
\left(\frac{2dt \e}{n} \right)^{n} 
\le \left(\frac{2dt e}{\delta r_t} \right)^{r_t D_0(A)}
& = \exp \Big( D_0(A) r_t \Big( \log \frac{2d \e }{\delta}  + \log t - \log r_t  \Big) \\
& = \exp \Big( D_0(A) r_t \Big( \log \frac{2d \e }{\delta}  - q \log t + (1+q) \log \log t \Big) \Big),
\end{align*}
so that by setting $\gamma(t) = - (\log t)^{-1} (  \log \frac{2d \e }{\delta} + (1+q) \log \log t )$ we obtain the desired inequality. 
\end{proof}

\begin{proof}[Proof of Proposition~\ref{p:uppbound}~\ref{item:upper_bound_compactified_Z_n}]
Fix $\eps>0$ and $\eta>0$ and choose $\kappa>0$ so small that 
\begin{equation}\label{limpoint}
\Prob\big(\bPi([\eps,\infty)\times Q_\kappa)=0\big)\geq 1-\frac{\eta}{2}.
\end{equation} 
The $\Prob$-almost-sure convergence $\bPi_n \rightarrow \bPi$ in $\cM_\rp^\circ$ (see Lemma~\ref{l:P(t)toPi}) entails that, see for example Remark~\ref{remark:vague_convergence}, 
 \begin{equation}\label{limpointproc}
 \bPi_n([\eps,\infty)\times Q_\kappa) \  
 \underset{n\to \infty}{\overset{\Prob-a.s.}{\longrightarrow}} 
 \ \bPi([\eps,\infty)\times Q_\kappa).
\end{equation} 
Therefore there exists an $N\in \N$ such that $\Prob(\cB_N)\ge 1- \eta$, where 
\begin{align*}
\cB_N : = \cB_N^{\epsilon,\kappa} = 
\left\{ \bPi_n([\eps,\infty)\times Q_\kappa)=0 \mbox{ for all } n \ge N \right\}.
\end{align*}
Henceforth, we will work on the event $\cB_N$.
Let $R>0$. 
Observe that on $\cB_N$, for any $n\in\N$, one has 
$\supp_{\R^d} W_n^\epsilon \subset \cE_n$ for 
\begin{align*}
\cE_n = \cE_{n,\epsilon,R,\kappa} = (\supp_{\R^d} \bPi_n^\epsilon) \cap ( Q_R \setminus Q_\kappa). 
\end{align*} 
Therefore, by adapting the last inequality in  \eqref{eqn:estimate_Z_n_R_min} to restricting to subsets $A$ of $\cE_n$, we have on $\cB_N$, for all $n\ge N$
\begin{align}
\label{eqn:estimate_Z_n_R_min_better}
\bZ_n^{R,-}
\le 
\sum_{A \subset \cE_n } \ee^{\epsilon \bgamma_n} \exp \Big(\sup_{\substack{\mu \in \cW_R \\ \supp_{\R^d} \mu = A}}  \bgamma_n \Phi_{\bPi_n}(\mu) \Big)
\P\Big[ A \subset  \supp_{\R^d} \bW_n \Big]. 
\end{align}
Therefore, by Lemma~\ref{l:ineq_prob_supp_W_contains_A} and because $D_0(A) = \cD_{\bPi_n}(\mu)$ for any $\mu \in \cW$ with $\supp_{\R^d} \mu = A$, we have 
\begin{equation}
\begin{aligned}
\bZ_n^{R,-}
& \le 
\sum_{A\subset \cE_n} \ee^{\epsilon \bgamma_n} \exp \Big(
\sup_{\substack{\mu \in \cW_R \\ \supp_{\R^d} \mu = A}}  
\Big[
\bgamma_n \Phi_{\bPi_n^\epsilon}(\mu) 
- \bgamma_n (1+\gamma(t_n)) q \cD_{\bPi_n^\epsilon}(\mu) \Big] . 
\Big) \\
& \le 
 \ee^{\epsilon \bgamma_n} 2^{\# \cE_n} \exp \Big(
\sup_{ \mu \in \cW }  
\Big[
\bgamma_n \Phi_{\bPi_n^\epsilon}(\mu) 
- \bgamma_n (1+\gamma(t_n)) q \cD_{\bPi_n^\epsilon}(\mu) \Big] . 
\Big) 
\end{aligned}
\label{eqn:estimate_bZ_n_R_min_with_supremum}
\end{equation}
Since $\bPi_n \rightarrow \bPi$ in $\cM_\rp^\circ$ 
it follows by Remark~\ref{remark:vague_convergence} that for any $\epsilon>0$,  $\bPi_n^\epsilon \rightarrow \bPi^\epsilon$ in $\cM_\rp^\circ$, where $\bPi^\epsilon$ is the restriction of $\bPi$ to $[\epsilon,\infty) \times Q_R \setminus Q_\kappa$. Therefore, 
\begin{align*}
\#\cE_n =
\# \cE_{n,\epsilon,R,\kappa} 
=  \bPi_n^\epsilon( (0,\infty) \times Q_R \setminus Q_\kappa ) =  \bPi_n([\eps,\infty)\times Q_R\setminus Q_\kappa) 
\  \rightarrow \  
\bPi ([\eps,\infty)\times Q_R\setminus Q_\kappa) , 
\end{align*}
and thus 
\begin{align*}
\limn \frac{1}{\bgamma_n} (\log 2^{\# \cE_n}) =0, 
\end{align*}
and, 
since $\lim_{n\to \infty} \gamma(t_n)=0$, we have by Theorem~\ref{theorem:gamma_convergence}~\ref{item:gamma_conv_on_cpt} and  Proposition~\ref{p:gamma_convergence_Psi}~\ref{item:gamma_convergence_Psi}, 
\begin{align*}
\limsup_{n\to\infty}
\sup_{ \mu \in \cW }  
\Big[
\bgamma_n \Phi_{\bPi_n^\epsilon}(\mu) 
- \bgamma_n (1+\gamma(t_n)) q \cD_{\bPi_n^\epsilon}(\mu) \Big] 
\le \sup_{\mu\in \cW}[\Phi_{\bPi^\epsilon}(\mu)-q\, \mathcal D_{\bPi^\epsilon }(\mu)]
\le \Xi(\bPi^\epsilon). 
\end{align*}
Therefore, on $\cB_N$, for any $R>0$, 
\begin{align*}
\limsupn \frac{1}{\bgamma_n} \log \bZ_n^{R,-} 
\le \epsilon + \Xi(\bPi^\epsilon). 
\end{align*}
So summarizing the above, for every $\epsilon$ and $\eta$ in $(0,\infty)$ there exist a $\kappa>0$ and an $N\in\N$ such that $\Prob[\cB_N^{\epsilon,\kappa}] \ge 1-\eta$ and thus 
\begin{align*}
\Prob\Big[ \limsupn \frac{1}{\bgamma_n} \log \bZ_n^{R,-} 
\le \epsilon + \Xi(\bPi^\epsilon) \Big] 
\ge \Prob[\cB_N^{\epsilon,\kappa}] \ge 1-\eta. 
\end{align*}
As $\bPi^\epsilon \rightarrow \bPi$ in $\cM_\rp^\circ$ almost surely, we have  for any sequence $(\epsilon_k)_{k\in\N}$ with $\epsilon_k \downarrow \infty$;  $\limsup_{k \rightarrow 0} \Xi(\bPi^{\epsilon_k}) \le \Xi(\bPi)$  almost surely by Theorem~\ref{theorem:gamma_convergence}~\ref{item:gamma_conv_optimizers}\begin{calc}(strictly speaking,  as this theorem only considers sequences  we should replace the occurrences of ``$\epsilon$'' by ``$\epsilon_k$'' for a sequence $(\epsilon_k)_{k\in\N}$ that converges to $0$ and replace ``$\limsup_{\epsilon \downarrow 0}$'' by ``$\limsupk$'')\end{calc}. 
Therefore, for all $\eta>0$ and $\zeta>0$ there exists an $\epsilon>0$ such that 
\begin{align*}
\Prob[\epsilon + \Xi(\bPi^\epsilon) \le \Xi(\bPi) +\zeta] \ge 1-\eta, 
\end{align*}
and thus 
\begin{align*}
\Prob\Big[ \limsupn \frac{1}{\bgamma_n} \log \bZ_n^{R,-} 
\le \Xi(\bPi) + \zeta \Big] 
\ge 
\Prob\Big[ \limsupn \frac{1}{\bgamma_n} \log \bZ_n^{R,-} 
\le \epsilon +  \Xi(\bPi^\epsilon) \Big] 
\ge 1- \eta. 
\end{align*}
As the above holds for any $\eta>0$ and $\zeta>0$, (by first taking $\eta$ to zero and then $\zeta$) we obtain \eqref{eqn:limsup_Z_R_min}.
\end{proof}

\subsection{Proof of Proposition~\ref{p:uppbound}~\ref{item:upper_bound_compactified_Z_n_away_max}}
\label{subsec:proof_prop_upp_bound_cpt_away}

In this section we prove Proposition~\ref{p:uppbound}~\ref{item:upper_bound_compactified_Z_n_away_max} by mentioning where to adapt the proof in of Proposition~\ref{p:uppbound}~\ref{item:upper_bound_compactified_Z_n} as in the previous section. 

\begin{proof}[Proof of Proposition~\ref{p:uppbound}~\ref{item:upper_bound_compactified_Z_n_away_max}]
Let us write 
\begin{align*}
\cC^\delta  = \{ \nu \in \cW : \fd(\nu,\mu^*) \ge \delta\}, 
\qquad \fS^\delta 
= \bigcup \, \{ \supp_{\R^d} \nu : \nu \in \cC^\delta(\mu^*) \}. 
\end{align*}
Thus $\fS^\delta$ is the subset of $\R^d$ where the $\nu$ that are at least at distance $\delta$ of $\mu^*$, are allowed to be supported. Then, similarly to \eqref{eqn:estimate_Z_n_R_min} and \eqref{eqn:estimate_Z_n_R_min_better}, the following estimates hold, with $\cE_n^\delta = \cE_n \cap \fS^\delta$, on $\cB_N$, for $n\ge N$
\begin{align*}
\bZ_n^{R,-,\delta}
& = \E\Big[\ee^{\bH_n(X)}\, \1\{\fd(\bW_n,\mu^*) \ge \delta \}\, \1\Big\{ \max_{s\in[0,t_n]}|X_s|\leq R \br_n\Big\} \Big] \\
& \le 
\sum_{A \subset \cE_n^\delta } \ee^{\epsilon \bgamma_n} \exp \Big(\sup_{\substack{\mu \in \cW_R \cap \cC^\delta \\ \supp_{\R^d} \mu = A}}  \bgamma_n \Phi_{\bPi_n^\epsilon}(\mu) \Big)
\P\Big[ A \subset  \supp_{\R^d} \bW_n \Big].
\end{align*}
Then, similar to \eqref{eqn:estimate_bZ_n_R_min_with_supremum}, by using that $\cE_n^\delta \subset \cE_n$, we obtain (on $\cB_N$, for $n\ge N$)
\begin{equation*}
\begin{aligned}
\bZ_n^{R,-,\delta}
& \le 
\sum_{A\subset \cE_n^\delta} \ee^{\epsilon \bgamma_n} \exp \Big(
\sup_{\substack{\mu \in \cW_R \\ \supp_{\R^d} \mu = A}}  
\Big[
\bgamma_n \Phi_{\bPi_n^\epsilon}(\mu) 
- \bgamma_n (1+\gamma(t_n)) q \cD_{\bPi_n^\epsilon}(\mu) \Big] . 
\Big) \\
& \le 
 \ee^{\epsilon \bgamma_n} 2^{\# \cE_n} \exp \Big(
\sup_{ \mu \in \cC^\delta }  
\Big[
\bgamma_n \Phi_{\bPi_n^\epsilon}(\mu) 
- \bgamma_n (1+\gamma(t_n)) q \cD_{\bPi_n^\epsilon}(\mu) \Big] . 
\Big) 
\end{aligned}
\end{equation*}
The rest of the proof follows in the same fashion as in the proof of Proposition~\ref{p:uppbound}~\ref{item:upper_bound_compactified_Z_n} in the previous section, by taking $\Xi^\delta$ (see \eqref{eqn:def_Xi_delta}) instead of $\Xi$. 
\end{proof}


\appendix 

\section{The space \texorpdfstring{$\fE$}{} }
\label{section:space_fE}

\begin{lemma}
\label{l:fE}
Let $\mathfrak{E}$ be the union of $(0,\infty) \times \R^d$ with $(0,\infty]$, $\fE = ((0,\infty) \times \R^d) \cup (0,\infty]$. 
Define $\fd : \fE \times \fE \rightarrow [0,\infty)$ by 
\begin{align*}
\fd(s, s') 
&= \left|s - s' \right|  
& & s, s'\in (0,\infty], \\
\fd(s, (f,y)) 
&= \frac{1}{f} + \left| \frac{f}{1\vee |y|} - s \right|  
& & s \in (0,\infty], (f,y) \in (0,\infty) \times \R^d, \\
\fd((f,y), (f',y')) 
& = \frac{1}{f\wedge f'} \left( 1- \ee^{-|\log f- \log f'| -|y-y'|} \right) + \left| \frac{f}{1\vee |y|} - \frac{f'}{1 \vee |y'|}  \right| 
& & (f,y), (f',y') \in (0,\infty) \times \R^d. 
\end{align*}
\begin{enumerate}
\item 
\label{item:metric_on_fE}
$\fd$ is a metric on $\fE$.
\item 
\label{item:iota_cont_and_open}
The function $\iota : (0,\infty) \times \R^d \rightarrow \fE$, $\iota((f,y)) = (f,y)$, \quad $(f,y) \in (0,\infty) \times \R^d$, is continuous and open. 
\item 
\label{item:top_on_fE}
$\fE$ equipped with the topology generated by $\fd$ is a locally compact Polish space, such that \ref{item:compactness_in_fE_wrt_cH} and \ref{item:open_and_continuous_embed} of Lemma~\ref{l:about_cM_p} hold. 
Moreover, for $s,h>0$, the closure of $\cH_h^s$, is given by 
\begin{align}
\label{eqn:closure_cH}
\overline{\cH_h^s} 
= \{(f,y) \in (0,\infty) \times \R^d : f \ge s|y| + \height\} \cup \left[s,\infty \right],
\end{align}
and is a compact set. 
\item 
\label{item:every_compact_in_cH}
For every compact set $K$ in $\fE$ there exist $\height, \slope >0$ such that $K \cap [(0,\infty) \times \R^d]  \subset \cH_\height^\slope$ for some $\height,\slope>0$.
\end{enumerate}
\end{lemma}
\begin{proof}
\ref{item:metric_on_fE}
The idea behind this is very similar to \cite[Section~13]{BKS18} (which considers a larger space than $\R \times \R^d$ instead of $(0,\infty)\times \R^d$). 
\begin{calc}
Somehow the homeomorphism $(0,\infty) \times \R^d \rightarrow \R \times \R^d$, $(f,y) \mapsto (\log f, y)$ is put inbetween to relate these spaces, but the ``second part'' of the metric is quite similar to keep the limit as being the ``ratio''. 
\end{calc}
In order to see that $\fd$ is a metric, we have to show that the triangle inequality is satisfied. 
If $\fa,\fb,\fc\in \fE$, then $\fd(\fa,\fb) \le \fd(\fa,\fc) + \fd(\fc,\fb)$ follows easily if at least one element of $\fa,\fb,\fc$ is in $(0,\infty]$. 
Therefore, we show that $\fd$ satisfies the triangle inequality on $(0,\infty) \times \R^d$. 
It is rather easy to see that it suffices to prove that $\overline \fd$ is a metric on $\R \times \R^d$ (by plugging in $\lambda = \log f$ and $z=y$), where 
\begin{align*}
\overline \fd((\lambda,z), (\lambda',z')) & = \ee^{-(\lambda \wedge \lambda')} \left( 1- \ee^{-|\lambda - \lambda'| -|z-z'|} \right) & & (\lambda,z), (\lambda',z') \in \R \times \R^d. 
\end{align*}
As we will see, this can be boiled down to the fact that $(1-\ee^{-a}) ( 1- \ee^{-b}) \ge 0$ for $a,b\ge 0$. 
Let $(\lambda,z), (\lambda ', z'),$ $(\lambda '', z'') \in \R \times \R^d$. 
We may assume $\lambda < \lambda'$ and $\lambda ' = \lambda+ a, \lambda '' = \lambda +b$, $a>0$, $b\in \R$. Then with $p= |z-z'|$, $q=|z-z''|$, $r=|z'-z''|$, so that $p\le q+r$ and $\ee^{-|z-z'|}= \ee^{-p} \ge \ee^{-q-r}$, 
\begin{align*}
& \ee^{\lambda} \Big[ \overline \fd((\lambda, z), (\lambda'',z''))  + \overline \fd((\lambda'', z''), (\lambda',z'))  - 
\overline \fd ((\lambda, z), (\lambda',z')) \Big] \\
& = 
\ee^{-(b\wedge 0)} \left( 1- \ee^{-b -q} \right)
+ \ee^{- (a \wedge b)} \left( 1- \ee^{-|b-a| -r} \right)
- \left( 1- \ee^{-a -p} \right) \\
& \ge 
 \left( 1- \ee^{-b -q} \right)
+ \ee^{-a} \left( 1- \ee^{-|b-a| -r} \right)
- \left( 1- \ee^{-a -q-r} \right) \\
& =  - \ee^{-b -q} 
+ \ee^{-a} - \ee^{-|b-a|-a -r} 
+ \ee^{-a -q-r} \\
& =  \ee^{-a} (1  - \ee^{a-b -q} - \ee^{-|b-a|-r} + \ee^{ -q-r} ) 
  \\
&
\begin{cases}
\ge \ee^{-a} (1  - \ee^{ -q} - \ee^{-r} + \ee^{ -q-r} ) 
= \ee^{-a} (1-\ee^{-q})(1-\ee^{-r}) \ge0 & \mbox{if } b>a, \\
= \ee^{-a} (1  - \ee^{|b-a| -q})(1 - \ee^{-|b-a|-r}) \ge 0 & \mbox{if } b<a. 
\end{cases}
\end{align*}

\ref{item:iota_cont_and_open}
It is rather straightforward to check that a sequence $(f_n,y_n)_{n\in\N}$ in $(0,\infty) \times \R^d$ converges to an element $(f,y)$ of $(0,\infty) \times \R^d$ with respect to $\fd$ if and only if it converges with respect to the Euclidean metric on $(0,\infty) \times \R^d$. Therefore $\iota$ is continuous and open. 

\ref{item:top_on_fE}
That \eqref{eqn:closure_cH} holds follows by the definition of $\fd$. 
Observe that for a sequence $(a_n)_{n\in\N}$ in $\overline{\cH_h^s}$ there either exists a subsequence that is contained in $[s,\infty]$ or a subsequence in $(0,\infty)\times \R^d$ of the form $(f_n,y_n)_{n\in\N}$ for which either 
\begin{enumerate}[label={\normalfont(\arabic*)}] 
\item $f_n$ is contained in a set of the form $[\height, M]$ for some $M\ge \height$ (and thus the $y_n$ are contained in a ball of radius $\slope(M+\height)$), or, 
\item $f_n \rightarrow \infty$ (and thus $\liminfn \frac{f_n}{|y_n|} \ge \slope$). 
\end{enumerate}
In both cases one can find a subsequence of $(f_n,y_n)_{n\in\N}$ that converges in $\overline{\cH_h^s}$. Hence $\overline{\cH_h^s}$ is compact. 

Every $(f,y) \in (0,\infty) \times \R^d$ has a compact neighbourhood in $(0,\infty) \times \R^d$ and therefore in $\fE$, because $\iota$ is continuous.  
On the other hand, every $t\in (0,\infty]$ has a compact neighbourhood, for example $\overline{\cH_h^s}$ for $s<t$ and $h= \frac{1}{t-s}$ (indeed, observe that $\{a \in \fE :\fd(a,t) <t-s\}\subset \cH_h^s$). 
Therefore $\fE$ is locally compact. 

Observe that a sequence $(f_n,y_n)_{n\in\N}$ whose elements belong to $(0,\infty) \times \R^d$ is a Cauchy sequence in $\fE$ either if it is a Cauchy sequence in $[h,h^{-1}] \times \R^d$ for some $h\in (0,1)$ or if $f_n\rightarrow \infty$ and $\frac{f_n}{|y_n|} \rightarrow \slope$ for some $\slope \in (0,\infty]$, this $\slope$ is then the limit in $\fE$.  
From this we infer that $\fE$ is complete. 
It is separable as $\Q_{>0} \times \Q^d \cup \Q_{>0}$ is dense, where $\Q_{>0} = (0,\infty) \cap \Q$. 
Therefore $\fE$ is a Polish space. 

\ref{item:every_compact_in_cH}
This follows by the fact that every compact set in $\fE$ is a subset of $\overline{\cH_\height^\slope}$ for some $\height,\slope>0$. 
Indeed, first it will be clear that every compact set is a subset of $\{(f,y) \in (0,\infty) \times \R^d: f \ge h\} \cup (0,\infty]$ for some $h>0$. 
Secondly, $(0,s]$ is not compact for all $s>0$ and moreover, if $(f_n,y_n)_{n\in\N}$ is a sequence in $(0,\infty) \times \R^d$ with $f_n,y_n \rightarrow \infty$ and $\frac{f_n}{|y_n|} \rightarrow 0$, then it does not possess a subsequence that converges in $\fE$. 
\end{proof}

\begin{proof}[Proof of Lemma~\ref{l:about_cM_p}]
The existence of $\fE$ for which \ref{item:compactness_in_fE_wrt_cH} and \ref{item:open_and_continuous_embed} hold, follows by Lemma~\ref{l:fE}. 

\ref{item:cM_circ_as_subset}
Because $\iota$ is an open map, $\iota(B)$ is a Borel set in $\fE$ for every Borel set $B$ in $(0,\infty)\times \R^d$. Hence $\cP\circ \iota$ defines a measure on $\fE$, clearly with values in $\N_0 \cup \{\infty\}$. 
It is a Radon measure because $\iota$ is continuous and so $\iota(K)$ is compact in $\fE$ for every compact set $K$ in $(0,\infty)\times \R^d$. Therefore it is a Point measure on $(0,\infty)\times \R^d$. \begin{calc}
because for such $K$, one has $\cP(\iota(K)) <\infty$ because $\cP$ is Radon on $\fE$. 
\end{calc}
\begin{calc}
Observe that not every element of $\cM_\rp((0,\infty) \times \R^d)$ is of the form $\cP \circ \iota$ for some $\cP$ in $\cM_\rp^\circ$. Take $\sumn \delta_{(|y_n|+1,y_n)}$ for a sequence $(y_n)_{n\in\N}$ in $\R^d$ with $|y_n|\rightarrow \infty$ for example (it attains the value $\infty$ on the set $\cH_s^s$ for any $s\in (0,1)$). 
\end{calc}

\ref{item:extension_to_fE} 
Because the embedding is continuous, it follows that $\overline \cP$ is a measure on $\fE$. Hence it is an element of $\cM_\rp$ if and only if it is a Radon measure. 
Therefore by \ref{item:compactness_in_fE_wrt_cH} it follows that $\overline \cP$ is an element of $\cM_\rp$ if and only if $\cP(\cH_h^s) <\infty$ for all $s,h>0$. 
Suppose the latter is the case. 
Then $\supp \overline \cP \subset (0,\infty) \times \R^d$, as if otherwise, then there exists a sequence $(f_n,y_n)_{n\in\N}$ in this support that converges in $\fE$ to $2s$ for a $s\in (0,\infty]$. This can only be the case if that sequence is contained in $\cH_h^s$ for some $h>0$, in which case $\cP(\cH_h^s)=\infty$. 
\end{proof}

\section{Measurability of the maximizer \texorpdfstring{$\mu^*$}{}}
\label{sec:measurability_of_maximizer}

In this section we show the measurability of the $\mu^*$ as in Lemma~\ref{l:maximizer_under_PPP}. 

First observe that $\bPi$, $\bPi_n$ and $\bPi_n^{\ssup{L}}$ for all $n,L\in\N$ are all good point measures $\Prob$-almost surely by Lemma~\ref{l:at_most_one_finite_max} and Lemma~\ref{l:part_goodness}.
Let $(\Omega,\cF)$ be the underlying (complete) measurable space of $\Prob$. 
Let $\Omega_1 \in \cF$ be such that on $\Omega_1$, \  $\bPi$, $\bPi_n$ and $\bPi_n^{\ssup{L}}$ for all $n,L\in\N$ are all good point measures and such that $\bPi_n \rightarrow \bPi$ on $\Omega_1$. 
Then, for all $\omega \in \Omega_1$ and all $n,L\in\N$, there exist $\mu^*[\omega], \mu^*_n[\omega]$ and $\mu^*_{n,L}[\omega]$ such that 
\begin{align*}
\Psi_{\bPi[\omega]}(\mu^*[\omega]) = \Xi(\bPi[\omega]), \quad 
\Psi_{\bPi_n [\omega]}(\mu^*[\omega]) = \Xi(\bPi_n[\omega]), \quad 
\Psi_{\bPi_n^{\ssup{L}}[\omega]} (\mu^*[\omega]) = \Xi(\bPi_n^{\ssup{L}}[\omega]). 
\end{align*}
Let us set $\mu^*[\omega]= \mu^*_n[\omega] = \mu^*_{n,L}=0$ for all $n,L\in \N$. 
As, on $\Omega_1$, $\bPi_n^{\ssup{L}} \xrightarrow{L\rightarrow \infty} \bPi_n$ and $\bPi_n \xrightarrow{L \rightarrow \infty} \bPi$, by  Theorem~\ref{thm:variational_formula_deterministic}~\ref{item:var_f_continuity} we have $\mu^*_{n,L} \xrightarrow{L\rightarrow \infty} \mu^*_n$ and $\mu^*_n \xrightarrow{n \rightarrow \infty} \mu^*$ on $\Omega$. Therefore it suffices to show that $\mu^*_{n,L}$ is measurable for all $n,L\in\N$  in order to conclude that $\mu^*$ is measurable (likewise, $\mu^*_n$ for all $n\in\N$).

As $\bPi$ is almost surely good, there exists an $\Omega_1\subset \Omega$ with $\Prob(\Omega_1)=1$ such that for each $\omega \in \Omega_1$, by Theorem~\ref{thm:variational_formula_deterministic}~\ref{item:var_f_maximizer} there exists a unique $\mu^* [\omega] \in \fF(\bPi(\omega))$ such that 
\begin{align*}
\Psi_{\bPi[\omega]}(\mu^*[\omega]) = \Xi(\bPi[\omega]). 
\end{align*}
We will show that there exists an $\Omega^*\subset \Omega$ with $\Prob(\Omega^*)=1$ such that $\omega \mapsto \mu^*[\omega]$ is measurable. 

This follows because $\bPi_n^L$ has a finite support: 
For $\omega \in \Omega_1$, there exist $m[\omega]\in \N_0$ and distinct $(f_1[\omega],y_1[\omega]),$ $\dots,$ $(f_m[\omega],y_m[\omega])$ in $\supp \bPi_n^L[\omega]$ such that 
\begin{align*}
\Psi_{\bPi_n^L[\omega]}(\mu_{n:L}^*[\omega]) = 
\Xi(\bPi_n^L[\omega]) 
& = \varphi_m(f_1[\omega],\dots, f_m[\omega]) - q D_0 (y_1[\omega],\dots,y_m[\omega]) . 
\end{align*}
These $m$, $f_1,\dots,f_m$ and $y_1,\dots,y_m$ are measurable as this is a finite optimization problem. 
Then $\mu_{n:L}^* = \sum_{i=1}^m w_i \delta_{f_i,y_i}$, where the $w_i$ are measurable functions of the $f_1,\dots,f_m$ due to Proposition~\ref{p:AnaPhik}. 

\begin{calc}
\begin{nobs}
[Measurability of a finite optimization problem]
Let $\UU_1,\dots,\UU_m$ be random variables such that $\Prob[\UU_i = \UU_j] =0$ for all $i,j$ with $i\ne j$. 
Suppose $\Omega_1$ is such that $\UU_i(\omega) \ne \UU_j(\omega)$ for all $i,j$ with $i\ne j$ and all $\omega \in \Omega_1$. 
On $\Omega_1$ there exists a function $i$ such that 
\begin{align*}
i(\omega) = \argmax_{j=1}^m \UU_j(\omega) \forqq{\omega\in \Omega_1}. 
\end{align*}
This function is measurable as for all $k\in \{1,\dots,m\}$
\begin{align*}
\{\omega \in \Omega_1 : i(\omega) =k \} 
= \{ \omega \in \Omega_1 : \UU_k(\omega) - \max_{j=1}^m \UU_j(\omega) = 0\}, 
\end{align*}
and because $\UU_k$ and the maximum function $\max_{j=1}^m \UU_j$ are measurable. 
\end{nobs}
\end{calc}

\begin{calc}

\section{References to Resnick's book}

$\cM_{\rm p}(E)$ denotes the set of point measures on $E$, where $\mu$ is a point measure on $E$ if $\mu = \sum_{i=1}^\infty k_i \delta_{x_i}$ for $k_1,k_2,\dots $ in $\N$ and $x_1,x_2,\dots $ in $E$ such that for all compact $K\subset E$ there are only finitely many $i$ such that $x_i \in K$. 

\begin{theorem} \cite[Proposition 3.13]{Re87}
\label{theorem:resnick_on_compact}
Let $E$ be a complete separable metric space. 
Suppose $m_n, m \in \cM_{\rm p}(E)$ and $m_n \xrightarrow{v} m$. 
Let $K\subset E$ be compact and such that $m(\partial K) = 0$. 
Then there exists an $N$ such that for a $n\ge N$ there exist labelling of points in $m_n$ and $m$
\begin{align*}
m_n (\cdot \cap K) = \sum_{i=1}^I \delta_{x_i^n} , \qquad m (\cdot \cap K) = \sum_{i=1}^I \delta_{x_i}, 
\end{align*}
such that in $E^I$
\begin{align*}
(x_1^n, \dots, x_I^n) \rightarrow (x_1,\dots, x_I). 
\end{align*}
\end{theorem}

\end{calc}

\bibliography{references}{}
\bibliographystyle{alpha}

\end{document}